\documentclass{amsart}

\usepackage{amsthm, amssymb, amsmath, stmaryrd, mathrsfs}

\input xy
\xyoption{all}

\newcommand{\quash}[1]{}  %%Anything in \quash is ignored
\newcommand{\leftexp}[2]{{\vphantom{#2}}^{#1}{#2}}
\newcommand{\Hk}[2]{{\vphantom{\Hecke}}_{#1}\Hecke_{#2}}
\newcommand{\Hrs}[2]{{\vphantom{\calH}}_{#1}\calH^{\rs}_{#2}}
\newcommand{\const}[1]{\overline{\QQ}_{\ell,#1}}
\newcommand{\twtimes}[1]{\stackrel{#1}{\times}}

\newcommand{\homo}[2]{\mathbf{H}_{#1}({#2})}   % sheaf
\newcommand{\homog}[2]{\textup{H}_{#1}({#2})}  % plain group
\newcommand{\coho}[2]{\mathbf{H}^{#1}({#2})}    % sheaf
\newcommand{\cohog}[2]{\textup{H}^{#1}({#2})}     % plain group
\newcommand{\jiao}[1]{\langle{#1}\rangle}
\newcommand{\dquot}[2]{W_{#1}\backslash\tilW/W_{#2}}
\newcommand{\conv}[1]{\stackrel{#1}{\ast}}

\DeclareMathOperator{\SL}{SL}

\DeclareMathOperator{\SO}{SO}
\DeclareMathOperator{\Sp}{Sp}
\DeclareMathOperator{\Hom}{Hom}
\DeclareMathOperator{\Ext}{Ext}
\DeclareMathOperator{\End}{End}
\DeclareMathOperator{\Aut}{Aut}
\DeclareMathOperator{\Corr}{Corr}
\DeclareMathOperator{\Sym}{Sym}

\DeclareMathOperator{\Spec}{Spec}
\DeclareMathOperator{\Spf}{Spf}
\DeclareMathOperator{\Res}{Res}

\DeclareMathOperator{\id}{id}

\DeclareMathOperator{\ev}{ev}
\DeclareMathOperator{\Irr}{Irr}
\DeclareMathOperator{\Pic}{Pic}

\DeclareMathOperator{\codim}{codim}
\DeclareMathOperator{\coker}{coker}
\DeclareMathOperator{\Av}{Av}

\def\bG{\mathbf{G}}

\def\bI{\mathbf{I}}

\def\bP{\mathbf{P}}
\def\bQ{\mathbf{Q}}
\def\bR{\mathbf{R}}

\def\BB{\mathbb{B}}
\def\CC{\mathbb{C}}
\def\DD{\mathbb{D}}

\def\GG{\mathbb{G}}
\def\HH{\mathbb{H}}

\def\PP{\mathbb{P}}
\def\QQ{\mathbb{Q}}

\def\ZZ{\mathbb{Z}}

\def\frg{\mathfrak{g}}
\def\frt{\mathfrak{t}}

\def\frb{\mathfrak{b}}
\def\frl{\mathfrak{l}}
\def\frc{\mathfrak{c}}

\def\frA{\mathfrak{A}}
\def\frB{\mathfrak{B}}
\def\frF{\mathfrak{F}}

\def\calA{\mathcal{A}}
\def\calB{\mathcal{B}}
\def\calC{\mathcal{C}}
\def\calE{\mathcal{E}}
\def\calF{\mathcal{F}}
\def\calG{\mathcal{G}}
\def\calH{\mathcal{H}}
\def\calI{\mathcal{I}}
\def\calK{\mathcal{K}}
\def\calL{\mathcal{L}}
\def\calM{\mathcal{M}}
\def\calN{\mathcal{N}}
\def\calO{\mathcal{O}}
\def\calP{\mathcal{P}}
\def\calQ{\mathcal{Q}}
\def\calR{\mathcal{R}}

\def\tilT{\widetilde{T}}
\def\tilC{\widetilde{C}}

\def\tilg{\widetilde{\mathfrak{g}}}

\def\tilp{\widetilde{p}}
\def\tils{\widetilde{s}}
\def\tilw{\widetilde{w}}
\def\tilW{\widetilde{W}}

\def\tilv{\widetilde{v}}
\def\tilx{\widetilde{x}}

\def\tilI{\widetilde{\bI}}
\def\tilxi{\widetilde{\xi}}

\def\tilev{\widetilde{\ev}}
\def\tilphi{\widetilde{\varphi}}
\def\till{\widetilde{\mathfrak{l}}}
\def\tilL{\widetilde{L}}
\def\tilf{\widetilde{f}}

\def\tcP{\widetilde{\calP}}
\def\tcA{\widetilde{\calA}}

\def\hatG{\widehat{G}}

\def\hatO{\widehat{\calO}}
\def\hatg{\widehat{g}}

\def\hatom{\widehat{\omega}}

\def\upH{\textup{H}}

\def\unL{\underline{L}}
\def\unl{\underline{\frl}}
\def\unB{\underline{B}}
\def\unb{\underline{\frb}}
\def\unAut{\underline{\Aut}}

\def\undi{\underline{i}}

\def\Ln{L^{\natural}}

\def\xch{\mathbb{X}^*}
\def\xcoch{\mathbb{X}_*}

\def\one{\mathbf{1}}
\def\Ql{\overline{\QQ}_\ell}

\def\hotimes{\widehat{\otimes}}

\def\isom{\stackrel{\sim}{\to}}

\def\act{\textup{act}}
\def\rs{\textup{rs}}

\def\red{\textup{red}}

\def\proj{\textup{proj}}

\def\Ad{\textup{Ad}}
\def\ad{\textup{ad}}

\def\aff{\textup{aff}}
\def\nil{\textup{nil}}
\def\modo{\textup{ mod }}

\def\St{\textup{St}}
\def\st{\textup{st}}
\def\unSt{\underline{\St}}

\def\parab{\textup{par}}

\def\bch{\textup{b.c.}}
\def\adj{\textup{ad.}}

\def\XR{X_R}
\def\XS{S\times X}
\def\BM{\textup{BM}}

\def\Coor{\textup{Coor}}
\def\sing{\textup{sing}}
\def\punc{\textup{punc}}

\def\Bun{\textup{Bun}}
\def\Hit{\textup{Hit}}
\def\Grass{\mathcal{G}r}

\def\Flag{\mathcal{F}\ell}

\def\tilBun{\widetilde{\Bun}}
\def\Bunpar{\Bun^{\parab}}
\def\hatBun{\widehat{\Bun}}
\def\hatgloop{\widehat{G((t))}}

\def\stPic{\calP\textup{ic}}

\def\ani{\textup{ani}}
\def\Ah{\calA^{\heartsuit}}
\def\Aa{\calA^{\ani}}
\def\AHit{\calA^{\Hit}}
\def\tcArs{\tcA^{\rs}}

\def\Mpar{\mathcal{M}^{\parab}}
\def\Mparrs{\mathcal{M}^{\parab,\rs}}

\def\MHit{\mathcal{M}^{\Hit}}

\def\Hecke{\mathcal{H}\textup{ecke}}
\def\Heckep{\Hecke^{\parab}}

\def\fpar{f^{\parab}}

\def\fQl{f^{\parab}_*\Ql}
\def\fHQl{f^{\Hit}_*\Ql}

\def\tfQl{\widetilde{f}_*\Ql}

\def\bxi{\overline{\xi}}
\def\PQ{^{\bQ}_{\bP}}
\def\Wa{W_{\textup{aff}}}

\def\disk{\mathfrak{D}}
\def\pdisk{\disk^{\times}}

\def\cent{\Ql[\xcoch(T)]^W}
\def\grot{\GG^{\textup{rot}}_m}
\def\gcen{\GG^{\textup{cen}}_m}

\def\can{\textup{can}}
\def\canb{\omega_{\Bun}}
\def\Lcan{\calL_{\can}}

\def\forg{\textup{For}}
\def\tforg{\widetilde{\forg}}

\swapnumbers 
\theoremstyle{plain}
\newtheorem{theorem}[subsubsection]{Theorem}
\newtheorem{lemma}[subsubsection]{Lemma}
\newtheorem{cor}[subsubsection]{Corollary}
\newtheorem{prop}[subsubsection]{Proposition}

\newtheorem{ques}[subsubsection]{Question}

\newtheorem*{tha}{Theorem A}
\newtheorem*{thb}{Theorem B}
\newtheorem*{thc}{Theorem C}

\theoremstyle{definition}
\newtheorem{defn}[subsubsection]{Definition}
\newtheorem{cons}[subsubsection]{Construction}
\newtheorem{remark}[subsubsection]{Remark}
\newtheorem{exam}[subsubsection]{Example}

\numberwithin{equation}{section}

\title[Towards a Global Springer Theory II]{Towards a Global Springer Theory II:\\ the double affine action}
\author{Zhiwei Yun}
\address{Department of Mathematics, Princeton University, Princeton, NJ 08544, USA}
\email{zyun@math.princeton.edu}
\date{February 2009; Revised April 2009}

\begin{document}

\begin{abstract}
We construct an action of the graded double affine Hecke algebra (DAHA) on the parabolic Hitchin complex, extending the affine Weyl group action constructed in \cite{GSI}. In particular, we get representations of the degenerate DAHA on the cohomology of parabolic Hitchin fibers. We also generalize our construction to {\em parahoric} versions of Hitchin stacks, including the construction of 'tHooft operators as a special case. We then study the interaction of the DAHA action and the cap product action given by the Picard stack acting on the parabolic Hitchin stack.
\end{abstract}

\maketitle

\tableofcontents

\section{Introduction}

This paper is a continuation of \cite{GSI}. For an overview of the ideas and motivations of this series of papers, see the Introduction of \cite{GSI}. We will use the notations and conventions from \cite[Sec. 2]{GSI}. In particular, we fix a connected reductive group $G$ over an algebraically closed field $k$ with a Borel subgroup $B$, a connected smooth projective curve $X$ over $k$ and a divisor $D$ on $X$ of degree at least twice the genus of $X$. Recall from \cite[Def. 3.1.2]{GSI} that we defined the parabolic Hitchin moduli stack $\Mpar=\Mpar_{G,X,D}$ as the moduli stack of quadruples $(x,\calE,\varphi,\calE^B_x)$ where
\begin{itemize}
\item $x\in X$;
\item $\calE$ is a $G$-torsor on $X$ with a $B$-reduction $\calE^B_x$ at $x$;
\item $\varphi\in\cohog{0}{X,\Ad(\calE)(D)}$ is a Higgs field compatible with $\calE^B_x$.
\end{itemize}
We also defined the parabolic Hitchin fibration (see \cite[Def. 3.1.6]{GSI}):
\begin{equation*}
\fpar:\Mpar\to\calA\times X.
\end{equation*}

In \cite{GSI}, we have constructed an action of the extended affine Weyl group $\tilW$ on the parabolic Hitchin complex $\fQl$, which justifies to be called the ``global Springer action''. As me mentioned in \cite[Sec. 1.2]{GSI}, there are at least three pieces of symmetry acting on the complex $\fQl$: the affine Weyl group action, the cup product action given by certain Chern classes and the cap product action given by the Picard stack $\calP$. This paper is devoted to the study of the second and the third action on $\fQl$, as well as the interplay among the three actions.

Recall from \cite[Rem. 3.5.6]{GSI} that we have chosen an open subset $\calA$ of the anisotropic Hitchin base $\Aa$ on which the codimension estimate $\codim_{\AHit}(\calA_\delta)\geq\delta$ holds for any $\delta\in\ZZ_{\geq0}$. Throughout this paper, with the only exception of Sec. \ref{s:parahoric}, we will work over this open subset $\calA$. All stacks originally over $\AHit$ or $\Aa$ will be restricted to $\calA$ without changing notations. Note that when $\textup{char}(k)=0$, we may take $\calA=\Aa$.

\subsection{Main results}

\subsubsection{The double affine action}\label{sss:double}
The $\tilW$-action on $\fQl$ constructed in \cite[Th. 4.4.3]{GSI} and the Chern class action mentioned above together give a full symmetry of the graded {\em double affine Hecke algebra $\HH$} (DAHA) on $\fQl$, which we now define.

For simplicity, let us assume $G$ is almost simple and simply-connected, so that the affine Weyl group $\tilW=\xcoch(T)\rtimes W$ is a Coxeter group with simple reflections $\{s_0,s_1,\cdots,s_n\}$. The graded algebra $\HH$ is, as a vector space, the tensor product of the group ring $\Ql[\tilW]$ with the polynomial algebra $\Sym_{\Ql}(\xch(\tilT)_{\Ql})\otimes\Ql[u]$. Here $\tilT$ is the Cartan torus in the Kac-Moody group associated to the loop group $G((t))$ (see Sec. \ref{ss:KM}). The graded algebra structure of $\HH$ is uniquely determined by
\begin{itemize}
\item $\Ql[\tilW]$ is a subalgebra of $\HH$ in degree 0;
\item $\Sym_{\Ql}(\xch(\tilT)_{\Ql})$ is a subalgebra of $\HH$ with $\xi\in\xch(\tilT)$ in degree 2;
\item $u$ has degree 2, and is central in $\HH$;
\item For each simple reflection $s_i$ and $\xi\in\xch(\tilT)$,
\begin{equation}\label{eq:introDAHA}
s_i\xi-\leftexp{s_i}{\xi}s_i=\langle\xi,\alpha_i^\vee\rangle u
\end{equation}
Here $\alpha_i^\vee\in\xcoch(\tilT)$ is the coroot corresponding to $s_i$. 
\end{itemize}

We have a decomposition $\tilT=\gcen\times T\times\grot$, where $\gcen$ is the one dimensional central torus in the Kac-Moody group and $\grot$ is the one dimensional ``loop rotation'' torus. Let $\delta\in\xch(\grot)$ and $\Lambda_0\in\xch(\gcen)$ be the generators (here we are using Kac's notation for affine Kac-Moody groups, see \cite[6.5]{Kac}).

The moduli meaning of $\Bunpar_G$ gives a universal $B$-torsor on $\Bunpar_G$, and induces a $T$-torsor $\calL^T$ on $\Bunpar_G$. In particular, for any $\xi\in\xch(T)$, we have a line bundle $\calL(\xi)$ on $\Bunpar_G$ induced from $\calL^T$ and the character $\xi$. We can view these line bundles as line bundles on $\Mpar$ via the morphism $\Mpar\to\Bunpar_G$. We also have the {\em determinant line bundle} on $\Mpar$, which is (up to a power) the pull-back of the canonical bundle $\canb$ of $\Bun_G$.

\begin{tha}[See Th. \ref{th:daction}]
There is a graded algebra homomorphism
\begin{equation*}
\HH\to\bigoplus_{i\in\ZZ}\End^{2i}_{\calA\times X}(\fQl)(i)
\end{equation*}
extending the $\tilW$-action in \cite[Th. 4.4.3]{GSI}. The elements $\xi\in\xch(T),\delta,2h^\vee\Lambda_0$ ($h^\vee$ is the dual Coxeter number of $G$) and $u$ in $\HH$ act as cup products with the Chern classes of $\calL(\xi),\omega_X,\canb$ and $\calO_X(D)$ respectively, where $D$ is divisor on $X$ that we used to define $\Mpar$.
\end{tha}

In particular, for any point $(a,x)\in(\calA\times X)(k)$, we get an action of $\HH$ on the cohomology $\cohog{*}{\Mpar_{a,x}}$. It is easy to see that $u$ and $\delta$ acts trivially on $\cohog{*}{\Mpar_{a,x}}$. This gives geometric realizations of representations of the graded DAHA specialized at $u=\delta=0$.

The above theorem is inspired by the results of Lusztig (\cite{L88}) in the classical situation, where he constructed an action of the graded affine Hecke algebra on the Springer sheaf $\pi_*\Ql$, where $\pi:\tilg\to\frg$ is the Grothendieck simultaneous resolution. 

\subsubsection{Generalizations to the parahoric versions}
In classical Springer theory, many constructions for the Grothendieck simultaneous resolution can be generalized to partial Grothendieck resolutions using general parabolic subgroups, see \cite{BM}. In the global situation, the parahoric Hitchin moduli stacks (see Def. \ref{def:hitp}) play the role of partial resolutions. The main results of \cite{GSI} and Th. A above can all be generalized to parahoric Hitchin stacks of arbitrary type $\bP$. For example, Th. A generalizes to:

\begin{thb}[see Th. \ref{th:pardaction}] Fix a standard parahoric subgroup $\bP\subset G(F)$ and let $W_{\bP}$ be the Weyl group of the Levi factor of $\bP$. Let $\HH_{\bP}$ be the subalgebra of $\HH$ generated by $\Ql[\dquot{\bP}{\bP}]\subset\Ql[\tilW]$, $\Sym_{\Ql}(\xch(\tilT)_{\Ql})^{W_{\bP}}\subset\Sym_{\Ql}(\xcoch(\tilT)_{\Ql})$ and $\Ql[u]$. Then there is a natural graded algebra homomorphism:
\begin{equation*}
\HH_{\bP}\to\bigoplus_{i\in\ZZ}\End^{2i}_{\calA\times X}(f_{\bP,*}\Ql)(i).
\end{equation*}
\end{thb}

\subsubsection{Relation with the cap product action}

As we saw in \cite[Sec. 3.2]{GSI}, $\Mpar$ has another piece of symmetry, namely the fiberwise action of a Picard stack $\calP$ over $\calA$. This action induces an action of the homology complex $\homo{*}{\calP/\calA}$ on $\fQl$, called the {\em cap product} action, see Sec. \ref{ss:cap}. Since the homology complex $\homo{*}{\calP/\calA}$ is the exterior algebra of $\homo{1}{\calP/\calA}$ over $\homo{0}{\calP/\calA}$, the cap product action is determined by its restriction to $\homo{1}{\calP/\calA}$ and $\homo{0}{\calP/\calA}$. Also note that $\homo{0}{\calP/\calA}$ is isomorphic to $\Ql[\pi_0(\calP/\calA)]$, the group algebra of the sheaf of fiberwise connected components of $\calP\to\calA$.

The next result is about the interplay between the DAHA action constructed in Th. B and the cap product action by the homology of $\calP$.

\begin{thc}[see Prop. \ref{p:capW}, Th. \ref{th:comp}, Prop. \ref{p:hxi} and Cor. \ref{c:stalkcomm} respectively]
\begin{enumerate}
\item []
\item The action of $\homo{*}{\calP/\calA}$ on $\fQl$ commutes with the action of $\tilW$, $u$ and $\delta$.
\item The action of $\cent$ on $\bR^m\fQl$ given by restricting the $\tilW$-action factors through the action of $\Ql[\pi_0(\calP/\calA)]$ on $\bR^m\fQl$ via a natural homomorphism $\cent\to\Ql[\pi_0(\calP/\calA)]$.
\item For a local section $h$ of $\homo{1}{\calP/\calA}$, and $\xi\in\xch(T)$, their actions on $\fQl$ satisfy the commutation relation:
\begin{equation*}
[\xi,h]=c_\xi(h_{\st})
\end{equation*}
where $h_{\st}$ is the {\em stable part} of $h$ (see Def. \ref{def:stP}), $c_\xi:\homo{1}{\calP/\calA}_{\st}\to\coho{*}{\tcA/\calA}(1)$ is a linear map defined in (\ref{eq:dualc1Q}), and $c_\xi(h_{\st})$ (a local section of $\coho{1}{\tcA/\calA}(1)$) acts on $\fQl$ by cup product. 
\item For any point $(a,x)\in(\calA\times X)(k)$, the cap product action of $\homog{*}{\calP_a}$ on $\cohog{*}{\Mpar_{a,x}}$ commutes with the action of the subalgebra $\Ql[\tilW]\otimes\Sym(\xch(T)_{\Ql})$ of the degenerate graded DAHA $\HH/(\delta,u)$.
\end{enumerate}
\end{thc}

\subsection{Organization of the paper and remarks on the proofs}

In Sec. \ref{s:parahoric}, we define parahoric versions of Hitchin moduli stacks. Many properties of $\calM_{\bP}$ parallel those of $\Mpar$, and we only mention them without giving proofs. We also give examples of parahoric Hitchin moduli stacks in Section \ref{ss:geomex}. This section is used in the proof of Th. A.

In Sec. \ref{s:DAHA}, we construct the graded DAHA action on the parabolic Hitchin complex (i.e., we prove Th. A). For this, we need some knowledge on the line bundles on $\Bunpar_G$ and the connected components of $\Mpar$, which we review in Sec. \ref{ss:linebd} and Sec. \ref{ss:conncomp}. The proof of the relation (\ref{eq:introDAHA}) in the DAHA is essentially a calculation of the equivariant cohomology of the Steinberg variety for $\SL(2)$, which we carry out in Sec. \ref{ss:simplerefl}. 

In Sec. \ref{s:genaction}, we generalize the main results in \cite{GSI} and the graded DAHA action to parahoric Hitchin stacks. In particular, we get an action of $\cent$ on the usual Hitchin complex $\fHQl\boxtimes\const{X}$. This can be viewed as 'tHooft operators in the context of constructible sheaves.

In Sec. \ref{s:cap}, we study the relation between the cap product action of $\homo{*}{\calP/\calA}$ on $\fQl$ and the graded DAHA action constructed in the Th. A. The proof of Th. C(2) uses the idea of deforming the product of the affine Grassmannian $\Grass_G$ and the usual flag variety $\calB$ into the affine flag variety $\Flag_G$, which first appeared in Gaitsgory's work \cite{Ga}.

In App. \ref{s:capapp}, we review the notion of the Pontryagin product and the cap product, which is used in Sec. \ref{s:cap}. 

In App. \ref{s:complcorr}, we prove lemmas concerning the relation between cohomological correspondences and cup/cap products.

\subsection*{Acknowledgment}
I would like to thank G.Lusztig for drawing my attention to the paper \cite{L88}, which is crucial to this paper. I would also like to thank V.Ginzburg and R.Kottwitz for helpful discussions.

% parahoric

\section{Parahoric versions of the Hitchin moduli stack}\label{s:parahoric}

In this section, we generalize the notion of Hitchin stacks to arbitrary parahoric level structures, not just the Iwahori level structure considered in \cite[Sec. 3]{GSI}. Throughout this section, let $F=k((t))$ be the field of formal Laurent series over $k$, and let $\calO_F=k[[t]]$ be its valuation ring. The reductive group $G$ over $k$ determines a split group scheme $\bG=G\otimes_{\Spec k}\Spec\calO_F$ over $\calO_F$.

Technically speaking, this section is only used in the proof of Th. \ref{th:daction}, but many results about $\Mpar$ also have their counterparts for parahoric Hitchin stacks, as we will see in Sec. \ref{s:genaction}.

\subsection{Local coordinates}\label{ss:pre}
Parahoric subgroups are local notions. In order to make sense of them over a global curve, we have to deal with parahoric subgroups in a ``Virasoro-equivariant'' way. For this, we need to consider local coordinates on the curve $X$.

\subsubsection{The group of coordinate changes} We follow \cite[2.1.2]{Ga} in the following discussion. Let $\Aut_{\calO}$ be the pro-algebraic group of automorphisms of the topological ring $\calO_F$. More precisely, for any $k$-algebra $R$, $\Aut_{\calO}(R)$ is the set of $R$-linear continuous automorphisms of the topological ring $R[[t]]=R\hotimes_k\calO_F$ (with $t$-adic topology). If we define $\Aut_{\calO,n}$ to be the algebraic group of automorphisms of $\calO_F/t^{n+1}\calO_F$, then $\Aut_{\calO}$ is the projective limit of $\Aut_{\calO,n}$.

\subsubsection{The space of local coordinates} We have a canonical $\Aut_{\calO}$-torsor $\Coor(X)$ over $X$, called the {\em space of local coordinates of $X$}, defined as follows. For any $k$-algebra $R$, the set $\Coor(X)(R)$ consists of pairs $(x,\alpha)$ where $x\in X(R)$ and $\alpha$ is an $R$-linear continuous isomorphism $\alpha:R[[t]]\isom\hatO_{x}$ (here $\hatO_x$ is the completion of $\calO_{\XR}$ along the graph $\Gamma(x)$, see \cite[Sec. 2.2]{GSI} for notations). An element $\sigma\in\Aut_{\calO}(R)=\Aut(R[[t]])$ acts on $\Coor(X)(R)$ from the right by
\begin{equation*}
(x,\alpha)\cdot\sigma=(x,\alpha\circ\sigma), 
\end{equation*}
hence realizing the forgetful morphism $\Coor(X)\to X$ as a right $\Aut_{\calO}$-torsor.

The following fact is well-known.

\begin{lemma}\label{l:completediag}
Consider the $\Spf\calO_F$-bundle $\Coor(X)\twtimes{\Aut_{\calO}}\Spf(\calO_F)$ associated to the $\Aut_{\calO}$-torsor $\Coor(X)$ and the tautological action of $\Aut_{\calO}$ on $\Spf\calO_F$. We have a natural isomorphism over $X$
\begin{equation*}
\Coor(X)\twtimes{\Aut_{\calO}}\Spf\calO_F\cong \widehat{X^2}_\Delta,
\end{equation*}
where the RHS is the formal completion of $X^2$ along the diagonal $\Delta(X)\subset X^2$, viewed as a formal $X$-scheme via the projection to the first factor.
\end{lemma}

The pro-algebraic group $\Aut_{\calO}$ has the Levi quotient $\GG_m$ given by
\begin{eqnarray}\label{eq:defgrot}
\Aut_{\calO}&\to&\GG_m\\
\notag
\sigma&\mapsto&\sigma(t)/t\modo t.
\end{eqnarray}
We call this quotient $\GG_m$ the {\em rotation torus}, and denote it by $\grot$. Since we have fixed a uniformizing parameter $t$ of $F$, the quotient map $\Aut_{\calO}\to\grot$ admits a section under which $\lambda\in\grot(k)$ is the automorphism of $\calO_F$ given by: $a(t)\mapsto a(\lambda t)$ (for $a(t)\in k[[t]]=\calO_F$). We shall also identify $\grot$ with the subgroup of $\Aut_{\calO}$ given by the image of this section.

From Lem. \ref{l:completediag}, we immediately get

\begin{cor}\label{c:canX}
The $\grot$-torsor associated to the $\Aut_{\calO}$-torsor $\Coor(X)$ is naturally isomorphic to the $\GG_m$-torsor associated to the canonical bundle of $X$, i.e.,
\begin{equation*}
\Coor(X)\twtimes{\Aut_{\calO}}\grot\isom\rho_{\omega_X}.
\end{equation*}
(See \cite[Sec. 2.2]{GSI} for notations such as $\rho_{\omega_X}$).
\end{cor}

\subsection{Parahoric subgroups}\label{ss:pargp}
The purpose of this subsection is to fix some notations concerning the parahoric subgroups of $G(F)$. Since $G$ has an $\calO_F$-model $\bG$, the group $\Aut(\calO_F)$ acts on the (semisimple) Bruhat-Tits building $\frB(G,F)$ of $G(F)$ in a simplicial way, and hence on the set of parahoric subgroups of $G(F)$. We have also fixed a Borel subgroup $B\subset G$, which gives rise to an Iwahori subgroup $\bI\subset\bG$. Any parahoric subgroup $\bP$ containing $\bI$ is called a {\em standard parahoric subgroup} of $G(F)$.

\begin{remark}
We make a slight digression on the fixed point set of $\Aut(\calO_F)$ on $\frB(G,F)$. The maximal facets that are fixed by $\Aut(\calO_F)$ are in bijection with the fixed points of the $\Aut(\calO_F)$-action on the $k$-points of the affine flag variety $\Flag_G=G(F)/\bI$, which are given by $G(k)\tilw\bI/\bI$, for $\tilw\in\tilW$. In other words, the fixed point locus of the $\Aut(\calO_F)$ on $\frB(G,F)$ is the $G(k)$-orbit of the standard apartment (given by $T(F)\subset G(F)$, for some maximal torus $T\subset G$ over $k$). In particular, a facet in $\frB(G,F)$ is stable under $\Aut(\calO_F)$ if and only if it is pointwise fixed by $\Aut(\calO_F)$. Hence, a parahoric subgroup $\bP\subset G(F)$ is stable under $\Aut(\calO_F)$ if and only if its facet $\frF_{\bP}\subset\frB(G,F)$ is pointwise fixed by $\Aut(\calO_F)$. In particular, all standard parahoric subgroups are stable under $\Aut(\calO_F)$. 
\end{remark}

Let $\bP\subset G(F)$ be a parahoric subgroup. By Bruhat-Tits theory, $\bP$ determines a smooth group scheme over $\Spec\calO_{F}$ with generic fiber $G\otimes_k F$ and whose set of $\calO_F$-points is equal to $\bP$ (\cite[3.4.1]{Tits}). We still denote this $\calO_F$-group scheme by $\bP$. Let $\frg_{\bP}$ be the Lie algebra of $\bP$, which is a free $\calO_F$-module of rank $\dim_k G$. Let $L_{\bP}$ be the Levi quotient of $\bP$, which is a connected reductive group over $k$. Let $\widetilde{\bP}$ be the stabilizer of $\frF_{\bP}$ under $G(F)$ (equivalently, $\widetilde{\bP}$ is the normalizer of $\bP$ in $G(F)$). Let $\omega_{\bP}$ be the finite group $\widetilde{\bP}/\bP$.

% Let $W_{\bP}\subset\Wa$ be the finite Weyl group generated by the affine reflections fixing $\frF_{\bP}$ pointwise. Then $W_{\bP}$ can be identified with the Weyl group of $L_{\bP}$. Let $W'_{\bP}$ be the subgroup of $\tilW$ fixing $\frF_{\bP}$ pointwise. Note that if $G$ is semisimple, it could still happen that $W_{\bP}\subsetneqq W'_{\bP}$. The quotient $W'_{\bP}/W_{\bP}$ is the subgroup of $\Omega$ fixing each vertex of the affine Dynkin diagram of $G$ which define $\bP$; for example, maximal parahorics are correspond to a single vertex. Let $\tilW_{\bP}\subset\tilW$ be the stabilizer of the facet corresponding to $\bP$. Let $\Omega_{\bP}=\tilW_{\bP}/W_{\bP}$ be the quotient group, which is naturally a subgroup of $\Omega\cong\tilW/\Wa$. 

Let $G((t))=\Res_{F/k}(G\otimes_kF)$ and $G_{\bP}=\Res_{\calO_F/k}\bP$ be the (ind-)$k$-groups obtained by Weil restrictions. We call $G((t))$ the {\em loop group} of $G$. For $\bP=\bG$, we write $G[[t]]$ instead of $G_{\bG}$.

Let $\bP\subset G(F)$ be a parahoric subgroup stabilized by $\Aut(\calO_F)$, then $\Aut(\calO_F)$ naturally acts on the group scheme $\bP$, lifting its action on $\calO_F$. Moreover generally, for any $k$-algebra $R$, the group $\Aut_{\calO}(R)$ acts on $G_{\bP}(R)=\bP(R[[t]])$ and $G((t))(R)=G(R((t)))$, giving an action of the pro-algebraic group $\Aut_{\calO}$ on the group scheme $G_{\bP}$ and group ind-scheme $G((t))$. Since $\Aut_{\calO}$ acts on $G_{\bP}$, it also acts on the Levi quotient $L_{\bP}$, and the action necessarily factors through a finite-dimensional quotient of $\Aut_{\calO}$. We can form the twisted product
\begin{equation}\label{eq:twistLP}
\unL_{\bP}:=\Coor(X)\twtimes{\Aut_{\calO}}L_{\bP}
\end{equation}
which is a reductive group scheme over $X$ with geometric fibers isomorphic to $L_{\bP}$. Let $\frl_{\bP}$ be the Lie algebra of $L_{\bP}$, and let $\unl_{\bP}$ be the Lie algebra of $\unL_{\bP}$, which is the vector bundle $\Coor(X)\twtimes{\Aut_{\calO}}\frl_{\bP}$ over $X$.

Let $\bP$ be a standard parahoric subgroup. Note that the Borel $B$ gives a Borel subgroup $B^{\bP}_{\bI}\subset L_{\bP}$ whose quotient torus is canonically isomorphic to $T$. Let $W_{\bP}$ be the Weyl group of $L_{\bP}$ determined by $B^{\bP}_{\bI}$ and $T$. Then $W_{\bP}$ is naturally a subgroup of $\tilW$. In fact, any maximal torus in $B$ gives an apartment $\frA$ in $\frB(G,F)$, on which $\tilW$ acts by affine transformations. The Weyl group $W_{\bP}$ can be identified with the subgroup of $\tilW$ which fixes $\frF_{\bP}$ pointwise. The resulting subgroup $W_{\bP}\subset\tilW$ is independent of the choice of the maximal torus in $B$.

% parahoric Bun

\subsection{Bundles with parahoric level structures}\label{ss:bunpar}

\begin{defn}\label{def:tilBun} Let $\tilBun_{\infty}:k-\textup{Alg}\to{\textup{Groupoids}}$  be the fpqc sheaf associated to the following presheaf $\tilBun_{\infty}^{pre}$: for any $k$-algebra $R$, $\tilBun_{\infty}^{pre}(R)$ is the groupoid of quadruples $(x,\alpha,\calE,\tau_x)$ where
\begin{itemize}
\item $x\in X(R)$ with graph $\Gamma(x)\subset\XR$;
\item $\alpha:R[[t]]\isom\hatO_{x}$ a local coordinate;
\item $\calE$ is a $G$-torsor over $\XR$;
\item $\tau_x:G\times\disk_x\isom\calE|_{\disk_x}$ is a trivialization of the restriction of $\calE$ to $\disk_x=\Spec\hatO_x$.
\end{itemize}
\end{defn}

In other words, $\tilBun_{\infty}$ parametrizes $G$-bundles on $X$ with a full level structure at a point of $X$, and a choice of local coordinate at that point.

\begin{cons}\label{cons:gloopact} Consider the semi-direct product $G((t))\rtimes\Aut_{\calO}$ formed using the action of $\Aut_{\calO}$ on $G((t))$ defined in Sec. \ref{ss:pargp}. We claim that this group ind-scheme naturally acts on $\tilBun_{\infty}$ from the right. In fact, for any $k$-algebra $R$, $g\in G(R((t))),\sigma\in\Aut(R[[t]])$ and $(x,\alpha,\calE,\tau_x)\in\tilBun_{\infty}(R)$, let
\begin{equation}\label{eq:dotg}
R_{g,\sigma}(x,\alpha,\calE,\tau_x)=(x,\alpha\circ\sigma,\calE^g,\tau^g_x).
\end{equation}
Let us explain the notations. By a variant of the main result of \cite{BL} (using Tannakian formalism to reduce to the case of vector bundles), to give a $G$-torsor on $\XR$ is the same as to give $G$-torsors on $\XR-\Gamma(x)$ and on $\disk_x$ respectively, together with a $G$-isomorphism between their restrictions to $\pdisk_x$. Now let $\calE^g$ be the $G$-torsor on $\XR$ obtained by gluing $\calE|_{\XR-\Gamma(x)}$ with the trivial $G$-torsor $G\times\disk_x$ via the isomorphism
\begin{equation*}
G\times\pdisk_{x}\xrightarrow{\alpha^{-1}_*g\alpha_*}G\times\pdisk_{x}\xrightarrow{\tau_x}\calE|_{\pdisk_x}
\end{equation*}
Here $\alpha_*:\pdisk_x\to\Spec R((t))$ is induced by $\alpha$, hence the first arrow is the transport of the left multiplication by $g$ on $G\times\Spec R((t))$ to $G\times\pdisk_x$ via the local coordinate $\alpha$. Since left multiplication by $g\in G(R((t)))$ is an automorphism of the trivial right $G$-torsor on $\Spec R((t))$, $\alpha^{-1}_*g\alpha_*$ is an automorphism of the trivial right $G$-torsor on $\pdisk_x$. The trivialization $\tau^g_x$ is tautologically given by the construction of $\calE^g$.
\end{cons}

Let $\bP$ be a parahoric subgroup of $G(F)$ which is stable under $\Aut(\calO_F)$.
\begin{defn}\label{def:modpar}
The moduli stack $\Bun_{\bP}$ of {\em $G$-bundles over $X$ with parahoric level structures of type $\bP$} is the fpqc sheaf associated to the quotient presheaf $R\mapsto\tilBun_{\infty}(R)/(G_{\bP}\rtimes\Aut_{\calO})(R)$.
\end{defn}

We will use the notation $(x,\calE,\tau_x\modo\bP)$ to denote the point in $\Bun_{\bP}$ which is the image of $(x,\alpha,\calE,\tau_x)\in\tilBun_{\infty}(R)$.

\begin{remark}
Instead of taking quotients of $\tilBun_{\infty}$, we could also define $\Bun_{\bP}$ as the quotient of $\Bun_{\infty}$ by the group $\Coor(X)\twtimes{\Aut_{\calO}}G_{\bP}$, where $\Bun_{\infty}=\tilBun_{\infty}/\Aut_{\calO}$ is the moduli stack of $G$-torsors on $X$ with a full level structure at point of $X$.
\end{remark}

We look at several special cases of the above construction. The first case is $\bP=\bG=G(\calO_F)$. In this case we have:

\begin{lemma}\label{l:BunGagree}
There is a canonical isomorphism $\Bun_{\bG}\cong\Bun_G\times X$, where $\Bun_G$ is the usual moduli stack of $G$-torsors over $X$.
\end{lemma}
\begin{proof}
In other words, we need to show that the forgetful morphism $\tilBun_{\infty}\to\Bun_G\times X$ is a $G[[t]]\rtimes\Aut_{\calO}$-torsor. The only not-so-obvious part is the essential surjectivity, i.e., for any $k$-algebra $R$ and any $(x,\calE)\in X(R)\times\Bun_G(R)$, we have to find a trivialization of $\calE|_{\disk_x}$ locally in the flat or \'etale topology of $\Spec R$. Since $\calE|_{\Gamma(x)}$ is a $G$-torsor, by definition, there is an \'etale covering $\Spec R'\to \Spec R$ which trivializes $\calE|_{\Gamma(x)}$, i.e., there is a section $\tau'_{0}:\Spec R'\to\calE|_{\disk_x}$ over $\Spec R'\to\Spec R=\Gamma(x)\subset\disk_x$. Since $\calE|_{\disk_x}$ is smooth over $\disk_x$, the section $\tau'_0$ extends to a section $\tau':\Spec(R'\otimes_R\hatO_x)\to\calE|_{\disk_x}$. In other words, after pulling back to the \'etale covering $\Spec R'\to\Spec R$, $\calE|_{\disk_x}$ can be trivialized.
\end{proof}

The second case is when $\bP\subset G(\calO_F)$. Such $\bP$ are in 1-1 correspondence with parabolic subgroups $P\subset G$ over $k$. In this case, using Lem. \ref{l:BunGagree}, it is easy to see that $\Bun_{\bP}$ is the moduli stack of $G$-torsors on $X$ with a parabolic reduction of type $P$ at a point of $X$. More precisely, $\Bun_{\bP}(R)$ classifies tuples $(x,\calE,\calE^P_x)$, where
\begin{itemize}
\item $x\in X(R)$ with graph $\Gamma(x)$;
\item $\calE$ is a $G$-torsor over $\XR$;
\item $\calE_x^P$ is a $P$-reduction of the $G$-torsor $\calE|_{\Gamma(x)}$ over $\Gamma(x)$.
\end{itemize}
In particular, if $\bP=\bI$, the standard Iwahori subgroup, $\Bun_{\bI}$ is what we denoted by $\Bunpar_G$ in \cite[Sec. 3]{GSI}.

% properties of Bun_P

\subsection{Properties and examples of $\Bun_{\bP}$}

We first explore the dependence of $\Bun_{\bP}$ on the choice of $\bP$.

\begin{lemma}\label{l:Pclass} There is a natural right action of $\Omega_{\bP}=\widetilde{\bP}/\bP$ on $\Bun_{\bP}$. Moreover, suppose two parahoric subgroups $\bP$ and $\bQ$ are conjugate under $G(F)$ and are both stable under $\Aut(\calO_F)$, then there is an isomorphism $\Bun_{\bP}\isom\Bun_{\bQ}$ which is canonical up to pre-composition with the $\Omega_{\bP}$-action on $\Bun_{\bP}$ (or up to post-composition with the $\Omega_{\bQ}$-action on $\Bun_{\bQ}$).  
\end{lemma}
\begin{proof}
Let $g\in G(F)$ be such that $\bQ=g^{-1}\bP g$. Let $\tilBun_{\bP}=\tilBun_{\infty}/G_{\bP}$, which is an $\Aut_{\calO}$-torsor over $\Bun_{\bP}$. Since $g^{-1}\bP g=\bQ$, the natural right action of $g$ on $\tilBun_{\infty}$ (see Construction \ref{cons:gloopact}) descends to an isomorphism $R_g:\tilBun_{\bP}\isom\tilBun_{\bQ}$. Moreover, for any $\sigma\in\Aut(R[[t]])$, we have a commutative diagram
\begin{equation}\label{d:sigmaPQ}
\xymatrix{\tilBun_{\bP}(R)\ar[r]^{R_g}\ar[d]^{\cdot\sigma} & \tilBun_{\bQ}(R) \ar[d]^{\cdot\sigma} \\
\tilBun_{\bP}(R)\ar[r]^{R_{\sigma(g)}} & \tilBun_{\bQ}(R)}
\end{equation}
Since both $\bP$ and $\bQ$ are stable under $\sigma$, $\sigma(g)g^{-1}$ normalizes $\bP$, and hence $\sigma(g)g^{-1}\in\widetilde{\bP}(R)$. The assignment $\sigma\mapsto\sigma(g)g^{-1}$ gives a morphism from the connected pro-algebraic group $\Aut_{\calO}$ to the discrete group $\widetilde{\bP}/\bP$, which must be trivial. Therefore $\sigma(g)g^{-1}\in\bP(R)$. It is clear that right multiplication by any element in $\bP$ induces the identity morphism on $\tilBun_{\bP}$, therefore
\begin{equation*}
R_{\sigma(g)}=R_{g}\circ R_{\sigma(g)g^{-1}}=R_{g}:\tilBun_{\bP}(R)\isom\tilBun_{\bQ}(R).
\end{equation*}
Using diagram (\ref{d:sigmaPQ}), we conclude that $R_g:\tilBun_{\bP}(R)\isom\tilBun_{\bQ}(R)$ is equivariant under $\Aut_{\calO}(R)$, hence descends to an isomorphism
\begin{equation}\label{eq:isomPQ}
R_g:\Bun_{\bP}\isom\Bun_{\bQ}.
\end{equation}

Finally we define the $\Omega_{\bP}$-action on $\Bun_{\bP}$, and check that the isomorphism (\ref{eq:isomPQ}) is canonical up to this action. If $g'$ is another element of $G(F)$ such that $\bQ=g'^{-1}\bP g'$, then $g'g^{-1}$ normalizes $\bP$, hence $g'g^{-1}\in\widetilde{\bP}$. We can write the isomorphism $R_{g'}$ as the composition:
\begin{equation*}
\Bun_{\bP}\xrightarrow{R_{g'g^{-1}}}\Bun_{\bP}\xrightarrow{R_g}\Bun_{\bQ},
\end{equation*}
where the first isomorphism only depends on the image of $g'g^{-1}$ in $\Omega_{\bP}=\widetilde{\bP}/\bP$. Taking $\bQ=\bP$, we get the desired $\Omega_{\bP}$-action on $\Bun_{\bP}$. In general, the isomorphisms between $\Bun_{\bP}$ and $\Bun_{\bQ}$ given by the various $R_g$ only differ by the action of $\Omega_{\bP}$ on $\Bun_{\bP}$.
\end{proof}

For parahoric subgroups $\bP,\bQ$, let
\begin{equation*}
\Omega_{\bP,\bQ}:=\{[g]\in\bP\backslash G(F)/\bQ|g^{-1}\bP g=\bQ\}.
\end{equation*}
For $g\in G(F)$ such that $g^{-1}\bP g=\bQ$, let $[g]\in\Omega_{\bP,\bQ}$ be the corresponding double coset. In particular, $\Omega_{\bP}=\Omega_{\bP,\bP}$. If $\bP$ and $\bQ$ are both stable under $\Aut(\calO_F)$, the proof of Lem. \ref{l:Pclass} gives for each $[g]\in\Omega_{\bP,\bQ}$ a canonical isomorphism
\begin{equation}\label{eq:rg}
R_{[g]}:\Bun_{\bP}\to\Bun_{\bQ}.
\end{equation}

Suppose we have an inclusion $\bP\subset\bQ$ of parahoric subgroups which are both stable under $\Aut(\calO_F)$, then by construction we have a forgetful morphism
\begin{equation}\label{eq:omPQ}
\forg\PQ:\Bun_{\bP}\to\Bun_{\bQ}
\end{equation}
whose fibers are isomorphic to $G_{\bQ}/G_{\bP}$, which is a partial flag variety of the reductive group $L_{\bQ}$. In particular, $\forg\PQ$ is representable, proper, smooth and surjective.

\begin{cor}\label{c:Bunstack}
For any parahoric subgroup $\bP\subset G(F)$ stable under $\Aut(\calO_F)$, the stack $\Bun_{\bP}$ is an algebraic stack locally of finite type.
\end{cor}
\begin{proof}
By Lem. \ref{l:Pclass}, we only need to check the statement for standard parahoric subgroups. Since the morphism $\forg^{\bG}_{\bI}:\Bun_{\bI}\to\Bun_{\bG}\cong\Bun_G\times X$ is representable and of finite type, and $\Bun_G$ is an algebraic stack locally of finite type, $\Bun_{\bI}$ is also algebraic and locally of finite type. On the other hand, since $\forg^{\bP}_{\bI}:\Bun_{\bI}\to\Bun_{\bP}$ is representable, smooth and surjective, $\Bun_{\bP}$ is also algebraic and locally of finite type.
\end{proof}

We describe examples of $\Bun_{\bP}$ for classical groups of type A, B and C.

\begin{exam}\label{ex:BunA} Let $G=\SL(n)$. The standard parahoric subgroups are in 1-1 correspondence with sequences of integers
\begin{equation*}
\undi=(0\leq i_0<\cdots<i_m<n), m\geq0
\end{equation*}
For each such sequence $\undi$, let $\bP_{\undi}$ be the corresponding parahoric subgroup. Then $\Bun_{\bP_{\undi}}$ classifies
\begin{equation*}
(x,\calE_{i_0}\supset\calE_{i_1}\supset\cdots\supset\calE_{i_m}\supset\calE_{i_0}(-x),\delta)
\end{equation*}
where $x\in X$, $\calE_{i_j}$ are vector bundles of rank $n$ on $X$ such that $\calE_{i_0}/\calE_{i_j}$ has length $i_j-i_0$ for $j=0,1,\cdots,m$, and $\delta$ is an isomorphism $\det(\calE_{i_0})\cong\calO_X(-i_0x)$.
\end{exam}

\begin{exam}\label{ex:BunBC} Let $G=\SO(2n+1)$ (resp. $G=\Sp(2n)$). The standard parahoric subgroups are in 1-1 correspondence with sequences of integers
\begin{equation*}
\undi=(0\leq i_0<\cdots<i_m\leq n), m\geq0.
\end{equation*}
For each such sequence $\undi$, let $\bP_{\undi}$ be the corresponding parahoric subgroup. Then $\Bun_{\bP_{\undi}}$ classifies
\begin{equation*}
(x,\calE_{i_0}\supset\cdots\supset\calE_{i_m}\supset\calE^{\bot}_{i_m}(-x)\supset\cdots\supset\calE^{\bot}_{i_0}(-x)\supset\calE_{i_0}(-x),\sigma)
\end{equation*}
where $x\in X$, $\calE_{i_j}$ are vector bundles of rank $2n+1$ (resp. $2n$) on $X$ such that $\calE_{i_0}/\calE_{i_j}$ has length $i_j-i_0$ for $j=0,1,\cdots,m$, and $\sigma$ is a symmetric (resp. alternating) pairing
\begin{equation*}
\sigma:\calE_{i_0}\otimes\calE_{i_0}\to\calO_X.
\end{equation*}
For any subsheaf $\calE\subset\calE_{i_0}$ of finite colength, $\calE^{\bot}$ denotes the subsheaf of the sheaf of rational sections $f$ of $\calE$ such that $\sigma(f,\calE)\subset\calO_X$, which is also a vector bundle over $X$.
\end{exam}

% parahoric Hitchin moduli

\subsection{The parahoric Hitchin fibrations}
In this subsection, we define parahoric analogues of $\Mpar$ and $\fpar$ considered in \cite[Sec. 3.1]{GSI}. These are analogues of the partial Grothendieck resolutions in the classical Springer theory, cf. \cite{BM}.

We first define the notion of Higgs fields in the parahoric situation. 

\begin{cons}[the Higgs fields]\label{cons:adP} For any $k$-algebra $R$ and $(x,\alpha,\calE,\tau_x)\in\tilBun_{\infty}(R)$, consider the composition
\begin{equation}\label{eq:j}
j_*j^*\Ad(\calE)\to\Ad(\calE)\otimes\hatO^{\punc}_x\xrightarrow{\tau_x^{-1}}\frg\otimes_k\hatO^{\punc}_x\xrightarrow{\alpha^{-1}}\frg\otimes_kR((t))
\end{equation}
where $j:\XR-\Gamma(x)\hookrightarrow \XR$ is the inclusion and the first arrow is the natural embedding. Let $\Ad_{\bP}(\calE)$ be the preimage of $\frg_{\bP}\otimes_kR\subset\frg\otimes_kR((t))$ under the injection (\ref{eq:j}). Sheafifying this procedure, the assignment $(x,\calE,\tau_x\modo\bP)\mapsto\Ad_{\bP}(\calE)$ gives a quasi-coherent sheaf $\Ad_{\bP}$ on $\tilBun_{\infty}\times X$.
\end{cons}

It is easy to check that
\begin{lemma}\label{l:AdP}
\begin{enumerate}
\item []
\item The quasi-coherent sheaf $\Ad_{\bP}$ descends to $\Bun_{\bP}\times X$;
\item The quasi-coherent sheaf $\Ad_{\bP}$ over $\Bun_{\bP}\times X$ is in fact coherent;
\item If both $\bP$ and $\bQ$ are stable under $\Aut(\calO_F)$ and $[g]\in\Omega_{\bP,\bQ}$, then there is a canonical isomorphism $R_{[g]}^*\Ad_{\bQ}\cong\Ad_{\bP}$ satisfying the obvious transitivity conditions. (Recall the isomorphism $R_{[g]}$ from (\ref{eq:rg})). In particular, $\Ad_{\bP}$ has a natural $\Omega_{\bP}$-equivariant structure.
\end{enumerate}
\end{lemma}
\begin{proof}
(1) Since $\frg_{\bP}$ is stable under $G_{\bP}\rtimes\Aut(\calO_F)$, the subsheaf $\Ad_{\bP}(\calE)\subset j_*j^*\Ad(\calE)$ only depends on the image of $(x,\alpha,\calE,\tau_x)$ in $\Bun_{\bP}(R)$.

(3) Similar to the proof of Lem. \ref{l:Pclass}.

(2) Fix any $k$-algebra $R$ and $(x,\calE,\tau_x\modo\bP)\in\Bun_{\bP}(R)$, we want to show that $\Ad_{\bP}(\calE)$ is a coherent sheaf on $\XR$. For $\bP=\bG$, clearly $\Ad_{\bG}(\calE)=\Ad(\calE)$ is coherent on $\XR$. For $\bP=\bI$ and any point $(x,\calE,\calE^B_x)\in\Bun_{\bI}(R)$ (recall $\calE^B_x$ is a $B$-reduction of $\calE|_{\Gamma(x)}$), we have an exact sequence
\begin{equation*}
0\to\Ad_{\bI}(\calE)\to\Ad(\calE)\to i_*\left(\Ad(\calE|_{\Gamma(x)})/\Ad(\calE^B_x)\right)\to0
\end{equation*}
where $i:\Gamma(x)\hookrightarrow\XR$ is the closed inclusion. Since the middle and final terms of the above exact sequence are coherent, $\Ad_{\bI}(\calE)$ is also coherent.

In general, by (3), we can reduce to the case $\bP\supset\bI$. In this case we have an embedding $\Ad_{\bI}(\calE)\hookrightarrow\Ad_{\bP}(\calE)$ whose cokernel is again a finite $R$-module supported on $\Gamma(x)$, hence $\Ad_{\bP}(\calE)$ is a coherent sheaf on $\XR$.
\end{proof}

As in the usual definition of the Hitchin moduli stack, we fix a divisor $D$ on $X$ with $\deg(D)\geq2g_X$.

\begin{defn}\label{def:hitp}
The {\em Hitchin moduli stack of $G$-bundles over $X$ (with respect to $D$) with parahoric level structures of type $\bP$} ({\em parahoric Hitchin moduli stack of type $\bP$} for short) is the fpqc sheaf $\calM_{\bP}:k-\textup{Alg}\to\textup{Groupoids}$ which associates to every $k$-algebra $R$ the groupoid of pairs $(\xi,\varphi)$ where 
\begin{itemize}
\item $\xi=(x,\calE,\tau_x\modo\bP)\in\Bun_{\bP}(R)$;
\item $\varphi\in\cohog{0}{\XR,\Ad_{\bP}(\calE)\otimes\calO_X(D)}$.
\end{itemize}
\end{defn}

\begin{remark}\label{r:standard}
By Lem. \ref{l:AdP}(3), the group $\Omega_{\bP}$ naturally acts on the parahoric Hitchin moduli stack $\calM_{\bP}$, and $\calM_{\bP}$ only depends on the the conjugacy class of $\bP$ up to this action of $\Omega_{\bP}$. Therefore we can concentrate on the study of $\calM_{\bP}$ for standard parahoric subgroups $\bP\supset\bI$.
\end{remark}

\begin{lemma}
The Hitchin moduli stack $\calM_{\bP}$ is an algebraic stack locally of finite type.
\end{lemma}
\begin{proof}
Consider the forgetful morphism $\calM_{\bP}\to\Bun_{\bP}$. The fiber of this morphism over a point $(x,\calE,\tau_x\modo\bP)\in\Bun_{\bP}(R)$ is the finite $R$-module $\cohog{0}{\XR,\Ad_{\bP}(\calE)(D)}$ (the finiteness follows from the coherence of $\Ad_{\bP}(\calE)$ as proved in Lem. \ref{l:AdP}(2) and the properness of $X$). Therefore, the forgetful morphism $\calM_{\bP}\to\Bun_{\bP}$ is representable and of finite type. By Cor. \ref{c:Bunstack}, $\Bun_{\bP}$ is algebraic and locally of finite type, hence so is $\calM_{\bP}$.
\end{proof}

\begin{cons}\label{cons:ev} We claim that there is a natural morphism
\begin{equation}\label{eq:ev}
\ev_{\bP}:\calM_{\bP}\to[\unl_{\bP}/\unL_{\bP}]_D
\end{equation}
of ``evaluating the Higgs fields at the point of the $\bP$-level structure''. In fact, to construct $\ev_{\bP}$, it suffices to construct a morphism
\begin{equation*}
\tilev_{\bP}:\tilBun_{\infty}\times_{\Bun_{\bP}}\calM_{\bP}\to\frl\twtimes{\GG_m}\rho_D
\end{equation*}
which is equivariant under $G_{\bP}\rtimes\Aut_{\calO}$. Here the $G_{\bP}\rtimes\Aut_{\calO}$-action on the LHS is on the $\tilBun_{\infty}$-factor, and the action on the RHS factors through the $L_{\bP}\rtimes\Aut_{\calO}$-action on $\frl_{\bP}$, with $L_{\bP}$ acting by conjugation.

For any $k$-algebra $R$ and $(x,\alpha,\calE,\tau_x,\varphi)\in(\tilBun_{\infty}\times_{\Bun_{\bP}}\calM_{\bP})(R)$, by the definition of $\Ad_{\bP}(\calE)$, the maps in (\ref{eq:j}) give
\begin{equation*}
\Ad_{\bP}(\calE)\to\frg_{\bP}\otimes_kR\twoheadrightarrow\frl_{\bP}\otimes_kR.
\end{equation*}
Twisting by $\calO_X(D)$, we get
\begin{equation*}
\tilev_{\bP,x}:\cohog{0}{\XR,\Ad_{\bP}(\calE)(D)}\to\frl_{\bP}\otimes_kx^*\calO_X(D).
\end{equation*}
The assignment $(x,\alpha,\calE,\tau_x,\varphi)\mapsto\tilev_{\bP,x}(\varphi)$ gives the desired morphism $\tilev_{\bP}$. It is easy to check that $\tilev_{\bP}$ is equivariant under $G_{\bP}\rtimes\Aut_{\calO}$, hence giving the desired morphism $\ev_{\bP}$ in (\ref{eq:ev}).
\end{cons}

\subsubsection{Morphisms between two parahoric Hitchin stacks} For two standard parahoric subgroups $\bP\subset\bQ$, there is a unique parabolic subgroup $B\PQ\subset L_{\bQ}$, such that $\bP$ is the inverse image of $B\PQ$ under the natural quotient $\bQ\twoheadrightarrow L_{\bQ}$. There is a canonical $\Aut_{\calO}$-action on $B\PQ$ making the embedding $B\PQ\hookrightarrow L_{\bQ}$ equivariant under $\Aut_{\calO}$. Let $\frb\PQ$ be the Lie algebra of $B\PQ$ and let $\unB\PQ,\unb\PQ$ be the group scheme and Lie algebra over $X$ obtained by applying $\Coor(X)\twtimes{\Aut_{\calO}}(-)$.

Since $B\PQ$ is a quotient of $\bP$, the same construction as in Construction \ref{cons:ev} gives the {\em relative evaluation map}
\begin{equation*}
\ev\PQ:\calM_{\bP}\to[\unb\PQ/\unB\PQ]_D.
\end{equation*}
Similar to the morphism $\forg\PQ:\Bun_{\bP}\to\Bun_{\bQ}$ in (\ref{eq:omPQ}), there is a morphism
\begin{equation}\label{eq:tilomPQ}
\tforg\PQ:\calM_{\bP}\to\calM_{\bQ}
\end{equation}
lifting $\forg\PQ$.

\begin{lemma}\label{l:CartPQ}
We have a Cartesian diagram
\begin{equation*}
\xymatrix{\calM_{\bP}\ar[r]^{\ev\PQ}\ar[d]^{\tforg\PQ} & [\unb\PQ/\unB\PQ]_D\ar[d]^{\pi\PQ}\\
\calM_{\bQ}\ar[r]^{\ev_{\bQ}} & [\unl_{\bQ}/\unL_{\bQ}]_D}
\end{equation*}
In particular, the morphism $\tforg\PQ$ is proper and surjective.
\end{lemma}
\begin{proof}
The first statement follows by tracing down the contructions of the evaluation maps. The morphism $\pi\PQ$ is (locally on $X$) the partial Grothendieck resolution associated to the parabolic subgroup $B\PQ$ of $L_{\bQ}$, hence proper and surjective. Therefore $\tforg\PQ$ is also proper and surjective.
\end{proof}

Now we define the parahoric Hitchin fibrations. Recall from \cite[Sec. 3.1]{GSI} that we have the usual Hitchin base space $\AHit=\cohog{0}{X,\frc_D}$. For any $k$-algebra $R$ and $(x,\calE,\tau_x\modo\bP)\in\Bun_{\bP}(R)$, the natural map of taking invariants $\Ad(\calE)(D)\to\frc_D$ gives a map
\begin{eqnarray*}
\chi_{\bP,\calE}:\cohog{0}{\XR,\Ad_{\bP}(\calE)(D)}&\hookrightarrow&\cohog{0}{\XR-\Gamma(x),\Ad(\calE)(D)}\\
&\to& \cohog{0}{\XR-\Gamma(x),\frc_D}.
\end{eqnarray*}

\begin{lemma}\label{l:fp}
The image of the map $\chi_{\bP,\calE}$ lands in $\cohog{0}{\XR,\frc_D}$, hence giving a morphism
\begin{equation*}
f_{\bP}:\calM_{\bP}\to\AHit\times X.
\end{equation*}
\end{lemma}
\begin{proof}
For $\bP\subset\bG$, $\Ad_{\bP}(\calE)\subset\Ad(\calE)$, hence the image of $\chi_{\bP,\calE}$ obvious lands in $\cohog{0}{\XR,\frc_D}$. In particular, the statement holds for $\bP=\bI$.

In general, we may assume $\bI\subset\bP$. By Lem. \ref{l:CartPQ}, for any point $(\xi,\varphi)\in\calM_{\bP}(R)$, after passing to a fpqc base change of $R$, there is always a point $(\tilxi,\tilphi)\in\Mpar(R)$ mapping to it under $\tforg^{\bQ}_{\bI}$. Since $\chi_{\bP,\calE}(\varphi)=\chi_{\bI,\calE}(\tilphi)$, we conclude that $\chi_{\bP,\calE}(\varphi)\in \cohog{0}{\XR,\frc_D}$.
\end{proof}

\begin{defn}\label{def:Pfib}
The morphism $f_{\bP}:\calM_{\bP}\to\AHit\times X$ constructed in Lem. \ref{l:fp} is called the {\em parahoric Hitchin fibration of type $\bP$}.
\end{defn}

Let $\frc_{\bP}=\frl_{\bP}\sslash L_{\bP}=\frt\sslash W_{\bP}$ be the GIT quotient of the reductive Lie algebra $\frl_{\bP}$ over $k$. Recall that the Weyl group $W_{\bP}$ of $L_{\bP}$ can be identified with a subgroup of $\tilW$. The projection $\tilW\to W$ restricted to $W_{\bP}$ induces an injection $W_{\bP}\hookrightarrow W$. Moreover, for $\bP\subset\bQ$, we have an inclusion $W_{\bP}\subset W_{\bQ}$. Passing to the invariant quotients, we have canonical finite flat morphisms
\begin{equation*}
\frt\xrightarrow{q^{\bP}_{\bI}}\frc_{\bP}\xrightarrow{q\PQ}\frc_{\bQ}\xrightarrow{q_{\bQ}}\frc.
\end{equation*}

\begin{defn}\label{def:tcAP}
The {\em enhanced Hitchin base} $\tcA_{\bP}$ of type $\bP$ is defined by the following Cartesian square
\begin{equation*}
\xymatrix{\tcA_{\bP}\ar[r]\ar[d]^{q_{\bP}} & \frc_{\bP,D}\ar[d]^{q_{\bP}}\\
\AHit\times X\ar[r]^{\ev} & \frc_D}
\end{equation*}
where ``$\ev$'' is the evaluation map. Note that $\tcA_{\bI}$ is the universal cameral cover $\tcA$ in \cite[Def. 3.1.7]{GSI}, and $\tcA_{\bG}=\AHit\times X$.
\end{defn}

\begin{lemma}
The projection $\Coor(X)\times[\frl_{\bP}/L_{\bP}]\to[\frl_{\bP}/L_{\bP}]\to\frc_{\bP}$ factors through the quotient $\Coor(X)\twtimes{\Aut_{\calO}}[\frl_{\bP}/L_{\bP}]$ so that we have a morphism
\begin{equation*}
\chi_{\bP}:[\unl_{\bP}/\unL_{\bP}]\to\frc_{\bP}.
\end{equation*}
\end{lemma}
\begin{proof}
Since $\Aut_{\calO}$ is connected, its action on $L_{\bP}$ factors through the adjoint action of $L^{ad}_{\bP}$ on $L_{\bP}$. Therefore the adjoint GIT quotient $[\frl_{\bP}/L_{\bP}]\to\frc_{\bP}$ is automatically $\Aut_{\calO}$-invariant and the conclusion follows.
\end{proof}

Using the morphism $\calM_{\bP}\xrightarrow{\ev_{\bP}}[\unl_{\bP}/\unL_{\bP}]_D\xrightarrow{\chi_{\bP}}\frc_{\bP,D}$, we get the {\em enhanced parahoric Hitchin fibration of type $\bP$}:
\begin{equation}\label{eq:enhancefP}
\tilf_{\bP}:\calM_{\bP}\xrightarrow{(\chi_{\bP}\circ\ev_{\bP},f_{\bP})}\frc_{\bP,D}\times_{\frc_D}(\calA\times X)=\tcA_{\bP}.
\end{equation}

% properties of M_P

\subsection{Properties of $\calM_{\bP}$}\label{ss:geomex}
In this subsection, we list a few properties of $\calM_{\bP}$ and $f_{\bP}$ parallel to the properties studied in \cite[Sec. 3.2-3.5]{GSI} for the parabolic Hitchin fibration. Most proofs will be omitted because they are almost the same as the proofs in the case of $\Mpar$.
We also give examples of parahoric Hitchin fibers.

Recall from \cite[Sec. 3.2]{GSI} that we have a Picard stack $\calP$ over $\AHit$ whose fiber over $a\in\AHit$ classifies $J_a$-torsors over $X$.

\begin{cons} For each standard parahoric subgroup $\bP\subset G(F)$, we construct an action of $\calP$ on $\calM_{\bP}$ such that the morphisms $\tforg\PQ:\calM_{\bP}\to\calM_{\bQ}$ are equivariant under $\calP$.

For any $k$-algebra $R$ and an object $(x,\calE,\tau_x\modo\bP,\varphi)\in\Bun_{\bP}(R)$, we can define the sheaf of automorphisms $\unAut_{\bP}(\calE,\tau_x,\varphi)$ of this object. More precisely, $\unAut_{\bP}(\calE,\tau_x,\varphi)$ is a fpqc sheaf of groups over $\XR$ such that for any $u:U\to\XR$, $\unAut_{\bP}(\calE,\tau_x,\varphi)(U)$ is the set of automorphisms of $u^*\calE$ which preserve the Higgs field $u^*\varphi$ and the trivialization $u^*\tau_x$ up to $G_{\bP}$. For $\bP\subset\bQ$ we have an inclusion of sheaves of groups
\begin{equation}\label{eq:Autincl}
\unAut_{\bP}(\calE,\tau_x,\varphi)\hookrightarrow\unAut_{\bQ}(\calE,\tau_x,\varphi).
\end{equation}
Moreover, we always have an inclusion
\begin{equation*}
\unAut_{\bP}(\calE,\tau_x,\varphi)\hookrightarrow j_*(\unAut(\calE,\varphi)|_{\XR-\Gamma(x)}),
\end{equation*}
where $j:\XR-\Gamma(x)\hookrightarrow\XR$ is the inclusion.

Let $a=f^{\Hit}(\calE,\varphi)\in\AHit(R)$. Recall from \cite[Sec. 3.2]{GSI} we see that the group scheme $J$ naturally maps to the universal centralizer group scheme over $\frg$. Another way to say this is that we have a natural homomorphism $J_a\in\unAut(\calE,\varphi)$. Consider the homomorphism of sheaves of groups
\begin{equation}\label{eq:JPaut}
J_a\to j_*j^*J_a\to j_*(\unAut(\calE,\varphi)|_{\XR-\Gamma(x)}).
\end{equation}
We claim that the image of this homomorphism lies in $\unAut_{\bP}(\calE,\tau_x,\varphi)$. In fact, since $\tforg^{\bP}_{\bI}$ is surjective by Lem. \ref{l:CartPQ}, we may assume that $(x,\calE,\tau_x\modo\bP,\varphi)$ comes from a point $(x,\calE,\tau_x\modo\bI,\varphi)\in\Mpar(R)$. From \cite[Lem. 3.2.2]{GSI}, we see that the group scheme $J$ naturally maps to the centralizer group scheme of $\frb$. Again, this can be rephrased as saying that the image of $J_a\in\unAut(\calE,\varphi)$ lies in $\unAut_{\bI}(\calE,\tau_x,\varphi)$. By the inclusion (\ref{eq:Autincl}) applied to $\bI\subset\bP$, the image of $J_a$ under (\ref{eq:JPaut}) also lies in $\unAut_{\bP}(\calE,\tau_x,\varphi)$.

With the homomorphism $J_a\to\unAut_{\bP}(\calE,\tau_x,\varphi)$, we can define the action of $Q^J\in\calP_a(R)$ by
\begin{equation*}
Q^J\cdot(x,\calE,\tau_x\modo\bP,\varphi)=(x,Q^J\twtimes{J_a}(\calE,\tau_x\modo\bP,\varphi)).
\end{equation*}
\end{cons}

\begin{lemma}
The action of $\calP$ on $\calM_{\bP}$ preserves the morphism $\tilf_{\bP}:\calM_{\bP}\to\tcA_{\bP}$.
\end{lemma}
\begin{proof}
We have a commutative diagram
\begin{equation}\label{d:twosq}
\xymatrix{\calP\times\Mpar\ar@<1ex>[r]^{\act_{\bI}}\ar@<-1ex>[r]_{\proj_{\bI}}\ar[d]_{\id_{\calP}\times\tforg^{\bP}_{\bI}} & \Mpar\ar[d]^{\tforg^{\bP}_{\bI}}\ar[r]^{\tilf} & \tcA\ar[d]^{q^{\bP}_{\bI}}\\
\calP\times\calM_{\bP}\ar@<1ex>[r]^{\act_{\bP}}\ar@<-1ex>[r]_{\proj_{\bP}} & \calM_{\bP}\ar[r]^{\tilf_{\bP}} & \tcA_{\bP}}
\end{equation}
By \cite[Lem. 3.2.5]{GSI}, the $\calP$-action on $\Mpar$ preserves $\tilf$, hence
\begin{equation*}
q^{\bP}_{\bI}\circ\tilf\circ\act_{\bI}=q^{\bP}_{\bI}\circ\tilf\circ\proj_{\bI}.
\end{equation*}
Combining with the diagram \ref{d:twosq}, we get
\begin{equation*}
\tilf_{\bP}\circ\act_{\bP}\circ(\id_{\calP}\times\tforg^{\bP}_{\bI})=\tilf_{\bP}\circ\proj_{\bP}\circ(\id_{\calP}\times\tforg^{\bP}_{\bI}).
\end{equation*}
Since $\tforg^{\bP}_{\bI}:\Mpar\to\calM_{\bP}$ is surjective, we conclude that
\begin{equation*}
\tilf_{\bP}\circ\act_{\bP}=\tilf_{\bP}\circ\proj_{\bP}.
\qedhere
\end{equation*}
\end{proof}

\subsubsection{Local counterpart of $\calM_{\bP}$} For a point $x\in X(k)$ and an element $\gamma\in\frg(F_x)$, we can define the {\em affine Springer fiber of type $\bP$} in a similar way as one defines affine Springer fibers in the affine flag varieties. It is a closed sub-ind-scheme $M_{\bP,x}(\gamma)\subset\Flag_{\bP,x}:=(\Res_{F_x/k}G)/\Res_{\hatO_x/k}\bP_x$, here $\bP_x\subset G(F_x)$ is the parahoric subgroup corresponding to $\bP$ under any choice of local coordinate at $x$. Again, the group ind-scheme $P_x(J_a)$ acts on $M_{\bP,x}(\gamma)$ whenever $\chi(\gamma)=a\in\frc(\hatO_x)$.

Fix $(a,x)\in\Ah(k)\times X(k)$. As in \cite[Sec. 3.3]{GSI}, from the Kostant section $\epsilon:\AHit\to\MHit$ and the choices of local trivializations we get local data $\gamma_{a,x}\in\frg(\hatO_x)$. Analogous to the product formula in \cite[Prop. 3.3.3]{GSI}, we have:

\begin{prop}[Product formula]\label{eq:parprod}
Let $(a,x)\in\Ah(k)\times X(k)$, and let $U_a$ be the dense open subset $a^{-1}\frc_D^{\rs}$ of $X$. We have a homeomorphism of stacks:
\begin{equation*}
\calP_a\twtimes{P_x^{\red}(J_a)\times P'}\left(M^{\red}_{\bP,x}(\gamma_{a,x})\times M'\right)\to\calM_{\bP,a,x}.
\end{equation*}
where
\begin{eqnarray*}
P'&=&\prod_{y\in X-U_a-\{x\}}P_y^{\red}(J_a);\\
M'&=&\prod_{y\in X-U_a-\{x\}}M^{\Hit,\red}_y(\gamma_{a,y}).
\end{eqnarray*}
\end{prop}

Parallel to \cite[Prop. 3.4.1]{GSI}, we have

\begin{prop}\label{p:MPsmooth} Recall $\deg(D)\geq2g_X$. Then we have:
\begin{enumerate}
\item The stack $\calM_{\bP}|_{\Ah}$ is smooth;
\item The stack $\calM_{\bP}|_{\Aa}$ is Deligne-Mumford;
\item The morphism $f^{\ani}_{\bP}:\calM_{\bP}|_{\Aa}\to\Aa\times X$ is flat and proper.  
\end{enumerate}
\end{prop}

% small maps; this leads to another construction of affine Weyl action, at least for Waff.

\subsubsection{Small maps} In classical Springer theory, the morphisms between the various partial Grothendieck resolutions are small, cf. \cite{BM}. According to Lem. \ref{l:CartPQ}, the morphisms $\tforg\PQ$ between the parahoric Hitchin moduli stacks are base changes of the morphisms between partial Grothendieck resolutions; however, it it not clear that $\ev_{\bQ}$ is flat, so that we cannot conclude immediately that $\tforg\PQ$ is also small. In the following proposition, we prove the smallness of $\tforg\PQ$ over the locus $\calA\subset\Aa$ where the codimension estimate in \cite[Prop. 3.5.5]{GSI} holds (see \cite[Rem. 3.5.6]{GSI}).

\begin{prop}\label{p:forsmall} Let $\bP\subset\bQ$ be standard parahoric subgroups. Then
\begin{enumerate}
\item The morphism $\tforg\PQ:\calM_{\bP}|_{\calA}\to\calM_{\bQ}|_{\calA}$ is small.
\item The morphism $\nu\PQ:\calM_{\bP}|_{\calA}\to\calM_{\bQ}|_{\calA}\times_{\tcA_{\bQ}}\tcA_{\bP}$ is small and birational (i.e., a small resolution of singularities).
\end{enumerate}
\end{prop}
\begin{proof} Let us restrict all stacks over $\AHit$ to the open subset $\calA$ without changing notations.

(1) For each integer $d\geq1$, let $Z_{\geq d}$ be the closed subschemes of $\calM_{\bQ}$ over which the fibers of $\tforg\PQ$ have dimension $\geq d$. By Lem. \ref{l:CartPQ}, the fiber of $\tforg\PQ$ over a point $(x,\calE,\tau_x\modo\bQ,\varphi)\in\calM_{\bQ}(k)$ is the partial Springer fiber corresponding to the conjugacy class $\varphi(x)\in[\frl_{\bQ}/L_{\bQ}]$, and is contained in the affine Springer fiber $M_{\bP,x}(\gamma)$ (here $\gamma\in\frg(\hatO_x)$ is a representative of $\varphi$ at $x$ after choosing a trivialization of $\calE$ over $\disk_x$). Therefore, if $(x,\calE,\tau_x\modo\bQ,\varphi)\in Z_{\geq d}$, then
\begin{equation*}
\delta(a,x)=\dim M^{\parab}_{x}(\gamma)\geq\dim M_{\bQ,x}(\gamma)\geq\tforg^{\bQ,-1}_{\bP}(x,\calE,\tau_x\modo\bQ,\varphi)\geq d
\end{equation*}
where $a=f^{\Hit}(\calE,\varphi)\in\calA(k)$. Therefore the image of $Z_{\geq d}\to\calA\times X$ lies in $(\calA\times X)_{\geq d}$, which has codimension $\geq d+1$ by \cite[Cor. 3.5.7]{GSI}. 

On the other hand, let $Y_{\geq d}=\tforg^{\bQ,-1}_{\bP}(Z_{\geq d})$, which maps to $(\calA\times X)_{\geq d}$ by the above discussion. Since $f_{\bP}:\calM_{\bP}\to\calA\times X$ is flat by Prop. \ref{p:MPsmooth}(3), we have
\begin{equation*}
\codim_{\calM_{\bP}}(Y_{\geq d})\geq\codim_{\calA\times X}((\calA\times X)_{\geq d})\geq d+1.
\end{equation*}
Therefore
\begin{equation*}
\dim(Z_{\geq d})\leq\dim(Y_{\geq d})-d\leq\dim(\calM_{\bP})-(d+1)-d=\dim(\calM_{\bQ})-2d-1.
\end{equation*}
This proves the smallness of $\tforg\PQ$ (we know $\tforg\PQ$ is surjective by Lem. \ref{l:CartPQ}).

(2) The commutative diagram
\begin{equation*}
\xymatrix{[\unb\PQ/\unB\PQ]\ar[d]^{\pi\PQ}\ar[r] & [\unl_{\bP}/\unL_{\bP}]\ar[r]^{\chi_{\bP}} & \frc_{\bP}\ar[d]\\
[\unl_{\bQ}/\unL_{\bQ}]\ar[rr]^{\chi_{\bQ}} && \frc_{\bQ}}
\end{equation*}
is Cartesian over $\frc^{\rs}_{\bQ}$. Therefore by Lem. \ref{l:CartPQ}, the morphism $\nu\PQ$ is also an isomorphism over $(\calA\times X)^{\rs}$, hence birational. Since $q\PQ:\tcA_{\bP}\to\tcA_{\bQ}$ is finite and $\tforg\PQ$ is small by (1), we conclude that $\nu\PQ$ is also small. This proves (2).
\end{proof}

\subsection{Examples in classical groups}
In this subsection, we describe the fibers of parahoric Hitchin fibrations for classical groups of type A, B and C.

\begin{exam} Let $G=\SL(n)$. According to Example \ref{ex:BunA}, a standard parahoric subgroup $\bP\subset\SL(n,F)$ corresponds to a sequence of integers
\begin{equation*}
\undi=(0\leq i_0<\cdots<i_m<n), m\geq0.
\end{equation*}
The Hitchin base is
\begin{equation*}
\AHit=\bigoplus_{i=2}^n\cohog{0}{X,\calO_X(iD)}.
\end{equation*}
For $a=(a_2,\cdots,a_n)\in\Ah(k)$ (where $a_i\in\cohog{0}{X,\calO_X(iD)}$), define the spectral curve $Y_a$ as in \cite[Example 3.1.10]{GSI}. Fix a point $x\in X$. Then the parahoric Hitchin fiber $\calM_{\bP,a,x}$ classifies the data
\begin{equation*}
(\calF_{i_0}\supset\calF_{i_1}\supset\cdots\supset\calF_{i_m}\supset\calF_{i_0}(-x),\delta)
\end{equation*}
where $\calF_{i_j}\in\overline{\stPic}(Y_a)$ such that $\calF_{i_0}/\calF_{i_j}$ has length $i_0-i_j$ for $j=0,1,\cdots,m$, and $\delta$ is an isomorphism $\det(p_{a,*}\calF_{i_0})\cong\calO_X(-i_0x)$.
\end{exam}

\begin{exam}\label{ex:MBC} Let $G=\SO(2n+1)$ (resp. $G=\Sp(2n)$). According to Example \ref{ex:BunBC}, a standard parahoric subgroup $\bP$ corresponds to a sequence of integers
\begin{equation*}
\undi=(0\leq i_0<\cdots<i_m\leq n), m\geq0.
\end{equation*}
The Hitchin base is
\begin{equation*}
\AHit=\bigoplus_{i=1}^n\cohog{0}{X,\calO_X(2iD)}.
\end{equation*}
For $a=(a_1,\cdots,a_n)\in\Ah(k)$ (where $a_i\in\cohog{0}{X,\calO_X(2iD)}$), we have the spectral curve $Y_a$ in the total space of $\calO_X(D)$ defined by the equation
\begin{equation*}
t\sum_{i=0}^{n}a_it^{2(n-i)}=0;\hspace{1cm}\left(\textup{resp. }\sum_{i=0}^{n}a_it^{2(n-i)}=0\right)
\end{equation*}
where $a_0=1$. The curve $Y_a$ is equipped with the involution $\tau$ sending $t$ to $-t$. Fix a point $x\in X$. Then the parahoric Hitchin fiber $\calM_{\bP,a,x}$ classifies the data
\begin{equation*}
(\calF_{i_0}\supset\cdots\supset\calF_{i_m}\supset\calF^{\bot}_{i_m}(-x)\supset\cdots\supset\calF^{\bot}_{i_0}(-x)\supset\calF_{i_0}(-x),\sigma)
\end{equation*}
where
\begin{itemize}
\item $\calF_{i_j}\in\overline{\stPic}(Y_a)$ such that $\calF_{i_0}/\calF_{i_j}$ has length $i_0-i_j$ for $j=0,1,\cdots,m$;
\item $\sigma:\tau^*\calF_{i_0}\to\calF_{i_0}^\vee$ is a map of coherent sheaves on $Y_a$ such that $\tau^*\sigma=\sigma^\vee$ (resp. $\tau^*\sigma=-\sigma^\vee$). Here $(-)^\vee$ means the relative Grothendieck-Serre duality for coherent sheaves on $Y_a$ with respect to the $X$;
\item For $j=0,1,\cdots,m$, define $\calF_{i_j}^\bot:=(\sigma(\tau^*\calF_{i_j}))^\vee$, which naturally contains $\calF_{i_0}$;
\item $\coker(\sigma)$ has length $i_0$.
\end{itemize}
\end{exam}

% daha action

\section{The graded double affine Hecke algebra action}\label{s:DAHA}
In this section, we enrich the affine Weyl group action constructed in \cite[Sec. 4]{GSI} into an action of the graded double affine Hecke algebra (DAHA). We also generalize the graded DAHA action to the case of parahoric Hitchin stacks. To save notations, {\em we assume that $G$ is almost simple throughout this section}.

%------------------------------------------

% Ch: double affine

%------------------------------------------

% eqn clear

\subsection{The Kac-Moody group}\label{ss:KM}
In this subsection, we recall the construction of the Kac-Moody group associated to the loop group $G((t))$.

\subsubsection{The determinant line bundle} For any $k$-algebra $R$, we have an additive functor 
\begin{equation*}
\det:D_{perf}(R)\to\Pic(R) 
\end{equation*}
Here $D_{perf}(R)$ is the derived category of perfect complexes of $R$-modules and $\Pic(R)$ is the Picard category of invertible $R$-modules. The functor $\det$ sends a projective $R$-module $M$ of finite rank $m$ the invertible $R$-module $\wedge^{m}M$.

We may define a line bundle $\Lcan$ on $G((t))$, which is pulled back from $\Grass_G=G((t))/G[[t]]$. For any $R[[t]]$-submodule $\Xi$ of $\frg\otimes_{k}R((t))$ which is commensurable with the standard $R[[t]]$-submodule $\Xi_0:=\frg\otimes_kR[[t]]$ (i.e., $t^N\Xi_0\subset\Xi\subset t^{-N}\Xi_0$ for some $N\in\ZZ_{\geq0}$ and $t^{-N}\Xi_0/\Xi$ and $\Xi/t^N\Xi_0$ are both projective $R$-modules), define the {\em relative determinant line} of $\Xi$ with respect to $\Xi_0$ to be:
\begin{equation*}
\det(\Xi:\Xi_0)=(\det(\Xi/\Xi\cap\Xi_0))\otimes_R(\det(\Xi_0/\Xi\cap\Xi_0))^{\otimes-1}.
\end{equation*}
For any $g\in G(R((t)))$, consider its action on $\frg\otimes_{k}R((t))$ by the adjoint representation. The functor $\Lcan$ then sends $g$ to the invertible $R$-module $\det(\Ad(g)\Xi_0:\Xi_0)$. Since $\Ad(g)\Xi_0$ only depends on the image of $g$ in $\Grass_G(R)$, the line bundle $\Lcan$ is descends to $\Grass_G$.

Let $\hatgloop=\rho_{\Lcan}\to G((t))$ be the total space of the $\GG_m$-torsor associated to the line bundle $\Lcan$. The set $\hatgloop(R)$ consists of pairs $(g,\gamma)$ where $g\in G(R((t)))$ and $\gamma$ is an $R$-linear isomorphism $R\isom\det(\Ad(g)\Xi_0:\Xi_0)$. There is a natural group structure on $\hatgloop$: for $(g_1,\gamma_1)$ and $(g_2,\gamma_2)\in\hatgloop$, their product $(g_1,\gamma_1)\cdot(g_2,\gamma_2)$ is $(g_1g_2,\gamma)$, where $\gamma$ is the isomorphism
\begin{eqnarray*}
\gamma_1\otimes\Ad(g_1)(\gamma_2):R\otimes_RR&\isom&\det(\Ad(g_1)\Xi_0:\Xi_0)\otimes_R\det(\Ad(g_1g_2)\Xi_0:\Ad(g_1)\Xi_0)\\
&=&\det(\Ad(g_1g_2)\Xi_0:\Xi_0).
\end{eqnarray*}
The group $\hatgloop$ is in fact a central extension
\begin{equation}\label{eq:cenext}
1\to\gcen\to\hatgloop\to G((t))\to1.
\end{equation}
Here we use $\gcen$ to denote the one-dimensional central torus of $\hatgloop$, which can be identified as the fiber of $\hatgloop$ over the identity element $1\in G((t))$. When $g\in G[[t]]$, we have $\Ad(g)\Xi_0=\Xi_0$, hence a canonical trivialization of $\det(\Ad(g)\Xi_0:\Xi_0)$. This gives a canonical splitting of the central extension (\ref{eq:cenext}) over the subgroup $G[[t]]\subset G((t))$.

\subsubsection{The completed Kac-Moody group} From the construction it is clear that the action of $\Aut_{\calO}$ on $G((t))$ lifts to an action on $\hatgloop$, hence we can form the semi-direct product
\begin{equation}\label{eq:calG}
\calG:=\hatgloop\rtimes\Aut_{\calO}.
\end{equation}
We call this object the {\em (complete) Kac-Moody group} associated to the loop group $G((t))$.
 
Let $\bI^u\subset\bI$ be the unipotent radical and $G^u_{\bI}\subset G_{\bI}$ be the corresponding pro-unipotent radical. Let $\Aut^u_{\calO}\subset\Aut_{\calO}$ be the pro-unipotent radical. Consider the subgroups
\begin{eqnarray*}
\calG_{\bI}&:=&\gcen\times G_{\bI}\rtimes\Aut_{\calO}\subset\calG;\\
\calG^u_{\bI}&:=&G^u_{\bI}\rtimes\Aut^u_{\calO}\subset\calG_{\bI}.
\end{eqnarray*}
We define the {\em universal Cartan torus} for the Kac-Moody group $\calG$ to be
\begin{equation}\label{eq:deftilT}
\tilT:=\calG_{\bI}/\calG^u_{\bI}=\gcen\times T\times\grot.
\end{equation}
We will denote the canonical generators of $\xcoch(\gcen),\xch(\gcen),\xcoch(\grot)$ and $\xch(\grot)$ by $K_{\can},\Lambda_{\can},d$ and $\delta$.  Let
\begin{equation*}
\langle\cdot,\cdot\rangle:\xch(\tilT)\times\xcoch(\tilT)\to\ZZ
\end{equation*}
be the natural pairing.

Let $(\cdot|\cdot)_{\can}$ be the Killing form on $\xcoch(T)$:
\begin{equation}\label{eq:cankill}
(x|y)_{\can}:=\sum_{\alpha\in\Phi}\jiao{\alpha,x}\jiao{\alpha,y}.
\end{equation}
where $\Phi\subset\xch(T)$ is the set of roots of $G$. Let $\theta\in\Phi$ be highest root and $\theta^\vee\in\Phi^\vee$ be the corresponding coroot. Let $\rho$ be half of the sum of the positive roots in $\Phi$. Let $h^\vee$ be the dual Coxeter number of $\frg$, which is one plus the sum of coefficients of $\theta^\vee$ written as a linear combination of simple coroots. We have the following fact:

\begin{lemma}\label{l:dCox}
$\dfrac{1}{2}(\theta^\vee|\theta^\vee)_{\can}=2(\jiao{\rho,\theta^\vee}+1)=2h^\vee.$
\end{lemma}
\begin{proof}
Since $\theta$ is the highest root, for any positive root $\alpha\neq\theta$, we have $\jiao{\alpha,\theta^\vee}=0$ or $1$ (see \cite[Chap VI, 1.8, Prop. 25(iv)]{Bour}). Hence $\jiao{\alpha,\theta^\vee}^2=\jiao{\alpha,\theta^\vee}$ for $\alpha\in\Phi^+-\{\theta\}$. Therefore
\begin{eqnarray*}
\dfrac{1}{2}(\theta^\vee|\theta^\vee)_{\can}=\sum_{\alpha\in\Phi^+}\jiao{\alpha,\theta^\vee}^2&=&\jiao{\theta,\theta^\vee}^2+\sum_{\alpha\in\Phi^+-\{\theta\}}\jiao{\alpha,\theta^\vee}\\
&=&4+\jiao{2\rho-\theta,\theta^\vee}=2(\jiao{\rho,\theta^\vee}+1).
\end{eqnarray*}
Since $\jiao{\rho,\alpha^\vee_i}=1$ for every simple coroot $\alpha^\vee_i\in\Phi$, we get
\begin{equation*}
\jiao{\rho,\theta^\vee}+1=h^\vee.
\qedhere
\end{equation*}
\end{proof}

\subsubsection{The $\tilW$-action on $\tilT$} For any section $\iota$ of the quotient $B\to T$, we can consider the normalizer $\calN$ of $\gcen\times\iota(T)\times\grot$ in $\hatgloop\rtimes\grot$, and we have a canonical isomorphism $\calN/(\gcen\times\iota(T)[[t]]\times\grot)\isom\tilW$. The conjugation action of $\calN$ on $\gcen\times\iota(T)\times\grot$ induces an action of $\tilW$ on $\tilT$, which is independent of the choice of the section $\iota$. Therefore, we get canonical actions of $\tilW$ on $\xcoch(\tilT)$ and $\xch(\tilT)$, denoted by $\eta\mapsto\leftexp{\tilw}{\eta}$ and $\xi\mapsto\leftexp{\tilw}{\xi}$. The natural pairing $\jiao{\cdot,\cdot}$ is invariant under $\tilW$: i.e.,  $\jiao{\xi,\eta}=\jiao{\leftexp{\tilw}{\xi},\leftexp{\tilw}{\eta}}$.

\begin{lemma}\label{l:Kacact}
The actions of $\tilW$ on $\xcoch(\tilT)$ and $\xch(\tilT)$ are given by:
\begin{enumerate}
\item $w\in W$ fixes $K_{\can},d,\Lambda_{\can}$ and $\delta$, and acts in the usual way on $\xcoch(T)$ and $\xch(T)$;
\item $\lambda\in\xcoch(T)$ acts on $\eta\in\xcoch(\tilT)$ and $\xi\in\xch(\tilT)$ by
\begin{eqnarray*}
\leftexp{\lambda}{\eta}=\eta-\jiao{\delta,\eta}\lambda+\left((\eta|\lambda)_{\can}-\dfrac{1}{2}(\lambda|\lambda)_{\can}\jiao{\delta,\eta}\right)K_{\can};\\
\leftexp{\lambda}{\xi}=\xi-\jiao{\xi,K_{\can}}\lambda^*+\left(\jiao{\xi,\lambda}-\dfrac{1}{2}(\lambda|\lambda)_{\can}\jiao{\xi,K_{\can}}\right)\delta.
\end{eqnarray*}
here $\lambda\mapsto\lambda^*$ is the isomorphism $\xcoch(T)_{\QQ}\isom\xch(T)_{\QQ}$ induced by the form $(\cdot|\cdot)_{\can}$. 
\end{enumerate}
\end{lemma}
\begin{proof}
(1) is clear from the fact that $W\subset G(k)\subset G(\calO_F)$.

To check (2), we choose a maximal torus in $B$ and call it $T$. We write any element in $\tilT$ as $(c,x,\sigma)$ for $c\in\gcen,x\in T$ and $\sigma\in\grot$. We need to check that
\begin{eqnarray}\notag
\Ad(t^\lambda)(c,1,1)=(c,1,1);\\
\label{eq:commt}\Ad(t^\lambda)(1,x,1)=(\prod_{\alpha\in R}x^{\jiao{\alpha,\lambda}\alpha},x,1);\\
\label{eq:commrot}\Ad(t^\lambda)(1,1,\sigma)=(\sigma^{-\frac{1}{2}(\lambda|\lambda)_{\can}},\sigma^{-\lambda},\sigma)
\end{eqnarray}
Here, $x^\alpha$ is the image of $x$ under $\alpha:T\to\GG_m$; similarly $\sigma^{-\lambda}$ is the image of $\sigma$ under $-\lambda:\GG_m\to T$.

To verify the $\gcen$-coordinates in (\ref{eq:commt}) and (\ref{eq:commrot}), notice that the $\gcen$-coordinate of $\Ad(t^\lambda)(1,x,\sigma)$ is the same as the (scalar) action of $(x,\sigma)\in T\times\grot$ on the line $\det(\Ad(t^{-\lambda})\Xi_0:\Xi_0)$, via the adjoint representation. In terms of the root space decomposition $\frg=\frt\oplus(\oplus_{\alpha\in\Phi}\frg_\alpha)$, we have
\begin{equation*}
\det(\Ad(t^{-\lambda})\Xi_0:\Xi_0)=\left(\bigotimes_{\jiao{\alpha,\lambda}>0}\bigotimes_{i=-\jiao{\alpha,\lambda}}^{-1}t^i\frg_\alpha\right)\otimes\left(\bigotimes_{\jiao{\alpha,\lambda}<0}\bigotimes_{i=0}^{-\jiao{\alpha,\lambda}-1}t^i\frg_{\alpha}\right)^{\otimes-1}
\end{equation*}
Therefore, as a $T\times\grot$-module, $\det(\Ad(t^{-\lambda})\Xi_0:\Xi_0)$ has weight
\begin{eqnarray*}
\left(\sum_{\alpha\in\Phi}\jiao{\alpha,\lambda}\alpha,-\sum_{\alpha\in\Phi}\frac{1}{2}\jiao{\alpha,\lambda}^2-\frac{1}{2}\jiao{\alpha,\lambda}\right)=\left(\sum_{\alpha\in\Phi}\jiao{\alpha,\lambda}\alpha,-\frac{1}{2}(\lambda|\lambda)_{\can}\right).
\qedhere
\end{eqnarray*}
\end{proof}

\begin{remark}\label{rm:norm} Let
\begin{equation*}
K:=2h^\vee K_{\can};\hspace{1cm}\Lambda_0:=\dfrac{1}{2h^\vee}\Lambda_{\can}.
\end{equation*}
We see from Lem. \ref{l:Kacact} that our definitions of $K,\Lambda_0,d$ and $\delta$ are consistent (up to changing $\lambda$ to $-\lambda$) with the notation for Kac-Moody algebras in \cite[6.5]{Kac}. The simple roots of the complete Kac-Moody group $\calG$ are $\{\alpha_0=\delta-\theta,\alpha_1,\cdots,\alpha_n\}\subset\xch(T\times\grot)$; the simple coroots are $\{\alpha_0^\vee=K-\theta^\vee,\alpha_1^\vee,\cdots,\alpha_n^\vee\}\subset\xcoch(\gcen\times T)$.
\end{remark}

\subsection{Line bundles on $\Bunpar_G$}\label{ss:linebd}
Let $\canb$ be the canonical bundle of $\Bun_G$. Since the tangent complex at a point $\calE\in\Bun_G(R)$ is $\bR\Gamma(\XR,\Ad(\calE))[1]$, the value of the canonical bundle $\canb$ at the point $\calE$ is the invertible $R$-module $\det\bR\Gamma(\XR,\Ad(\calE))$. Let $\hatBun_{\infty}\to\tilBun_{\infty}$ be the total space of the $\GG_m$-torsor associated to the pull-back of $\canb$. More concretely, for any $k$-algebra $R$, $\hatBun_{\infty}(R)$ classifies tuples $(x,\alpha,\calE,\tau_x,\epsilon)$ where $(x,\alpha,\calE,\tau_x)\in\tilBun_{\infty}(R)$ and $\epsilon$ is an $R$-linear isomorphism $R\isom\det\bR\Gamma(\XR,\Ad(\calE))$.

\begin{cons}\label{cons:hataction}
There is a natural action of $\calG$ on $\hatBun_{\infty}$, lifting the action of $G((t))\rtimes\Aut_{\calO}$ on $\tilBun_{\infty}$ in Construction \ref{cons:gloopact}. In fact, for $(x,\alpha,\calE,\tau_x)\in\tilBun_{\infty}(R)$ and $g\in G(R((t)))$, the $G$-torsor $\calE^g$ is obtained by gluing the trivial $G$-torsor on $\disk_x\cong\Spec R[[t]]$ (using $\alpha$) with $\calE|_{\XR-\Gamma(x)}$ via the identification $\tau_x\circ g$. Hence $\Ad(\calE^g)$ is obtained by gluing $\frg(\hatO_x)\cong\frg\otimes_kR[[t]]=\Xi_0$ with $\Ad(\calE)|_{\XR-\Gamma(x)}$ via the identification $\Ad(\tau_x)\circ\Ad(g)$. In other words, $\Ad(\calE^g)$ is obtained by gluing $\Ad(g)\Xi_0$ with $\Ad(\calE)|_{\XR-\Gamma(x)}$ via $\Ad(\tau_x)$. Thus we have a canonical isomorphism of invertible $R$-modules
\begin{equation*}
\left(\det\bR\Gamma(\XR,\Ad(\calE^g))\right)\otimes_R\left(\det\bR\Gamma(\XR,\Ad(\calE))\right)^{\otimes-1}\cong\det(\Ad(g)\Xi_0:\Xi_0).
\end{equation*}
Therefore, for trivializations $\epsilon:R\isom\det\bR\Gamma(\XR,\Ad(\calE))$ and $\gamma:R\isom\det(\Ad(g)\Xi_0:\Xi_0)$, $\epsilon\otimes\gamma$ defines a trivialization of $\det\bR\Gamma(\XR,\Ad(\calE^g))$. We then define the action of $\hatg=(g,\gamma,\sigma)\in(\hatgloop\rtimes\Aut_{\calO})(R)=\calG(R)$ on $(x,\alpha,\calE,\tau_x,\epsilon)\in\hatBun_{\infty}(R)$ by
\begin{equation*}
R_{\hatg}(x,\alpha,\calE,\tau_x,\epsilon)=(x,\alpha\circ\sigma,\calE^g,\tau_x^g,\epsilon\otimes\gamma).
\end{equation*}
\end{cons}

\begin{cons}\label{cons:linebd}
We define a natural $\tilT$-torsor $\calL^{\tilT}$ on $\Bunpar_G$, hence line bundles $\calL(\xi)$ for $\xi\in\xch(\tilT)$. Consider the quotient $\calL^{\tilT}=\hatBun_{\infty}/\calG^u_{\bI}$ (as a fpqc sheaf). The right translation of $\calG_{\bI}$ on $\hatBun_{\infty}$ descends to a right action of $\tilT=\calG_{\bI}/\calG^u_{\bI}$ on $\calL^{\tilT}$, and realizes the natural projection $\calL^{\tilT}\to\Bunpar_G$ as a $\tilT$-torsor. For each character $\xi\in\xch(\tilT)$, we define $\calL(\xi)$ to be the line bundle on $\Bunpar_G$ associated to $\calL^{\tilT}$ and the character $\xi$.
\end{cons}

We can easily identify the line bundles $\calL(\xi)$ for various $\xi\in\xch(\tilT)$:

\begin{lemma}\label{l:easyline}
\begin{enumerate}
\item []
\item $\calL(\Lambda_{\can})$ is the pull-back of $\canb$ via the forgetful morphism $\Bunpar_G\to\Bun_G$;
\item For $\xi\in\xch(T)$, the value of the line bundle $\calL(\xi)$ at a point $(x,\calE,\calE^B_x)\in\Bunpar_G(R)$ is the invertible $R$-module associated to the $B$-torsor $\calE^B_x$ over $\Gamma(x)\cong\Spec R$ and the character $B\to T\xrightarrow{\xi}\GG_m$;
\item $\calL(\delta)$ is isomorphic to the pull-back of $\omega_X$ via the morphism $\Bunpar_G\to X$ (cf. Lem. \ref{c:canX}).
\end{enumerate}
\end{lemma}

% DAHA

\subsection{The graded double affine Hecke algebra and its action}\label{ss:daha}

\subsubsection{Comparison of $\tilW$ and $\Wa$} Recall that $\tilW=\xcoch(T)\rtimes W$ is the extended affine Weyl group associated to $G$ and $\Wa=\ZZ\Phi^\vee\rtimes W$ is the affine Weyl group, where $\ZZ\Phi^\vee$ is the coroot lattice. It is well-known that $\Wa$ is a Coxeter group with simple reflections $\Sigma_{\aff}=\{s_0,s_1,\cdots,s_n\}$, where $s_1,\cdots,s_n$ are simple reflections of the finite Weyl group $W$ corresponding to our choice of the Borel $B\subset G$. We have an exact sequence
\begin{equation}\label{eq:extweyl}
1\to\Wa\to\tilW\to\Omega\to1.
\end{equation}
where $\Omega=\xcoch(T)/\ZZ\Phi^\vee$. Let $\tilI$ be the normalizer of $\bI$ in $G(F)$, which defines an extension $G_{\tilI}$ of $G_{\bI}$ by $\Omega_{\bI}=\tilI/\bI$. Let $\hatG_{\tilI}$ be the preimage of $G_{\tilI}$ in $\hatgloop$ and $\calG_{\tilI}:=\hatG_{\tilI}\rtimes\Aut_{\calO}$. Then $\calG_{\tilI}$ normalizes $\calG_{\bI}$ and the conjugation action of $\calG_{\tilI}$ on $\calG_{\bI}$, after passing to the quotient $\calG_{\bI}\twoheadrightarrow\tilT$, induces an action of $\Omega_{\bI}$ on $\tilT$. It is easy to see that this action is a sub-action of the $\tilW$-action on $\tilT$, hence we can naturally view $\Omega_{\bI}$ as a subgroup of $\tilW$. It is easy to verify that the composition $\Omega_{\bI}\hookrightarrow\tilW\twoheadrightarrow\Omega$ is an isomorphism. Therefore we can write $\tilW$ as a semi-direct product $\tilW=\Wa\rtimes\Omega_{\bI}$.

\begin{defn}\label{def:daha}
The {\em graded double affine Hecke algebra} (or {\em graded DAHA} for short) is an evenly graded $\Ql$-algebra $\HH$ which, as a vector space, is the tensor product
\begin{equation*}
\HH=\Ql[\tilW]\otimes\Sym_{\Ql}(\xch(\tilT)_{\Ql})\otimes\Ql[u].
\end{equation*}
Here $\Ql[\tilW]$ is the group ring of $\tilW$, and $\Ql[u]$ is a polynomial algebra in the indeterminate $u$. The grading on $\HH$ is given by
\begin{itemize}
\item $\deg(\tilw)=0$ for $\tilw\in\tilW$;
\item $\deg(u)=\deg(\xi)=2$ for $\xi\in\xch(\tilT)$.
\end{itemize}
The algebra structure on $\HH$ is determined by
\begin{enumerate}
\item\label{eq:WT} $\Ql[\tilW]$, $\Sym_{\Ql}(\xch(\tilT)_{\Ql})$ and $\Ql[u]$ are subalgebras of $\HH$;
\item\label{eq:ucen} $u$ is in the center of $\HH$;
\item\label{eq:commrefl} For any simple reflection $s_i\in\Sigma_{\aff}$ (corresponding to a simple root $\alpha_i$) and $\xi\in\xch(\tilT)$, we have
\begin{equation*}
s_i\xi-\leftexp{s_i}{\xi}s_i=\langle\xi,\alpha_i^\vee\rangle u;
\end{equation*}
\item\label{eq:commomega} For any $\omega\in\Omega_{\bI}$ and $\xi\in\xch(\tilT)$, we have 
\begin{equation*}
\omega\xi=\leftexp{\omega}{\xi}\omega.
\end{equation*}
\end{enumerate}
\end{defn}

\begin{remark} When $\tilW$ is replaced by the finite Weyl group $W$, and $\tilT$ is replaced by $T$, the corresponding algebra is the equal-parameter case of the ``graded affine Hecke algebras'' considered by Lusztig in \cite{L88}.
\end{remark}

Our main goal in this section is to construct an action of $\HH$ on the parabolic Hitchin complex $\fQl\in D^b_c(\calA\times X)$.
 
\begin{cons}\label{cons:action} We define the action of the generators of $\HH$ on $\fQl$.
\begin{itemize}
\item The action of $\tilW$ has been constructed in \cite[Th. 4.4.3]{GSI}.

\item The action of $u$. The Chern class of the line bundle $\calO_X(D)$ (after pulling back to $\calA\times X$) gives a morphism in $D^b_c(\calA\times X)$:
\begin{equation*}
c_1(D):\const{\calA\times X}\to\const{\calA\times X}[2](1).
\end{equation*}
In general, for any object $K\in D^b_c(\calA\times X)$, the cup product with $c_1(D)$ defines a map $\cup c_1(D):K\to K[2](1)$. In particular, for $K=\fQl$, we get the action of $u$:
\begin{equation*}
u=\cup c_1(D):\fQl\to\fQl[2](1).
\end{equation*}

\item The action of $\xch(\tilT)$. Recall from Construction \ref{cons:linebd} that we have a $\tilT$-torsor $\calL^{\tilT}$ over $\Bunpar_G$, and the associated line bundle $\calL(\xi)$ for $\xi\in\xch(\tilT)$. We also use $\calL(\xi)$ to denote its pull back to $\Mpar$. The Chern class of $\calL(\xi)$ gives a map:
\begin{equation*}
c_1(\calL(\xi)):\const{\Mpar}\to\const{\Mpar}[2](1).
\end{equation*}
We define the action of $\xi$ on $\fQl$ to be
\begin{equation*}
\xi=\fpar_*(c_1(\calL(\xi))):\fQl\to\fQl[2](1).
\end{equation*}
By Lem. \ref{l:easyline}, the action of $\Lambda_{\can}\in\xch(\gcen)$ is given by the cup product with (the pull-back of) $c_1(\canb)$; the action of $\delta\in\xch(\grot)$ is given by the cup product with (the pull-back of) $c_1(\omega_X)$.
\end{itemize}
\end{cons}

\begin{theorem}\label{th:daction}
The actions of $\tilW,u$ and $\xch(\tilT)$ on $\fQl$ given in Construction \ref{cons:action} extends to an action of $\HH$ on $\fQl$. More precisely, we have a graded algebra homomorphism
\begin{equation*}
\HH\to\bigoplus_{i\in\ZZ}\End^{2i}_{\calA\times X}(\fQl)(i)
\end{equation*}
such that the image of the elements in $\tilW\cup\{u\}\cup\xch(\tilT)\subset\HH$ are the same as the ones given in Construction \ref{cons:action}.
\end{theorem}

If we fix a point $(a,x)\in(\calA\times X)(k)$, we can specialize the above theorem to the action of $\HH$ on the stalk of $\fQl$ at $(a,x)$, i.e., $\cohog{*}{\Mpar_{a,x}}$.

\begin{cor}
For $(a,x)\in(\calA\times X)(k)$, Construction \ref{cons:action} gives an action of $\HH/(\delta,u)$ on $\cohog{*}{\Mpar_{a,x}}$. In other words, the actions of $\bxi\in\xch(\gcen\times T)$ and $\tilw\in\tilW$ satisfy the following simple relation:
\begin{equation*}
\tilw\bxi=\leftexp{\tilw}{\bxi}\tilw.
\end{equation*}
Here $\bxi\mapsto\leftexp{\tilw}{\bxi}$ is the action of $\tilw$ on $\xch(\gcen\times T)=\xch(\tilT)/\xch(\grot)$.
\end{cor}
\begin{proof}
Since the restrictions of $\calO_X(D)$ and $\omega_X$ to $\Mpar_{a,x}$ are trivial, the actions of $\delta$ and $u$ on $\cohog{*}{\Mpar_{a,x}}$ are zero.
\end{proof}

Sec. \ref{ss:conncomp} through Sec. \ref{ss:pfdaha} are devoted to the proof of Th. \ref{th:daction}. We will check that the four conditions in Def. \ref{def:daha} hold for the actions defined in Construction \ref{cons:action}. 

The condition (\ref{eq:WT}) in Def. \ref{def:daha} is trivial from construction. The condition (\ref{eq:ucen}) is also easy to check. In fact, since the $\tilW$-action is constructed from self-correspondences of $\Mpar$ over $\calA\times X$, it commutes with the cup product with any class in $\cohog{*}{\calA\times X}$. In particular, the $\tilW$-action commutes with the $u$-action and the $\xch(\grot)$-action. On the other hand, the action of $\xi\in\xch(\tilT)$ is defined as the cup product with the Chern class $c_1(\calL(\xi))\in\cohog{2}{\Mpar}(1)$, which certainly commutes with the cup product with the pull-back of $c_1(D)\in\cohog{2}{X}(1)$. Therefore the $\xch(\tilT)$-action also commutes with the $u$-action. This verifies the condition (\ref{eq:ucen}) in Def. \ref{def:daha}.

We will verify the condition (\ref{eq:commomega}) in Sec. \ref{ss:conncomp} and the condition (\ref{eq:commrefl}) in Sec. \ref{ss:simplerefl} and Sec. \ref{ss:pfdaha}.

\subsection{Remarks on Hecke correspondences}\label{ss:rmHecke}
In this subsection, we study the relation between the reduced Hecke correspondence $\calH_{\tilw}$ introduced in \cite[Def. 4.3.9]{GSI} and the stratum $\Hecke^{\Bun}_{\tilw}$ in the Hecke correspondence $\Hecke^{\Bun}$ (see \cite[Section 4.2]{GSI}), for the same subscript $\tilw\in\tilW$.

\subsubsection{Rewriting the Hecke correspondences} Let us first write $\Hecke^{\Bun}_{\tilw}$ more precisely using the identification $\Bunpar_G=\tilBun_{\infty}/G_{\bI}\rtimes\Aut_{\calO}$ in Sec. \ref{ss:bunpar}. We have
\begin{equation*}
\Hecke^{\Bun}=\tilBun_{\infty}\twtimes{G_{\bI}\rtimes\Aut_{\calO}}\left(G((t))\rtimes\Aut_{\calO}/G_{\bI}\rtimes\Aut_{\calO}\right)
\end{equation*}
with the two projections given by
\begin{eqnarray*}
\overleftarrow{b}(\xi,\widetilde{g})&=&\xi\modo G_{\bI}\rtimes\Aut_{\calO};\\
\overrightarrow{b}(\xi,\widetilde{g})&=&R_{\widetilde{g}}(\xi)\modo G_{\bI}\rtimes\Aut_{\calO}.
\end{eqnarray*}
for $\xi\in\tilBun_{\infty}(S)$ and $\widetilde{g}\in (G((t))\rtimes\Aut_{\calO})(S)/(G_{\bI}\rtimes\Aut_{\calO})(S)$. The Bruhat decomposition gives
\begin{eqnarray*}
G((t))&=&\bigsqcup_{\tilw\in\tilW}G_{\bI}\tilw G_{\bI};\\
G((t))\rtimes\Aut_{\calO}&=&\bigsqcup_{\tilw\in\tilW}(G_{\bI}\rtimes\Aut_{\calO})\tilw(G_{\bI}\rtimes\Aut_{\calO}).
\end{eqnarray*}
Hence we can write
\begin{equation}\label{eq:Heckeloop}
\Hecke^{\Bun}_{\tilw}=\tilBun_{\infty}\twtimes{G_{\bI}\rtimes\Aut_{\calO}}\left((G_{\bI}\rtimes\Aut_{\calO})\tilw(G_{\bI}\rtimes\Aut_{\calO})/G_{\bI}\rtimes\Aut_{\calO}\right).
\end{equation}

Recall that we have a morphism (see \cite[Diagram (4.3)]{GSI})
\begin{equation}\label{eq:betaHH}
\beta:\Heckep\to\Hecke^{\Bun}.
\end{equation}

\begin{lemma}\label{l:hhecke}
\begin{enumerate}
\item []
\item The image of $\calH_{\tilw}^{\rs}$ in $\Hecke^{\Bun}$ is contained in $\Hecke^{\Bun}_{\tilw}$;
\item The image of $\calH_{\tilw}$ in $\Hecke^{\Bun}$ is contained in $\Hecke^{\Bun}_{\leq\tilw}$.
\end{enumerate}
\end{lemma}
\begin{proof}
Since $\calH_{\tilw}$ is the closure of $\calH^{\rs}_{\tilw}$ by definition, (2) follows from (1). To check (1), it suffices to check the geometric points (or even the $k$-points). Fix $(a,x)\in(\calA\times X)^{\rs}(k)$. Using the local-global product formula \cite[Prop. 3.3.3]{GSI}, $\Mpar_{a,x}$ is homeomorphic to
\begin{equation}\label{eq:prodx}
\calP_a\twtimes{P^{\red}_x(J_a)\times P'}(M^{\parab,\red}_x(\gamma)\times M'),
\end{equation}
where $M'$ and $P'$ are the products of local terms over $y\in X-\{x\}$, and $\gamma\in\frg(\hatO_x)$ lifting $a(x)\in\frc(\hatO_x)$. Since $a(x)$ has regular semisimple reduction in $\frc$, we can conjugate $\gamma$ by $G(\hatO_x)$ so that $\gamma\in\frt(\hatO_x)$ (here $\frt\subset\frb$ is a Cartan subalgebra). The choice of $\gamma$ gives a point $\tilx\in q_a^{-1}(x)$. In particular, we can use $\tilx$ to get an isomorphism $P_x(J_a)=T(F_x)/T(\hatO_x)$.

Fix a uniformizing parameter $t\in\hatO_x$. Now the reduced structure of $M^{\parab}_x(\gamma)\subset\Flag_{G,x}$ consists of the $T$-fixed points $\tilw\bI_x/\bI_x=t^{\lambda_1}w_1\bI_x/\bI_x$ for $\tilw_1=(\lambda_1,w_1)\in\tilW$. Under the isomorphism $M^{\parab}_{x}(\gamma)=M^{\Hit}_x(\gamma)\times q_a^{-1}(x)$, $\tilw\bI_x/\bI_x$ corresponds to the pair $(t^{\lambda_1}G(\hatO_x)/G(\hatO_x),w_1^{-1}\tilx)$.

We claim that the action of $\tilw=(\lambda,w)\in\tilW$ on $\Mpar_{a,x}$, under the product formula (\ref{eq:prodx}), is trivial on $M'$ and sends $\tilw_1\bI_x/\bI_x\in M^{\parab}_{x}(\gamma)$ to $\tilw_1\tilw\bI_x/\bI_x$. In fact, using the definition of the right $\tilW$-action in \cite[Cor. 4.3.8]{GSI}, we have
\begin{eqnarray*}
\tilw_1\bI_x/\bI_x\cdot\tilw
&=&(t^{\lambda_1}G(\hatO_x)/G(\hatO_x),w_1^{-1}\tilx)\cdot(\lambda,w)\\
&=&(\tils_{\lambda}(a,w_1^{-1}\tilx)t^{\lambda_1}G(\hatO_x)/G(\hatO_x),w^{-1}w_1^{-1}\tilx)\\
&=&(s_{w_1\lambda}(a,\tilx)t^{\lambda_1}G(\hatO_x)/G(\hatO_x),w^{-1}w_1^{-1}\tilx)\\
&=&(t^{w_1\lambda+\lambda_1}G(\hatO_x)/G(\hatO_x),(w_1w)^{-1}\tilx)\\
&=&\tilw_1\tilw\bI_x/\bI_x.
\end{eqnarray*}

Clearly, the pair $(\tilw_1\bI_x/\bI_x,\tilw_1\tilw\bI_x/\bI_x)\in\Flag_{G,x}\times\Flag_{G,x}$ is in relative position $\tilw$, hence the image of $\calH^{\rs}_{\tilw}$ in $\Hecke^{\Bun}$ is contained in $\Hecke^{\Bun}_{\tilw}$.
\end{proof}

% components

\subsection{Connected components of $\Mpar$}\label{ss:conncomp}
In this subsection, we check the condition (\ref{eq:commomega}) in Def. \ref{def:daha}. 

\subsubsection{Connected components of $\Bunpar$ and $\Mpar$} It is well-known that the set of connected components of $\Bun_G$ or $\Bunpar_G$ is naturally identified with $\Omega=\xcoch(T)/\ZZ\Phi^\vee$, such that the component containing the image of $\Bun_{G^{\textup{sc}}}\to \Bun_{G}$ is indexed by the identity element in $\Omega$ (here $G^{\textup{sc}}$ is the simply-connected form of the derived group $G^{\textup{der}}$ of $G$). For any $\omega\in\Omega$, let $\Bunpar_\omega$ be the corresponding component. We also write $\Mpar_\omega$ for the preimage of $\Bunpar_\omega$.

Recall that $\tilI$ is the normalizer of $\bI$ in $G(F)$, and $\Omega_{\bI}=\tilI/\bI$ can be identified with $\Omega$ via $\Omega_{\bI}\hookrightarrow\tilW\twoheadrightarrow\Omega$. For any $\omega\in\Omega_{\bI}$, by Lem. \ref{l:Pclass}, we have an automorphism $R_{\omega}$ of $\Bunpar_G$, which sends the connected component $\Bunpar_{\omega_1}$ to $\Bunpar_{\omega_1+\omega}$. Similarly remark applies to the action of $\Omega_{\bI}$ on $\Mpar$.

On the other hand, we can view $\omega\in\Omega_{\bI}$ as an element of $\tilW$. Therefore $\omega$ gives a double coset in $\bI\backslash G(F)/\bI$, hence a Hecke correspondence (see the beginning of \cite[Sec. 4.1]{GSI})
\begin{equation*}
\xymatrix{ & \Hecke^{\Bun}_{\omega}\ar[dl]_{\overleftarrow{b}_{\omega}}\ar[dr]^{\overrightarrow{b}_{\omega}} & \\
 \Bunpar_G\ar[dr] & & \Bunpar_G\ar[dl]\\
& X & }
\end{equation*}
classifying pairs of $G$-torsors with Borel reductions at a point of $X$ which are in relative position $\omega$.

\begin{lemma}\label{l:bunomega}
For $\omega\in\Omega_{\bI}$, the correspondence $\Hecke^{\Bun}_{\omega}$ is the graph of the automorphism $R_\omega:\Bunpar_G\to\Bunpar_G$.
\end{lemma}
\begin{proof}
It is clear that the Schubert cell $G_{\bI}\omega G_{\bI}/G_{\bI}$ consists of one point for any $\omega\in\Omega_{\bI}$. Therefore the lemma follows from the description (\ref{eq:Heckeloop}) of $\Hecke^{\Bun}_{\omega}$.
\end{proof}

\begin{cor}\label{c:omaction}
The reduced Hecke correspondence $\calH_{\omega}$ for the parabolic Hitchin stack $\Mpar$ is the graph of the automorphism $R_{\omega}:\Mpar\to\Mpar$. In particular, the action of $\omega\in\Omega_{\bI}\subset\tilW$ on $\fQl$ defined in \cite[Th. 4.4.3]{GSI} is the same as $R_\omega^*$.  
\end{cor}
\begin{proof}
By the construction of the $\Omega_{\bI}$-action on $\Mpar$, for any $m\in\Mpar(R)$ with image $x\in X(R)$, the Hitchin pairs on $\XR-\Gamma(x)$ given by restrictions of $m$ and $R_\omega m$ are canonically identified. Therefore, there is a natural embedding $\Gamma(R_{\omega})\hookrightarrow\Heckep$, where $\Gamma(R_\omega)$ is the graph of $R_\omega$.

Recall the morphism $\beta:\Heckep\to\Hecke^{\Bun}$ in (\ref{eq:betaHH}). We know from Lem. \ref{l:bunomega} that the $\beta(\Gamma(R_{\omega}))=\Hecke^{\Bun}_\omega$. In other words,
\begin{equation*}
\Gamma(R_\omega)\subset\beta^{-1}(\Hecke^{\Bun}_\omega)^{\red}.
\end{equation*}
On the other hand, by Lem. \ref{l:hhecke}, the reduced structure of $\beta^{-1}(\Hecke^{\Bun}_{\omega})^{\rs}$ is contained in $\calH^{\rs}_\omega$, hence $\Gamma(R_\omega)^{\rs}\subset\calH^{\rs}_\omega$. Since both $\Gamma(R_\omega)^{\rs}$ and $\calH^{\rs}_{\omega}$ are graphs, we must have $\Gamma(R_{\omega})^{\rs}=\calH^{\rs}_{\omega}$. Taking closures, we get $\Gamma(R_{\omega})=\calH_{\omega}$. 
\end{proof}

By Construction \ref{cons:action}, the action of $\xch(\tilT)$ on $\fQl$ is defined by the cup product with the Chern classes of the pull-back of the line bundles $\calL(\xi)$ from $\Bunpar_G$. Now Lem. \ref{c:omaction} reduces the verification of the condition (\ref{eq:commomega}) in Def. \ref{def:daha} to the following fact:

\begin{lemma}
For each $\omega\in\Omega_{\bI}$ and $\xi\in\xch(\tilT)$, there is an isomorphism of line bundles on $\Bunpar_G$:
\begin{equation*}
R_{\omega}^*\calL(\xi)\cong\calL(\leftexp{\omega}{\xi}).
\end{equation*}
\end{lemma}
\begin{proof}
Recall from Construction \ref{cons:linebd} that the right action of $\tilT$ on $\calL^{\tilT}$ comes from the right action of $\calG_{\bI}$ on $\hatBun_{\infty}$. On the other hand, the right action of $\Omega_{\bI}$ on $\Bunpar_G$ comes from the right action of $\calG_{\tilI}$ on $\hatBun_{\infty}$ (see the discussion in the beginning of Sec. \ref{ss:daha}). For any $\hatg\in\calG_{\bI}$ and $\hatom\in\calG_{\tilI}$, it is clear that:
\begin{equation*}
R_{\Ad(\hatom^{-1})\hatg}\circ R_{\hatom}=R_{\hatg\hatom}=R_{\hatom}\circ R_{\hatg}.
\end{equation*}
Taking the quotient by $\calG^u_{\bI}$, we get an equality of actions on $\calL^{\tilT}=\hatBun_{\infty}/\calG^u_{\bI}$:
\begin{equation*}
R_{\Ad(\omega^{-1})g}\circ R_{\omega}=R_{\omega}\circ R_{g},\textup{ for }\omega\in\Omega_{\bI},g\in\tilT.
\end{equation*}
Therefore the $\tilT$-torsor $R_{\omega}^*\calL^{\tilT}$ on $\Bunpar_G$ is the $\Ad(\omega)$-twist of $\calL^{\tilT}$. This proves the lemma.
\end{proof}

% SL(2)

\subsection{Simple reflections---a calculation in $\mathfrak{sl}_2$}\label{ss:simplerefl}

In this subsection, we check the condition (\ref{eq:commrefl}) in Def. \ref{def:daha} for $\xi\in\xcoch(T\times\grot)$. The idea is to reduce the problem to a calculation for the Steinberg variety of $\SL_2$. For $i=0,\cdots,n$, let $\bP_i$ be the standard parahoric subgroup whose Lie algebra $\frg_{\bP_i}$ is spanned by $\frg_{\bI}$ and the root space of $-\alpha_i$. We will abbreviate $L_{\bP_i},\frl_{\bP_i},B^{\bP_i}_{\bI},\frb^{\bP_i}_{\bI}$, etc. by $L_{i},\frl_i,B^i,\frb^i$, etc. 

\begin{lemma}\label{l:Hsi}
The reduced Hecke correspondence $\calH_{s_i}$ is a closed substack of $C_i=\Mpar\times_{\calM_{\bP_i}}\Mpar$.
\end{lemma}
\begin{proof}
For two point $(x,\calE_i,\varphi_i,\calE^B_{x,i})\in\Mpar(R)$ ($i=1,2$) with the same image in $\calM_{\bP_i}$, we have a canonical isomorphism $(\calE_1,\varphi_1)|_{\XR-\Gamma(x)}\cong(\calE_2,\varphi_2)|_{\XR-\Gamma(x)}$. Therefore, we have a canonical embedding of self-correspondences of $\Mpar$:
\begin{equation*}
\gamma_i:C_i:=\Mpar\times_{\calM_{\bP_i}}\Mpar\hookrightarrow\Heckep.
\end{equation*}

Note that $\Hecke^{\Bun}_{\leq s_i}=\Bunpar_G\times_{\Bun_{\bP_i}}\Bunpar_G$, therefore the image of $\gamma_i(C_i)$ in $\Hecke^{\Bun}$ lies in $\Hecke^{\Bun}_{\leq s_i}$.  Then by Lem. \ref{l:hhecke}, the reduced structure of $\gamma_i(C_i^{\rs})$ must lie in $\calH^{\rs}_{\leq s_i}=\calH^{\rs}_{e}\coprod\calH^{\rs}_{s_i}$. By Lem. \ref{l:CartPQ} applied to $\bI\subset\bP_i$, we see that $\tforg^{\bP_i,\rs}_{\bI}:\Mparrs\to\calM^{\rs}_{\bP_i}$ is an \'etale double cover. Therefore the two projections $C_i^{\rs}\rightrightarrows\Mparrs_{\bI}$ are also \'etale double covers. Since the two projections $\calH^{\rs}_{\leq s_i}\rightrightarrows\calM^{\rs}_{\bI}$ are also \'etale double covers, $\gamma_i$ must induce an isomorphism $C_i^{\rs}\isom\calH^{\rs}_{\leq s_i}$. Taking closures, we conclude that $\calH_{s_i}$ lies in $\gamma_i(C_i)$. This proves the lemma.
\end{proof}

By Lem. \ref{l:CartPQ} and Lem. \ref{l:Hsi}, we have a Cartesian diagram of correspondences
\begin{equation}\label{d:cst}
\xymatrix{C_i\ar[r]\ar@<-1ex>@/_/[d]_{\overleftarrow{c_i}}\ar@<1ex>@/^/[d]^{\overrightarrow{c_i}}\ar[r] & [\unSt_i/\unL_{i}]_D\ar@<-1ex>@/_/[d]\ar@<1ex>@/^/[d]\ar[r] & [\St_i/\Ln_i]\ar@<-1ex>@/_/[d]_{\overleftarrow{\st_i}}\ar@<1ex>@/^/[d]^{\overrightarrow{\st_i}}\\
\Mpar\ar[d]^{\tforg^i}\ar[r] & [\unb^i/\unB^i]_D\ar[d]\ar[r] & [\till_i/\Ln_i]\ar[d]^{\pi^i}\\
\calM_{\bP_i}\ar[r]^{} & [\unl_i/\unL_i]_D\ar[r] & [\frl_i/\Ln_i]}
\end{equation}
We explain the notations. Here $\St_i$ is the Steinberg variety $\till_i\times_{\frl_i}\till_i$ of $\frl_i$ and $\till_i$ is the Grothendieck simultaneous resolution of $\frl_i$. Recall that the action of $\Aut_{\calO}$ on $L_i$ factors through a finite dimensional quotient $Q$. We assume that $Q$ surjects to $\grot$. The conjugation action of $L_i$ on $L_i$ and the action of $\Aut_{\calO}$ on $L_i$ gives an action of $L_i\rtimes Q$ on $L_i$, and hence on $\frl_i,\till_i$ and $\St_i$. The group $\Ln_i$ in the diagram (\ref{d:cst}) is defined as
\begin{equation*}
\Ln_i=(L_i\rtimes Q)\times\GG_m,
\end{equation*}
which acts on $\frl_i,\till_i$ and $\St_i$ with $\GG_m$ acting by dilation.

The natural projection $B^i\to T$ extends to the projection $B^i\rtimes Q\to T\times\grot$. Therefore we have a morphism $[\till_i/\Ln_i]=[\frb^i/(B^i\rtimes Q\times\GG_m)]\to\BB (T\times\grot)$, which gives a $T\times\grot$-torsor on $[\till_i/\Ln_i]$. The associated line bundles on $[\till_i/\Ln_i]$ are denoted by $\calN(\xi)$, for $\xi\in\xch(T\times\grot)$.

Let $\St_{i}=\St^{+}_{i}\cup\St^-_{i}$ be the decomposition into two irreducible components, where $\St^+_i$ is the diagonal copy of $\till_i$, and $\St^-_i$ is the non-diagonal component. Let $\epsilon$ be the composition
\begin{equation*}
\epsilon:C_i\to[\unSt_i/\unL_i]_D\to[\St_i/\Ln_i].
\end{equation*}

\begin{lemma}\label{l:redcomm}
For $\xi\in\xch(T\times\grot)$, the action of $s_i\xi-\leftexp{s_i}{\xi}s_i-\jiao{\xi,\alpha^\vee_i}u$ on $\fQl$ is given by the following cohomological correspondence in $\Corr(C_i;\Ql[2](1),\Ql)$:
\begin{equation*}
\epsilon^*\left([\St^-_i/\Ln_i]\cup\left(\overrightarrow{\st_i}^*c_1(\calN(\xi))-\overleftarrow{\st_i}^*c_1(\calN(\leftexp{s_i}{\xi}))\right)-[\St^+_i/\Ln_i]\cup\jiao{\xi,\alpha_i^\vee}v\right)
\end{equation*}
where $v\in\cohog{2}{[\St_i/\Ln_i]}(1)$ is the image of the generator of $\cohog{2}{\BB\GG_m}(1)$ (for the $\GG_m$ factor in $\Ln_i$).
\end{lemma}
\begin{proof}
By Construction \ref{cons:action} and Lem. \ref{l:cup} about the cup product action on cohomological correspondences, the action of $s_i\xi-\leftexp{s_i}{\xi}s_i-\jiao{\xi,\alpha^\vee_i}u$ on $\fQl$ is given by the following cohomological correspondence in $\Corr(C_i;\Ql[2](1),\Ql)$:
\begin{equation*}
[\calH_{s_i}]\cup\left(\overrightarrow{c_i}^*c_1(\calL(\xi))-\overleftarrow{c_i}^*c_1(\calL(\leftexp{s_i}{\xi}))\right)-\jiao{\xi,\alpha^\vee_i}[\Delta(\Mpar)]\cup c_1(D).
\end{equation*}
where $[\calH_{s_i}]\in\Corr(C_i;\Ql,\Ql)$ is the image of fundamental class of $\calH_{s_i}$ via the natural closed embedding $\calH_{s_i}\hookrightarrow C_i$ (see Lem. \ref{l:Hsi}), and $\Delta(\Mpar)\subset C_i$ is the diagonal.

Therefore, to prove the lemma, we have to check
\begin{eqnarray}
\label{eq:LNxi1}\epsilon^*\overrightarrow{\st_i}^*\calN(\xi)&=&\overrightarrow{c_i}^*\calL(\xi);\\
\label{eq:LNxi2}\epsilon^*\overleftarrow{\st_i}^*\calN(\leftexp{s_i}{\xi})&=&\overleftarrow{c_i}^*\calL(\leftexp{s_i}{\xi});\\
\label{eq:vcD}\epsilon^*v&=&c_1(D)\in\cohog{2}{C_i}(1);\\
\label{eq:pullHsi}\epsilon^*[\St^-_i/\Ln_i]&=&[\calH_{s_i}]\in\Corr(C_i;\Ql,\Ql);
\end{eqnarray}

(\ref{eq:LNxi1}) Let $\ev:\Mpar\to[\till_i/\Ln_i]$ be the evaluation morphism. Then we have $\calL(\xi)=\ev^*\calN(\xi)$. By the first two rows of the diagram (\ref{d:cst}), we have
\begin{equation*}
\epsilon^*\overrightarrow{\st_i}^*\calN(\xi)=\overrightarrow{c_i}^*\ev^*\calN(\xi)=\overrightarrow{c_i}^*\calL(\xi).
\end{equation*}

(\ref{eq:LNxi2}) is proved in a similar way as (\ref{eq:LNxi1}).

(\ref{eq:vcD}) By definition, we have a commutative diagram
\begin{equation*}
\xymatrix{[\unSt_i/\unL_i]_D\ar[r]\ar[d] & [\St_i/L_i\rtimes Q\times\GG_m]\ar[d]\\
X\ar[r]^{\rho_D} & \BB\GG_m}
\end{equation*}
Therefore, the generator $v\in\cohog{2}{\BB\GG_m}(1)$ pulls back to $c_1(D)\in\cohog{2}{X}(1)$, which further pulls back to $c_1(D)\in\cohog{2}{C_i}(1)$.

(\ref{eq:pullHsi}) As a finite type substack of $\Heckep$, $C_i$ satisfies (G-2) in \cite[Def. A.5.1]{GSI} with respect to $(\calA\times X)^{\rs}\subset\calA\times X$ (see \cite[Lem. 4.4.4]{GSI}). By \cite[Lem. A.5.2]{GSI}, we only need to verify the equality (\ref{eq:pullHsi}) over $(\calA\times X)^{\rs}$, which is obvious.
\end{proof}

By Lem. \ref{l:redcomm}, the condition (\ref{eq:commrefl}) for $\xi\in\xch(T\times\grot)$ reduces to the following identity.

\begin{prop}
For each $\xi\in\xch(T\times\grot)$, the following identity hold in $\Corr([\St_i/\Ln_i];\Ql[2](1),\Ql)$:
\begin{equation}\label{eq:corr2}
[\St^-_i/\Ln_i]\cup\left(\overrightarrow{\st_i}^*c_1(\calN(\xi))-\overleftarrow{\st_i}^*c_1(\calN(\leftexp{s_i}{\xi}))\right)=[\St^+_i/\Ln_i]\cup\jiao{\xi,\alpha_i^\vee}u.
\end{equation}
\end{prop}
\begin{proof}
Since the reductive group $L_i$ has semisimple rank one, we can decompose $\xch(T\times\grot)_{\QQ}=\xch(T\times\grot)\otimes_{\ZZ}\QQ$ into $\pm1$-eigenspaces of the reflection $s_i$:
\begin{equation*}
\xch(T\times\grot)_{\QQ}=\xch(T\times\grot)^{s_i}_{\QQ}\oplus\QQ\alpha_i,
\end{equation*}
where $\alpha_i$ spans the $-1$-eigenspace of $s_i$.

To prove (\ref{eq:corr2}), it suffices to prove it for $\xi\in\xch(T\times\grot)^{s_i}$ and $\xi=\alpha_i$ separately. In the first case, taking Chern class induces an isomorphism
\begin{equation*}
c_1:\xch(T\times\grot)^{s_i}_{\Ql}\isom\cohog{2}{\BB(L_i\rtimes Q)}(1)\hookrightarrow\cohog{2}{\BB(B^i\rtimes Q)}(1),
\end{equation*}
Hence $c_1(\calN(\xi))$ lies in the image of the pull-back map
\begin{equation*}
\pi^{i,*}:\cohog{2}{\BB\Ln_i}(1)\to\cohog{2}{[\frl_i/\Ln_i]}(1)\to\cohog{2}{[\till_i/\Ln_i]}(1).
\end{equation*}
Since $\pi^i\circ\overleftarrow{\st_i}=\pi^i\circ\overrightarrow{\st_i}$, we conclude that
\begin{equation*}
\overrightarrow{\st_i}^*c_1(\calN(\xi))=\overleftarrow{\st_i}^*c_1(\calN(\xi))=\overleftarrow{\st_i}^*c_1(\calN(\leftexp{s_i}{\xi}))
\end{equation*}
Therefore, the LHS of (\ref{eq:corr2}) is zero. On the other hand, since $\leftexp{s_i}{\xi}=\xi$, we have $\jiao{\xi,\alpha_i^\vee}=0$, hence the RHS of (\ref{eq:corr2}) is also zero. This proves the identity (\ref{eq:corr2}) in the case $\xi\in\xch(T\times\grot)^{s_i}$.

Finally we treat the case $\xi=\alpha_i$. Since $L_i\rtimes Q$ is connected, the action of $L_i\rtimes Q$ on $L_i$ factors through a homomorphism $L_i\rtimes Q\to L^{\ad}_{i}$, where $L^{\ad}_i$ is the adjoint from of $L_i$ (isomorphic to $\textup{PGL}(2)$). Let $\PP^1_i=L_i/B^i=L_i\rtimes Q/B^i\rtimes Q=L^{\ad}_i/B^{\ad,i}$ be the flag variety of $L_i$ or $L^{\ad}_i$. The pull-back
\begin{equation*}
\upH^2_{L^{\ad}_i}(\PP^1_i)\to \upH^2_{L_i\rtimes Q}(\PP^1_i)=\cohog{2}{\BB(B^i\rtimes Q)}(1)=\xch(T\times\grot)_{\Ql}
\end{equation*}
has image $\Ql\alpha_i$, and the line bundle $\calN(\alpha_i)$ on $\till_i$ is the pull-back of the canonical bundle $\omega_{\PP^1_i}$ on $\PP^1_i$. We can therefore only consider the $L^{\ad}_i\times\GG_m$-action on $\St_i$. The equality (\ref{eq:corr2}) then reduces to the following identity in the $L^{\ad}_i\times\GG_m$-equivariant Borel-Moore homology group $\upH^{\BM,L^{\ad}_i\times\GG_m}_{2d-2}(\St_i)(1)$ ($d=\dim\St_i$):
\begin{equation}\label{eq:corrSt}
h^{-,*}c_1(\omega_{\PP^1_i\times\PP^1_i})\cup[\St^-_i]=2v\cup[\St^+_i],
\end{equation}
where $h^-$ is the $L^{\ad}_i\times\GG_m$-equivariant morphism $h^-:\St^-_i\to\PP^1_i\times\PP^1_i$.

We claim that both sides of (\ref{eq:corrSt}) are equal to the fundamental class $2[\St^{\nil}]\in\upH^{\BM,L^{\ad}_i\times\GG_m}_{2d-2}(\St_i)(1)$, where $\St^{\nil}$ is the preimage of the nilpotent cone in $\frl_i$ under $\St_i\to\frl_i$.

On one hand, $\St^{\nil}$ is the preimage of the diagonal $\Delta(\PP^1_i)\subset\PP^1_i\times\PP^1_i$ under $h^-$. Let $\calI_{\Delta}$ be the ideal sheaf of the diagonal $\Delta(\PP^1_i)$, viewed as an $L^{\ad}_i$-equivariant line bundle on $\PP^1_i\times\PP^1_i$. We claim that
\begin{equation}\label{eq:2line}
\calI_{\Delta}^{\otimes2}\cong\omega_{\PP^1_i\times\PP^1_i}\in\Pic_{L^{\ad}_i}(\PP^1_i\times\PP^1_i).
\end{equation}
In fact, since $L^{\ad}_i$ does not admit nontrivial characters, we have an isomorphism $\Pic_{L^{\ad}_i}(\PP^1_i\times\PP^1_i)\cong\ZZ\oplus\ZZ$ given by taking the degrees along the two rulings of $\PP^1_i\times\PP^1_i$. Then (\ref{eq:2line}) follows by comparing the degrees along the rulings.

Since the Poincar\'e dual of $c_1(\calI_\Delta)$ is the cycle class $[\Delta(\PP^1_i)]\in \upH^{\BM,L^{\ad}_i}_2(\PP^1_i\times\PP^1_i)(1)$, we get from (\ref{eq:2line}) that
\begin{equation}\label{eq:classdel}
c_1(\omega_{\PP^1_i\times\PP^1_i})\cup[\PP^1_i\times\PP^1_i]=2[\Delta(\PP^1_i)].
\end{equation}
Since the morphism $h^-$ is smooth ($\St^-_i$ is in fact the total space of a line bundle over $\PP^1_i\times\PP^1_i$), we can pull-back (\ref{eq:classdel}) along $h^-$ to get
\begin{equation}\label{eq:Stnil1}
h^{-,*}c_1(\omega_{\PP^1_i\times\PP^1_i})\cup[\St^-_i]=2[\St^{\nil}_i]\in \upH^{\BM,L^{\ad}_i\times\GG_m}_{2d-2}(\St^-_i)(1).
\end{equation}

On the other hand, consider the projection $\tau:\till_i\to\frt\to\frt^{\ad}$ ($\frt^{\ad}$ is the universal Cartan for $L^{\ad}_i$), then $\St^{\nil}_i=\tau^{-1}(0)$. The class $[0]\in \upH^{\BM,\GG_m}_0(\frt^{\ad})(1)$ is the Poincar\'e dual of $u$ ($\GG_m$ acts on the affine line $\frt^{\ad}$ by dilation). Since $\tau$ is $L^{\ad}_i\times\GG_m$-equivariant and $L^{\ad}_i$ acts trivially on $\frt^{\ad}$, we conclude that
\begin{equation}\label{eq:Stnil2}
[\St^{\nil}_i]=u\cup[\St^+_i]\in \upH^{\BM,L^{\ad}_i\times\GG_m}_{2d-2}(\St^+_i)(1).
\end{equation}

If we view both identities (\ref{eq:Stnil1}) and (\ref{eq:Stnil2}) as identities in $\upH^{\BM,L^{\ad}_i\times\GG_m}_{2d-2}(\St_i)(1)$, we get the identity (\ref{eq:corrSt}). This completes the proof.
\end{proof}

% Pf

\subsection{Completion of the proof of Theorem \ref{th:daction}}\label{ss:pfdaha}
To prove Th. \ref{th:daction}, it only remains to verify the relation (\ref{eq:commrefl}) in Def. \ref{def:daha} for $\xi=\Lambda_{\can}$.

For each standard parahoric subgroup $\bP\subset G(F)$, define a line bundle $\calL_{\bP,\can}$ on $\Bun_{\bP}$ which to every point $(x,\calE,\tau_x\modo\bP)\in\Bun_{\bP}(R)$ assigns the invertible $R$-module $\det\bR\Gamma(\XR,\Ad_{\bP}(\calE))$. In particular, $\calL_{\bG,\can}$ is the pull-back of $\canb$ from $\Bun_G$ to $\Bun_{\bG}=\Bun_{G}\times X$.

\begin{lemma}\label{l:pullcan} For each standard parahoric subgroup $\bP\subset G(F)$, we have
\begin{equation}\label{eq:IP}
\calL_{\bI,\can}\otimes\calL(-2\rho_{\bP})\cong\forg^{\bP,*}_{\bI}\calL_{\bP,\can}\in\Pic(\Bunpar_G).
\end{equation}
Here $2\rho_{\bP}$ is the sum of positive roots in $L_{\bP}$ (with respect to the Borel $B^{\bP}_{\bI}$).
\end{lemma}
\begin{proof}
For $(x,\alpha,\calE,\tau_x\modo\bI)\in\tilBun_{\bI}(R)$, we have an exact sequence of vector bundles on $\XR$
\begin{equation}\label{eq:Q}
0\to\Ad_{\bI}(\calE)\to\Ad_{\bP}(\calE)\to i_*\calQ(\calE)\to0
\end{equation}
where $\calQ(\calE)$ is a coherent sheaf supported on $\Gamma(x)$. As $\calE$ varies, we can view $\calQ$ as a vector bundle over $\Bunpar_G$. Via the local coordinate $\alpha$ and the full level structure $\tau_x$, we can identify $\calQ(\calE)$ with the $R$-module $(\frg_{\bP}/\frg_{\bI})\otimes_kR$. In other words, we have
\begin{equation*}
\calQ\cong\tilBun_{\infty}\twtimes{G_{\bI}\rtimes\Aut_{\calO}}(\frg_{\bP}/\frg_{\bI}).
\end{equation*}
Taking the determinant, we get
\begin{equation*}
\det\calQ\cong\tilBun_{\infty}\twtimes{G_{\bI}\rtimes\Aut_{\calO}}\det(\frg_{\bP}/\frg_{\bI}).
\end{equation*}
Since the action of $G_{\bI}\rtimes\Aut_{\calO}$ on $\det(\frg_{\bP}/\frg_{\bI})$ factors through the quotient $G_{\bI}\rtimes\Aut_{\calO}\to T\times\grot\xrightarrow{-2\rho_{\bP}}\GG_m$, we conclude that 
\begin{equation}\label{eq:detQ}
\det\calQ\cong\calL(-2\rho_{\bP}).
\end{equation}

Taking the determinant of the exact sequence (\ref{eq:Q}), we get
\begin{equation*}
\det\bR\Gamma(\XR,\Ad_{\bP}(\calE))\cong\det\bR\Gamma(\XR,\Ad_{\bI}(\calE))\otimes\det\calQ(\calE)
\end{equation*}
Plugging in (\ref{eq:detQ}), we get the isomorphism (\ref{eq:IP}).
\end{proof}

\begin{cor}\label{c:pullcan}
For each $i=0,\cdots,n$, there is an isomorphism of line bundles on $\Bunpar_G$:
\begin{equation*}
\forg^{\bP_i,*}_{\bI}\calL_{\bP_i,\can}\cong\forg^{\bG,*}_{\bI}\canb\otimes\calL(2\rho-\alpha_i)
\end{equation*}
where $2\rho$ is the sum of positive roots in $G$. 
\end{cor}

Since the self-correspondence $\calH_{s_i}$ is over $\calM_{\bP_i}$, hence over $\Bun_{\bP_i}$, the action of $[\calH_{s_i}]$ commutes with $\cup c_1(\calL_{\bP_i,\can})$. Using Lem. \ref{c:pullcan}, we conclude that
\begin{equation}\label{eq:commLam}
s_i(\Lambda_{\can}+2\rho-\alpha_i)=(\Lambda_{\can}+2\rho-\alpha_i)s_i\in\Hom_{\calA\times X}(\fQl,\fQl[2](1)).
\end{equation}
Observe that for $i=1,\cdots,n$, we have
\begin{equation*}
\jiao{\Lambda_{\can}+2\rho-\alpha_i,\alpha_i^\vee}=2\jiao{\rho,\alpha_i^\vee}-2=0.
\end{equation*}
For $i=0$, we have
\begin{eqnarray*}
\jiao{\Lambda_{\can}+2\rho-\alpha_0,\alpha_0^\vee}&=&\jiao{\Lambda_{\can}+2\rho-\delta+\theta,K-\theta^\vee}\\
&=&\jiao{\Lambda_{\can},K}-2\jiao{\rho,\theta^\vee}-2=2h^\vee-2h^\vee=0.
\end{eqnarray*}
Here we have used the fact that $\jiao{\Lambda_{\can},K}=2h^\vee$ (see Rem. \ref{rm:norm}) and  $h^\vee=\jiao{\rho,\theta^\vee}+1$ (see Lem. \ref{l:dCox}). In any case, we have $\jiao{\Lambda_{\can}+2\rho-\alpha_i,\alpha_i^\vee}=0$ for $i=0,\cdots,r$. This, together with (\ref{eq:commLam}) means that the relation (\ref{eq:commrefl}) in Def. \ref{def:daha} holds for $s_i$ and $\xi=\Lambda_{\can}+2\rho-\alpha_i$. Since we have already proved the relation (\ref{eq:commrefl}) in Def. \ref{def:daha} for $s_i$ and $\xi=2\rho-\alpha_i\in\xch(T\times\grot)$ in Sec. \ref{ss:simplerefl}, we can subtract this relation from the one for $\xi=\Lambda_{\can}+2\rho-\alpha_i$, and conclude that the same relation also holds for $s_i$ and $\xi=\Lambda_{\can}$. This completes the proof of relation (\ref{eq:commrefl}) in Def. \ref{def:daha}, and hence the proof of Th. \ref{th:daction}.

\section{Generalizations to parahoric Hitchin moduli stacks}\label{s:genaction}

In this section, we generalize the main results in \cite[Sec. 4]{GSI} and Sec. \ref{s:DAHA} to the case of parahoric Hitchin moduli stacks of arbitrary type $\bP$. In particular, in the case $\bP=\bG$, we get an $\Ql$-analogue of the so-called {\em 'tHooft operators} considered by Kapustin-Witten in their gauge-theoretic approach to the geometric Langlands program (see \cite{KW}).

\subsection{The action of the convolution algebra}\label{ss:parconv}
We give two approaches to the parahoric version of \cite[Th. 4.4.3]{GSI}. One using Hecke correspondences (Construction \ref{cons:HPQ}) and the other using the smallness of $\tforg^{\bP}_{\bI}$ (Construction \ref{cons:alterPWaction}).

As in the case of $\Mpar$, we can define the Hecke correspondence between two parahoric Hitchin moduli stacks $\calM_{\bP}$ and $\calM_{\bQ}$ over $\calA\times X$:
\begin{equation}\label{d:HeckePQ}
\xymatrix{ & \Hk{\bP}{\bQ} \ar[dl]_{\overleftarrow{h}}\ar[dr]^{\overrightarrow{h}} & \\
\calM_{\bP}\ar[dr]_{f_{\bP}} & & \calM_{\bQ}\ar[dl]^{f_{\bQ}} \\
& \calA\times X &}
\end{equation}
For any scheme $S$, $\Hk{\bP}{\bQ}(S)$ is the groupoid of tuples 
\begin{equation*}
(x,\calE_1,\tau_{1,x}\modo\bP,\varphi_1,\calE_2,\tau_{2,x}\modo\bQ,\varphi_2,\alpha)
\end{equation*}
where
\begin{itemize}
\item $(x,\calE_1,\tau_{1,x}\modo\bP,\varphi_1)\in\calM_{\bP}(S)$;
\item $(x,\calE_2,\tau_{2,x}\modo\bQ,\varphi_2)\in\calM_{\bQ}(S)$;
\item $\alpha$ is an isomorphism of Hitchin pairs $(\calE_1,\varphi_1)|_{\XS-\Gamma(x)}\isom(\calE_2,\varphi_2)|_{\XS-\Gamma(x)}$.
\end{itemize}

\begin{cons}\label{cons:HPQ} For every double coset $W_{\bP}\tilw W_{\bQ}\subset\tilW$, we will construct a graph-like closed sub-correspondence $\calH_{W_{\bP}\tilw W_{\bQ}}$ of $\Hk{\bP}{\bQ}$. Let $\Hrs{\bP}{\bQ}$ be the reduced structure of the restriction of $\Hk{\bP}{\bQ}$ to $(\calA\times X)^{\rs}$. By definition, we have an isomorphism
\begin{equation}\label{eq:twobasechange}
\Mparrs\times_{\calM^{\rs}_{\bP}}\Hrs{\bP}{\bQ}\times_{\calM^{\rs}_{\bQ}}\Mparrs\cong\calH^{\rs}=\bigsqcup_{\tilw\in\tilW}\calH^{\rs}_{\tilw}.
\end{equation}
Recall that each $\calH^{\rs}_{\tilw}$ is the graph of the right $\tilw$-action on $\Mparrs$. Moreover, by Lem. \ref{l:CartPQ}, the projections $\tforg^{\bP}_{\bI}:\Mparrs\to\calM^{\rs}_{\bP}$ and $\tforg^{\bQ}_{\bI}:\Mparrs\to\calM^{\rs}_{\bQ}$ are the quotients under the right actions of $W_{\bP}\subset\tilW$ and $W_{\bQ}\subset\tilW$ on $\Mparrs$. If we identify $\calH^{\rs}_{\tilw}$ with $\Mparrs$ via $\overleftarrow{h_{\tilw}}$, we also get a right $\tilW$-action on $\calH^{\rs}_{\tilw}$. By (\ref{eq:twobasechange}), the projection $\calH^{\rs}_{\tilw}\to\Hrs{\bP}{\bQ}$ factors through the quotient
\begin{equation*}
\calH^{\rs}_{\tilw}\to\calH^{\rs}_{\tilw}/(W_{\bP}\cap\tilw W_{\bQ}\tilw^{-1})\hookrightarrow\Hrs{\bP}{\bQ}.
\end{equation*}
We define $\calH_{W_{\bP}\tilw W_{\bQ}}$ to be the closure of $\calH^{\rs}_{\tilw}/(W_{\bP}\cap\tilw W_{\bQ}\tilw^{-1})$ in $\Hk{\bP}{\bQ}$. Clearly, $\calH_{W_{\bP}\tilw W_{\bQ}}$ only depends on the double coset $W_{\bP}\tilw W_{\bQ}\subset\tilW$. The projections from $\calH^{\rs}_{W_{\bP}\tilw W_{\bQ}}$ to $\calM^{\rs}_{\bP}$ and $\calM^{\rs}_{\bQ}$ are finite \'etale, hence $\calH_{W_{\bP}\tilw W_{\bQ}}$ is graph-like.
\end{cons}

\subsubsection{The convolution algebras} To state a generalization of \cite[Th. 4.4.3]{GSI} to parahoric Hitchin moduli stacks, we first need to introduce certain convolution algebras. For standard parahoric subgroups $\bP,\bQ$, let
\begin{equation*}
\Ql[\dquot{\bP}{\bQ}]\subset\Ql[\tilW]
\end{equation*}
be the subspace of $\Ql$-valued functions on $\tilW$ which are nonzero only at finitely many elements of $\tilW$, left invariant under $W_{\bP}$ and right invariant under $W_{\bQ}$. We define the {\em convolution product}:
\begin{eqnarray}\label{eq:convQ}
\conv{\bQ}:\Ql[\dquot{\bP}{\bQ}]\otimes\Ql[\dquot{\bQ}{\bR}]\to\Ql[\dquot{\bP}{\bR}]\\
f_1\otimes f_2\mapsto(f_1\conv{\bQ}f_2)(\tilw)=\sum_{\tilv\in\tilW/W_{\bQ}}f_1(\tilv)f_2(\tilv^{-1}\tilw).
\end{eqnarray}
In particular, $\Ql[\dquot{\bP}{\bP}]$ becomes a unital algebra under $\conv{\bP}$ with identity element $\one_{W_{\bP}}$, the characteristic function of the double coset $W_{\bP}\subset\tilW$. However, the natural embedding $\Ql[\dquot{\bP}{\bP}]\subset\Ql[\tilW]$ is {\em not} an algebra homomorphism; it becomes an algebra homomorphism if we divide the inclusion map by $\#W_{\bP}$.

\begin{theorem}\label{th:paraction}
\begin{enumerate}
\item []
\item For each pair of standard parahoric subgroups $(\bP,\bQ)$, the assignment
\begin{equation*}
\one_{W_{\bP}\tilw W_{\bQ}}\mapsto[\calH_{W_{\bP}\tilw W_{\bQ}}]_\#
\end{equation*}
defines a map
\begin{equation*}
\Ql[\dquot{\bP}{\bQ}]\to\Corr(\Hk{\bP}{\bQ};\Ql,\Ql)\xrightarrow{(-)_\#}\Hom_{\calA\times X}(f_{\bQ,*}\Ql,f_{\bP,*}\Ql).
\end{equation*}
Then these maps are compatible with the convolution product in (\ref{eq:convQ}) and the composition of maps between the complexes $f_{\bP,*}\Ql$ for standard parahoric subgroups $\bP$.
\item In particular, for each standard parahoric subgroup $\bP$, there is an algebra homomorphism
\begin{equation*}
\Ql[\dquot{\bP}{\bP}]\to\Corr(\Hk{\bP}{\bP};\Ql,\Ql)\xrightarrow{(-)_\#}\End_{\calA\times X}(f_{\bP,*}\Ql).
\end{equation*}
sending $\one_{W_{\bP}\tilw W_{\bP}}$ to $[\calH_{W_{\bP}\tilw W_{\bP}}]_\#$. In other words, the convolution algebra $\Ql[\dquot{\bP}{\bP}]$ acts on the complex $f_{\bP,*}\Ql$.
\end{enumerate}
\end{theorem}

The proof of this theorem is similar to that of \cite[Th. 4.4.3]{GSI}. The key ingredient is an analogue of \cite[Lem. 4.4.4]{GSI} for $\Hk{\bP}{\bQ}$.

We give another way to construct the convolution algebra action in Th. \ref{th:paraction}, using the smallness of the forgetful morphisms $\tforg^{\bP}_{\bI}$.

\begin{cons}\label{cons:WP} We mimic the construction of the classical Springer action reviewed in \cite[Construction 4.1.1]{GSI}. Fix a standard parahoric $\bP$. By Prop. \ref{p:forsmall} that the morphism $\tforg^{\bP}_{\bI}:\Mpar\to\calM_{\bP}$ is small, therefore the shifted perverse sheaf $\tforg^{\bP}_{\bI,*}\Ql$ is the middle extension of its restriction to $\calM^{\rs}_{\bP}$. Over $\calM^{\rs}_{\bP}$, the morphism $\tforg^{\bP}_{\bI}$ is a right $W_{\bP}$-torsor by Lem. \ref{l:CartPQ}, therefore we get a left action of $W_{\bP}$ on $\tforg^{\bP}_{\bI,*}\Ql|_{\calM^{\rs}_{\bP}}$, and hence on $\tforg^{\bP}_{\bI,*}\Ql$ by middle extension. Taking direct image along $f_{\bP}$, we get a left action of $W_{\bP}$ on $f_{\bP,*}\tforg^{\bP}_{\bI,*}\Ql=\fQl$.
\end{cons}

\begin{lemma}\label{eq:twoconssame}
The $W_{\bP}$-action on $\fQl$ in Construction \ref{cons:WP} coincides with the restriction of the $\tilW$-action on $\fQl$ in \cite[Th. 4.4.3]{GSI} to $W_{\bP}$.
\end{lemma}
\begin{proof}
Over $(\calA\times X)^{\rs}$, it is easy to check that the right $W_{\bP}$-action on $\Mparrs$ given by Lem. \ref{l:CartPQ} coincides with the restriction of the right $\tilW$-action on $\Mparrs$ constructed in \cite[Cor. 4.3.8]{GSI}. Let $w\in W_{\bP}$. Since $\calH_w$ is the closure of the $w$-action on $\Mparrs$, we have an embedding
\begin{equation*}
\bigsqcup_{w\in W_{\bP}}\calH_{w}\subset\Mpar\times_{\calM_{\bP}}\Mpar\subset\Heckep.
\end{equation*}
Therefore we can view $\calH_w$ as a self correspondence of $\Mpar$ over $\calM_{\bP}$. The cohomological correspondence $[\calH_w]$ then gives an endomorphism $[\calH_w]_\#$ of $\tforg^{\bP}_{\bI,*}\Ql$, which coincides with the $W_{\bP}$-action given in Construction \ref{cons:WP} by the same argument as \cite[Lem. 4.1.3]{GSI}, using the fact that $\tforg^{\bP}_{\bI,*}\Ql$ is a middle extension. On the other hand, taking the direct image of $[\calH_w]_\#$ along $f_{\bP,*}$, we get the action of $[\calH_w]_\#$ on $\fQl$ considered in \cite[Th. 4.4.3]{GSI}. This proves the lemma.
\end{proof}

Let $A$ be a finite group acting on an object $\calF$ in a Karoubi complete $\Ql$-linear category $\calC$, then we have a canonical decomposition
\begin{equation}\label{eq:Aeigendecomp}
\calF=\bigoplus_{\rho\in\Irr(A)}\calF_{\rho}\otimes V_\rho
\end{equation}
where $\Irr(A)$ is the set of isomorphism classes of irreducible $\Ql$-representations of $A$. For each $\rho\in\Irr(A)$, $V_\rho$ is the vector space on which $A$ acts as $\rho$. In fact, the decomposition (\ref{eq:Aeigendecomp}) is given by the images of the simple idempotents under the map $\Ql[A]\to\End_{\calC}(\calF)$. In particular, we have a canonical direct summand of $\calF$ corresponding to the trivial representation of $A$, which we denote by $\calF^A$, and call it the {\em $A$-invariants of $\calF$}.

\begin{lemma}\label{l:WPinv}
There is a canonical isomorphism in $D^b_c(\calA\times X)$:
\begin{equation}\label{eq:WPinv}
f_{\bP,*}\Ql\cong(\fQl)^{W_{\bP}}
\end{equation}
\end{lemma}
\begin{proof}
The morphism $\tforg^{\bP}_{\bI}:\Mpar\to\calM_{\bP}$ gives a map of shifted perverse sheaves
\begin{equation}\label{eq:mapIP}
\const{\calM_{\bP}}\to\tforg^{\bP}_{\bI,*}\Ql.
\end{equation}
Since $\tforg^{\bP,\rs}_{\bI}$ is a $W_{\bP}$-torsor, it is clear that the restriction of the map (\ref{eq:mapIP}) to $\calM^{\rs}_{\bP}$ is the embedding of the $W_{\bP}$-invariants of the RHS. Since both sides of (\ref{eq:mapIP}) are middle extensions from $\calM^{\rs}_{\bP}$, we conclude that the map (\ref{eq:mapIP}) can be identified with the inclusion of the $W_{\bP}$-invariants on the RHS. Taking $f_{\bP,*}$ we get the isomorphism (\ref{eq:Winv}).
\end{proof}

\begin{cons}\label{cons:alterPWaction}
Now it is easy to give another proof of Th. \ref{th:paraction}. From the $\tilW$-action on $\fQl$, we clearly have a map
\begin{equation*}
\Ql[\dquot{\bP}{\bQ}]\to\End_{\calA\times X}((\fQl)^{W_{\bQ}},(\fQl)^{W_{\bP}}).
\end{equation*}
By Lem. \ref{l:WPinv}, this gives a map
\begin{equation*}
\Ql[\dquot{\bP}{\bQ}]\to\End_{\calA\times X}(f_{\bQ,*}\Ql,f_{\bP,*}\Ql).
\end{equation*}
It is easy to check that this map is the same as the one constructed in Th. \ref{th:paraction}, and its compatibility with convolutions and compositions is clear from the fact that $\tilW$ acts on $\fQl$.
\end{cons}

\subsection{The enhanced actions}\label{ss:parenhance}
We can also define the {\em enhanced action} on $\tilf_{\bP,*}\Ql$ as we did for $\tfQl$ in \cite[Prop. 4.4.6]{GSI}. Consider the two projections of $\Hk{\bP}{\bP}$ to $\tcA_{\bP}$:
\begin{equation*}
\Hk{\bP}{\bP}\xrightarrow{(\overleftarrow{h_{\bP}},\overrightarrow{h_{\bP}})}\calM_{\bP}\times_{\calA\times X}\calM_{\bP}\xrightarrow{(\tilf_{\bP},\tilf_{\bP})}\tcA_{\bP}\times_{\calA\times X}\tcA_{\bP}.
\end{equation*}

Let $\Hecke_{\bP,[e]}$ be the preimage of the diagonal $\tcA_{\bP}\subset\tcA_{\bP}\times_{\calA\times X}\tcA_{\bP}$, viewed as a self-correspondence of $\calM_{\bP}$ over $\tcA_{\bP}$.

\begin{cons}\label{cons:Heckelambda} For $\lambda\in\xcoch(T)$, let $|\lambda|_{\bP}$ denote its $W_{\bP}$-orbit in $\xcoch(T)$. For each $W_{\bP}$-orbit $|\lambda|_{\bP}$, we will construct a graph-like closed substack $\calH_{|\lambda|_{\bP}}\subset\Hecke_{\bP,[e]}$. By the definition of $\Hecke_{\bP,[e]}$, we have a morphism
\begin{equation*}
\Heckep_{[e]}\to\Hecke_{\bP,[e]}
\end{equation*}
as self-correspondences of $\calM_{\bP}$ over $\tcA_{\bP}$. Then we define $\calH_{|\lambda|_{\bP}}$ to be the reduced image of $\calH_{\lambda}$. Clearly, this image only depends on the $W_{\bP}$-orbit of $\lambda\in\xcoch(T)$. The two projections from $\calH^{\rs}_{|\lambda|_{\bP}}$ to $\calM^{\rs}_{\bP}$ are finite \'etale, hence $\calH_{|\lambda|_{\bP}}$ is graph-like.
\end{cons} 

\begin{prop}\label{p:parenhance}
There is a unique algebra homomorphism
\begin{equation}\label{eq:alpha}
\Ql[\xcoch(T)]^{W_{\bP}}\to\End_{\tcA_{\bP}}(\tilf_{\bP,*}\Ql),
\end{equation}
such that $\Av_{W_{\bP}}(\lambda):=\sum_{\lambda'\in|\lambda|_{\bP}}\lambda'$ acts by $[\calH_{|\lambda|_{\bP}}]_\#$ for any $\lambda\in\xcoch(T)$.
\end{prop}
\begin{proof}
The uniqueness is clear because $\{\Av_{W_{\bP}}(\lambda)|\lambda\in\xcoch(T)\}$ span $\cent$. 

By definition, we have an obvious associative convolution structure $\Hecke_{\bP,[e]}*\Hecke_{\bP,[e]}\to\Hecke_{\bP,[e]}$ given by forgetting the middle $\calM_{\bP}$. By the discussion in \cite[App. A.6]{GSI}, this gives algebra structures on $\Corr(\Hecke_{\bP,[e]};\Ql,\Ql)$ and $\Corr(\Hecke^{\rs}_{\bP,[e]};\Ql,\Ql)$. The same argument as \cite[Lem. 4.4.4]{GSI} shows that any finite type substack of $\Hecke_{\bP,[e]}$ satisfies the condition (G-2) with respect to $\tcArs_{\bP}\subset\tcA_{\bP}$. Therefore by \cite[Prop. A.6.2]{GSI}, it suffices to establish an algebra homomorphism $\Ql[\xcoch(T)]^{W_{\bP}}\to\Corr(\calH^{\Hit,\rs};\Ql,\Ql)$ sending $\Av_{W_{\bP}}(\lambda)$ to $[\calH^{\rs}_{|\lambda|_{\bP}}]$.

We have a commutative diagram of correspondences
\begin{equation}\label{d:heckecl}
\xymatrix{\Heckep_{[e]}\ar@<-1ex>@/_/[d]_{\overleftarrow{h}_{[e]}}\ar@<1ex>@/^/[d]^{\overrightarrow{h}_{[e]}}\ar[r]^{q_{\calH}} & \Hecke_{\bP,[e]}\ar@<-1ex>@/_/[d]_{\overleftarrow{h_{\bP,[e]}}}\ar@<1ex>@/^/[d]^{\overrightarrow{h_{\bP,[e]}}}\\
\Mpar\ar[d]^{\tilf}\ar[r] & \calM_{\bP}\ar[d]^{\tilf_{\bP}}\\
\tcA\ar[r]^{q^{\bP}_{\bI}} & \tcA_{\bP}}
\end{equation}
which is a base change diagram over $\tcArs_{\bP}$. Let $\calH^{\rs}_{\bP,[e]}$ is the reduced structure of $\Hecke^{\rs}_{\bP,[e]}$. Then $q^*_{\calH}$ gives an embedding of algebras (it is injective because $q_{\calH}$ is surjective)
\begin{equation}\label{eq:Winv}
q^*_{\calH}:\Corr(\calH^{\rs}_{\bP,[e]};\Ql,\Ql)\cong\cohog{0}{\calH^{\rs}_{\bP,[e]}}\hookrightarrow \cohog{0}{\calH^{\rs}_{[e]}}\cong\Corr(\calH^{\rs}_{[e]};\Ql,\Ql)\cong\Ql[\xcoch(T)].
\end{equation}
Here we used the fact that $\overleftarrow{h_{[e]}}^{\rs}$ and $\overleftarrow{h_{\bP,[e]}}^{\rs}$ are \'etale, so that we can identify their dualizing complexes with the constant sheaf.

By Construction \ref{cons:Heckelambda}, we have scheme-theoretically that
\begin{equation*}
q^{-1}_{\calH}(\calH^{\rs}_{|\lambda|_{\bP}})=\bigsqcup_{\lambda'\in|\lambda|_{\bP}}\calH^{\rs}_{\lambda'}.
\end{equation*}
Therefore $q^*_{\calH}[\calH^{\rs}_{|\lambda|_{\bP}}]$ is the element $\Av_{W_{\bP}}(\lambda)\in\Ql[\xcoch(T)]$ under the embedding (\ref{eq:Winv}). Since the elements $\{\Av_{W_{\bP}}(\lambda)|\lambda\in\xcoch(T)\}$ span the subalgebra $\Ql[\xcoch(T)]^{W_{\bP}}$ of $\Ql[\xcoch(T)]^{W_{\bP}}$, we conclude from (\ref{eq:Winv}) that the elements $[\calH^{\rs}_{|\lambda|_{\bP}}]$ also span a subalgebra of $\Corr(\calH^{\rs}_{\bP,[e]};\Ql,\Ql)$ isomorphic to $\Ql[\xcoch(T)]^{W_{\bP}}$. This completes the proof.
\end{proof}

\begin{remark}
The action of $\Ql[\dquot{\bP}{\bP}]$ on $f_{\bP,*}\Ql$ constructed in Th. \ref{th:paraction} and the action of $\Ql[\xcoch(T)]^{W_{\bP}}$ on $\tilf_{\bP,*}\Ql$ (hence on $f_{\bP,*}\Ql$) constructed in Prop. \ref{p:parenhance} are related by the embedding of algebras
\begin{eqnarray*}
\Ql[\xcoch(T)]^{W_{\bP}}&\hookrightarrow&\Ql[\dquot{\bP}{\bP}]\\
\Av_{W_{\bP}}(\lambda)&\mapsto&\one_{W_{\bP}\lambda W_{\bP}}.
\end{eqnarray*}
\end{remark}

\begin{remark}\label{rm:Hitaction}
In the special case $\bP=\bG$, Th. \ref{th:paraction} and Prop. \ref{p:parenhance} both give the same action of $\Ql[\xcoch(T)]^W$ on the complex $\fHQl\boxtimes\Ql$ on $\calA\times X$. This can be viewed as a realization of {\em 'tHooft operators} in the algebraic setting.
\end{remark}

\subsection{Parahoric version of the DAHA action}\label{ss:pardaction}

We also have a version of Th. \ref{th:daction} for general parahoric Hitchin fibrations. 

\begin{cons}\label{cons:charP} Fix a standard parahoric subgroup $\bP\subset G(F)$.  Let $\HH_{\bP}$ be the subalgebra of $\HH$ generated by $\Ql[\dquot{\bP}{\bP}]\subset\Ql[\tilW]$, $\Sym_{\Ql}(\xch(\tilT)_{\Ql})^{W_{\bP}}\subset\Sym_{\Ql}(\xcoch(\tilT)_{\Ql})$ and $\Ql[u]$. We now define the actions of generators of this algebra on $f_{\bP,*}\Ql$.
\begin{itemize}
\item In Th. \ref{th:paraction}(2), we have constructed an action of $\Ql[\dquot{\bP}{\bP}]$ on $f_{\bP,*}\Ql$. 
\item $u$ still acts by cup product with the pull-back of $c_1(\calO_X(D))$.
\item Let $\bP^u\subset\bP$ be the pro-unipotent radical and let $\calG^u_{\bP}=G_{\bP^u}\rtimes\Aut^u_{\calO}\subset G_{\bP}\rtimes\Aut_{\calO}\subset\calG$. Let $\hatG_{\bP}$ be the preimage of $G_{\bP}$ in $\hatgloop$. Consider the map $\hatBun_\infty/\calG^u_{\bP}\to\hatBun_{\infty}/\hatG_{\bP}\rtimes\Aut_{\calO}=\Bun_{\bP}$, which is a right torsor under the reductive group $\tilL_{\bP}:=\hatG_{\bP}\rtimes\Aut_{\calO}/\calG^u_{\bP}$. The group $\tilL_{\bP}$ is isomorphic to $\gcen\times L_{\bP}\times\grot$. The characteristic classes of this $\tilL_{\bP}$-torsor are given by
\begin{equation*}
\cohog{*}{\BB\tilL_{\bP}}\cong\cohog{*}{\BB\tilT}^{W_{\bP}}\cong\Sym(\xch(\tilT)_{\Ql}[-2](-1))^{W_{\bP}}.
\end{equation*}
These characteristic classes give a graded action of $\Sym(\xch(\tilT)_{\Ql})^{W_{\bP}}$ on $f_{\bP,*}\Ql$ by cup product.
\end{itemize}
\end{cons}

\begin{theorem}\label{th:pardaction}
There is a unique graded algebra homomorphism:
\begin{equation*}
\HH_{\bP}\to\bigoplus_{i\in\ZZ}\End^{2i}_{\calA\times X}(f_{\bP,*}\Ql)(i)
\end{equation*}
such that $\Ql[\dquot{\bP}{\bP}]$, $u$ and $\Sym(\xch(\tilT)_{\Ql})^{W_{\bP}}$ acts as in Construction \ref{cons:charP}.
\end{theorem}
This can be proved by a similar argument as Th. \ref{th:daction}. We omit the proof here.

\begin{remark}
Let $\one_{W_{\bP}}\in\Ql[\tilW]$ be the characteristic function of the subset $W_{\bP}\subset\tilW$, then $\frac{1}{\#W_{\bP}}\one_{W_{\bP}}$ is an idempotent in $\HH$. It is not hard to check that
\begin{equation*}
\HH_{\bP}=\one_{W_{\bP}}\HH\one_{W_{\bP}}.
\end{equation*}
Therefore, $\HH_{\bP}$ naturally acts on $f_{\bP,*}\Ql=(\fQl)^{W_{\bP}}$ (see Lem. \ref{l:WPinv}). This gives another proof of Th. \ref{th:pardaction}.
\end{remark}

% with cap

\section{Relation with the Picard stack action}\label{s:cap}
In this section, we study the interaction between the graded DAHA action on $\fQl$ and the cap product action on it by the homology complex of the Picard stack $\calP$. In Sec. \ref{ss:appcap}, we study the commutation relation between the cap product action and the graded DAHA action. In Sec. \ref{ss:compcent}, we relate the central part of the $\Ql[\tilW]$-action on the parabolic Hitchin complex to the action by the component groups of the Picard stack. For background on the cap product, we refer the readers to App. \ref{s:capapp}.

% Cap and correspondence

\subsection{The cap product and the double affine action}\label{ss:appcap}

We apply the general discussions in Sec. \ref{ss:cap} to the situation of the $\calP$-action on $\Mpar$ over $\tcA$ or over $\calA\times X$. Define the morphisms $p$ and $\tilp$ as:
\begin{equation}\label{eq:tilp}
\xymatrix{\tcA\ar[r]_{q}\ar@/^1pc/[rr]^{\tilp} & \calA\times X\ar[r]_{p} & \calA}
\end{equation}
Then we have the cap product actions
\begin{eqnarray}\label{eq:capP}
\cap&:&p^*\homo{*}{\calP/\calA}\otimes\fQl\to\fQl,\\
\label{eq:captilP}
\cap&:&\tilp^*\homo{*}{\calP/\calA}\otimes\tfQl\to\tfQl.
\end{eqnarray}

In this section, we study the relationship between these cap product actions and the graded double affine Hecke algebra action constructed in Th. \ref{th:daction}.

\begin{prop}\label{p:capW}
The cap product action of $p^*\homo{*}{\calP/\calA}$ on $\fQl$ commutes with the $\tilW$-action defined in \cite[Th. 4.4.3]{GSI}; the action of $\tilp^*\homo{*}{\calP/\calA}$ on $\tfQl$ commutes with the $\tilW$-equivariant structure defined in \cite[Prop. 4.4.6]{GSI}.
\end{prop}
\begin{proof}
We give the proof of the first statement; the proof of the second one is similar. In Lem. \ref{l:capcorr} of the Appendix, we give a sufficient condition for a cohomological correspondence to commute with the cap product. We want to apply this lemma to our situation.

The Picard stack $\calP$ acts on $\Heckep$ by twisting the two parabolic Hitchin data by the same $J_a$-torsor. This makes $\overleftarrow{h}$ and $\overrightarrow{h}$ both $\calP$-equivariant. We want to apply Lem. \ref{l:capcorr} to the correspondences $\calH_{\tilw}$ which are of finite type. To this end we have to check two things
\begin{enumerate}
\item The correspondence $\calH_{\tilw}$ is stable under the action of $\calP$;
\item The fundamental class $[\calH_{\tilw}]$, viewed as an element of $\Corr(\calH_{\tilw};\Ql,\Ql)$, is $\calP$-invariant (cf. Def. \ref{def:Pinv}).
\end{enumerate}
We show (1). We first show that $\calH^{\rs}_{\tilw}$ is stable under the $\calP$-action. Since $\calH^{\rs}_{\tilw}$ is the graph of the right $\tilw$-action on $\Mparrs$, it suffices to show that the right $\tilW$-action is $\calP$-equivariant. But this follows immediately from the explicit formula of the right $\tilW$-action defined in \cite[Cor. 4.3.8]{GSI}.

We then show that $\calH_{\tilw}$ is stable under $\calP$. Since $\calP$ is smooth over $\calA\times X$, $\calP\times_{\calA\times X}\calH_{\tilw}$ is smooth, and in particular flat over $\calH_{\tilw}$. Since $\calH^{\rs}_{\tilw}$ is dense in $\calH_{\tilw}$, we conclude that $\calP\times_{\calA\times X}\calH^{\rs}_{\tilw}$ is dense in $\calP\times_{\calA\times X}\calH_{\tilw}$. Consider the action map $\act:\calP\times_{\calA\times X}\calH_{\tilw}\to\Heckep$. We already know from above that $\act(\calP\times_{\calA\times X}\calH^{\rs}_{\tilw})$ scheme-theoretically lands in $\calH^{\rs}_{\tilw}$, therefore by the density property we just observed, $\act(\calP\times_{\calA\times X}\calH_{\tilw})$ also lands scheme-theoretically in $\calH_{\tilw}$. This proves (1).

(2) follows from (1): both $\act^![\calH_{\tilw}]$ and $\proj^![\calH_{\tilw}]$ are the fundamental class $[\calP\times_{\calA\times X}\calH_{\tilw}]$.
\end{proof}

Next we study the relation between the cap product by $\homo{*}{\calP/\calA}$ and the cup product
by the Chern classes of $\calL(\xi)$. 

\subsubsection{Rewriting the Chern classes} Suppose $\xi\in\xch(T)$. Recall from \cite[Construction 3.2.8]{GSI} that we have a tautological $T$-torsor $\calQ^T$ over $\tcP$. Let $\calQ(\xi)$ be the line bundle associated to $\calQ^T$ and the character $\xi:T\to\GG_m$. The Chern class of $\calQ(\xi)$ gives a map
\begin{equation}\label{eq:chernstar}
c_1(\calQ(\xi)):\const{\tcA}\to\coho{*}{\tcP/\tcA}[2](1)
\end{equation}
Since $\tcA$ and $\calA$ are both smooth, we have
\begin{equation}\label{eq:repulltcA}
\const{\tcA}\cong\tilp^!\const{\calA}[-2](-1)
\end{equation}
Let $\beta:\tcP\to\calP$ be the projection. Since both $\tcP$ and $\calP$ are smooth, we have $\beta^!\const{\calP}\cong\const{\tcP}[2](1)$. By proper base change, we have $\widetilde{g}_*\Ql\cong\widetilde{g}_*\beta^!\Ql[-2](-1)\cong\tilp^!g_*\Ql[-2](-1)$, where $\widetilde{g}:\tcP\to\tcA$ and $g:\calP\to\calA$ are the structure morphisms. Therefore
\begin{equation}\label{eq:repulltcP}
\coho{*}{\tcP/\tcA}\cong\tilp^!\coho{*}{\calP/\calA}[-2](-1).
\end{equation}
Using (\ref{eq:repulltcA}) and (\ref{eq:repulltcP}), we can rewrite (\ref{eq:chernstar}) as
\begin{equation*}
c_1(\calQ(\xi)):\tilp^!\const{\calA}\to\tilp^!\coho{*}{\calP/\calA}[2](1).
\end{equation*}
By adjunction, this gives a map
\begin{equation*}
c_1(\calQ(\xi))^{\natural}:\homo{*}{\tcA/\calA}\to\coho{*}{\calP/\calA}[2](1).
\end{equation*}

\begin{lemma}\label{l:clQnat}
Under the natural decomposition (\ref{eq:decompcohop}), the map $c_1(\calQ(\xi))^\natural$ factors through
\begin{equation}\label{eq:c1Qnat}
c_1(\calQ(\xi))^\natural:\homo{*}{\tcA/\calA}\to\coho{1}{\calP/\calA}_{\st}[1](1)\subset\coho{*}{\calP/\calA}[2](1).
\end{equation}
\end{lemma}
\begin{proof}
By construction, the line bundle $\calQ(\xi)$ on $\tcA\times_{\calA}\calP$ is the pull-back of the Poincar\'e line bundle on $\tcA\times_{\calA}\stPic(\tcA/\calA)$ using the morphism
\begin{equation}\label{eq:jxi}
\calP\xrightarrow{\jmath_{\calP}}\stPic_T(\tcA/\calA)\xrightarrow{I_{\xi}}\stPic(\tcA/\calA).
\end{equation}
where $I_\xi$ sends a $T$-torsor to the induced line bundle associated to the character $\xi$. By Lem. \ref{l:Picconn} below, the component group of $\stPic(X_a)$ is torsion-free for any $a\in\calA$. Since $\pi_0(\calP_a)$ is finite for $a\in\calA$, the morphism (\ref{eq:jxi}) necessarily factors through the neutral component $\stPic(\tcA/\calA)^0$ of $\stPic(\tcA/\calA)$. Therefore $c_1(\calQ(\xi))^\natural$ factors through
\begin{equation*}
\homo{*}{\tcA/\calA}\to\coho{*}{\stPic(\tcA/\calA)^0/\calA}[2](1)\to\coho{*}{\calP/\calA}[2](1).
\end{equation*}
Since $\stPic(\tcA/\calA)^0\to\calA$ has connected fibers, the above map has to land in the stable part of $\coho{*}{\calP/\calA}[2](1)$. 

By the definition of the tautological line bundle $\calQ(\xi)$, for each integer $N\in\ZZ$, we have
\begin{equation*}
(\id_{\tcA}\times[N])^*\calQ(\xi)\cong\calQ(N\xi)\cong{\calQ(\xi)}^{\otimes N}.
\end{equation*}
Therefore we have a commutative diagram 
\begin{equation*}
\xymatrix{\homo{*}{\tcA/\calA}\ar[rr]^{c_1(\calQ(\xi))}\ar@/_1.5pc/[rrr]_{Nc_1(\calQ(\xi))} && \coho{*}{\calP/\calA}[2](1)\ar[r]^{[N]^*} & \coho{*}{\calP/\calA}[2](1)}
\end{equation*}
This implies that $c_1(\calQ(\xi))$ factors through the eigen-subcomplex of $\coho{*}{\calP/\calA}_{\st}[2](1)$ with eigenvalue $N$ under the endomorphism $[N]^*$, i.e., $\coho{1}{\calP/\calA}_{\st}[1](1)$ (cf. Rem. \ref{rm:cohoP}).
\end{proof}

\begin{lemma}\label{l:Picconn}
For any reduced projective curve $C$ over $k$, the component group $\pi_0(\stPic(C))$ is a free abelian group of finite rank.
\end{lemma}
\begin{proof}
Let $\pi:\tilC\to C$ be the normalization, then we have an exact sequence of groups over $C$:
\begin{equation*}
1\to\GG_m \to\pi_*\GG_m\to\bigoplus_{c\in C^{\sing}}\calR_c\to1
\end{equation*}
where each $\calR_c$ is a {\em connected} commutative algebraic group, viewed as an \'etale sheaf supported at the singular point $c\in C^{\sing}$. This gives an exact sequence of Picard stacks
\begin{equation*}
\prod_{c\in C^{\sing}}\calR_c\to\stPic(C)\xrightarrow{\pi^*}\stPic(\tilC)\to1
\end{equation*}
Since $\prod_{c\in C^{\sing}}\calR_c$ is connected, we have $\pi_0(\stPic(C))\isom\pi_0(\stPic(\tilC))$. Since $\tilC$ is a disjoint union of smooth connected projective curves, $\pi_0(\stPic(\tilC))\isom\ZZ^{\Irr(C)}$ is a free abelian group of finite rank.
\end{proof}

Dually, we can also write the map $c_1(\calQ(\xi))^\natural$ in(\ref{eq:c1Qnat}) as
\begin{equation}\label{eq:dualc1Q}
c_\xi:=\DD(c_1(\calQ(\xi))^\natural):\homo{1}{\calP/\calA}_{\st}\to\coho{*}{\tcA/\calA}[1](1).
\end{equation}

\begin{prop}\label{p:hxi}
\begin{enumerate}
\item []
\item The cap product action of $p^*\homo{*}{\calP/\calA}$ on $\fQl$ commutes with the actions of $u$ and $\delta$. 
\item Suppose $h$ is a section of $\homo{1}{\calP/\calA}$ over an \'etale chart $U\to\calA$, which acts on $\fQl|_U$ via cap product $\cap$. Then we have the following commutation relation between the $h$-action and the $\xi$-action on $\fQl|_{U}$:
\begin{equation}\label{eq:hxicomm}
[\xi,h]:=\xi(h\cap)-(h\cap)\xi=c_\xi(h_{\st})
\end{equation}
where $h_{\st}$ is the stable part of $h$ and $c_\xi(h_{\st})\in\cohog{1}{\tcA_U/U}(1)$ acts on $\fQl$ by cup product.
\end{enumerate}
\end{prop}
\begin{proof}
(1) For $\xi=\delta$ or $u$, $\calL(\xi)$ is the pull-back of a line bundle on $X$. Since the action of $\calP$ preserves the base $\calA\times X$, the cap product by $p^*\homo{*}{\calP/\calA}$ commutes with $\cup c_1(\calL(\xi))$ on $\fQl$.

(2) By the commutative diagram in \cite[Lem. 3.2.9]{GSI}, we have an isomorphism of line bundles on $\tcP\times_{\tcA}\Mpar$:
\begin{equation*}
\act^*\calL(\xi)\cong\calQ(\xi)\boxtimes_{\tcA}\calL(\xi).
\end{equation*}

Unfolding the definition of the cap product, we get a commutative diagram
\begin{equation*}
\xymatrix{(g\times\fpar)_!(\Ql\boxtimes\Ql)\ar[rrr]^{c_1(\calQ(\xi))\otimes\id+\id\otimes c_1(\calL(\xi))}\ar[d]^{\wr} &&& (g\times\fpar)_!(\Ql\boxtimes\Ql)[2](1)\ar[d]^{\wr}\\
\fpar_!\act_!\proj^!\Ql &&& \fpar_!\act_!\proj^!\Ql[2](1)\\
\fpar_!\act_!\act^!\Ql\ar[u]^{\phi}_{\wr}\ar[d]^{\adj}\ar[rrr]^{\fpar_!\act_!c_1(\act^*\calL(\xi))}&&& \fpar_!\act_!\act^!\Ql[2](1)\ar[u]^{\phi[2](1)}_{\wr}\ar[d]^{\adj}\\
\fpar_!\Ql\ar[rrr]^{\fpar_{!}c_1(\calL(\xi))}&&&\fpar_!\Ql[2](1)}
\end{equation*}
In other words, for a local section $h$ of $\homo{1}{\calP/\calA}$ and a local section $\gamma$ of $\fQl$, we have
\begin{equation*}
(h\cup c_1(\calQ(\xi)))\cap\gamma+h\cap(\gamma\cup c_1(\calL(\xi)))=(h\cap\gamma)\cup c_1(\calL(\xi)).
\end{equation*}
This, together with (\ref{eq:dualc1Q}) implies (\ref{eq:hxicomm}).
\end{proof}

\begin{cor}\label{c:stalkcomm} For $(a,x)\in\calA(k)\times X(k)$, the cup-product action of $\homog{*}{\calP_a}$ on $\cohog{*}{\Mpar_{a,x}}$ commutes with the action of the subalgebra $\Ql[\tilW]\otimes\Sym(\xch(T)_{\Ql})\subset\HH/(\delta,u)$.
\end{cor}
\begin{proof} By Prop. \ref{p:capW} and Prop. \ref{p:hxi}(1), it only remains to show that the cup product commutes with $\xi\in\xch(T)$. Let $h\in\homog{1}{\calP_a}$. Restricting to a point $(a,x)$, the cohomology class $c_\xi(h_{\st})\in\cohog{1}{X_a}(1)\in\cohog{1}{q_a^{-1}(x)}(1)=0$ must be 0. Therefore $h$ also commutes with $\xi\in\xch(T)$ by (\ref{eq:hxicomm}).
\end{proof}

\begin{ques}
The commutation relation between the cap product by $\homo{*}{\calP/\calA}$ and the cup product by $c_1(\Lcan)$ on $\fQl$ remains unclear to the author.
\end{ques}

% comparison for Hitchin

\subsection{Comparison of the $\cent$-action and the $\pi_0(\calP/\calA)$-action}\label{ss:compcent}
The cap product (\ref{eq:capP}) in particular gives an action of $p^*\homo{*}{\calP/\calA}=p^*\Ql[\pi_0(\calP/\calA)]$ on the complex $\fQl$. On the other hand, the center $\cent$ of the group algebra $\Ql[\tilW]$ also acts on the complex $\fQl$ by \cite[Th. 4.4.3]{GSI}. The idea of relating the $\pi_0(\calP/\calA)$-action to the $\cent$-action is suggested to the author by B-C.Ng\^{o}.

\subsubsection{Comparison for the Hitchin complex}
Recall from Prop. \ref{p:parenhance} that we have a $\cent$-action on $\fHQl\boxtimes\Ql\in D^b_c(\calA\times X)$, which can be written as (since $p$ is smooth)
\begin{equation*}
\cent\otimes p^!\fHQl\to p^!\fHQl.
\end{equation*}
Applying the adjunction $(p_!,p^!)$ and the projection formula, we get
\begin{equation*}
\alpha:\cent\otimes\homog{*}{X}\otimes\fHQl\to\fHQl.
\end{equation*}
Decomposing $\alpha$ according to $\homog{*}{X}=\oplus_{i=0}^2\homog{i}{X}[i]$, we get
\begin{equation*}
\alpha_i:\cent\otimes\homog{i}{X}\otimes\fHQl\to\fHQl[-i].
\end{equation*}
The goal of this subsection is to describe the effects of $\alpha_i$ in terms of the cap product of $\homo{*}{\calP/\calA}$.

% \begin{remark}
% The fact that $\alpha$ is an action enables us write $\alpha_2$ in terms of $\alpha_1$. In fact, if we switch from homology to cohomology of $X$, then for any $\zeta\in\cent$, $\alpha_2(\zeta)$ can be written as a composition:
% \begin{eqnarray*}
% \fHQl\xrightarrow{\alpha_1(\zeta)}\cohog{1}{X}\otimes\fHQl[-1]&\xrightarrow{\id\otimes\alpha_1(\zeta)[-1]}&\cohog{1}{X}\otimes\cohog{1}{X}\otimes\fHQl[-2]\\
% &\xrightarrow{\cup\otimes\id}& \cohog{2}{X}\otimes\fHQl[-2].
% \end{eqnarray*}
% Therefore it is enough to understand the effects of $\alpha_0$ and $\alpha_1$.
% \end{remark}

\begin{prop}\label{p:comph}
There is a natural map
\begin{equation*}
\sigma:\cent\otimes\homo{*}{(\calA\times X)^{\rs}/\calA}\to\homo{*}{\calP/\calA} 
\end{equation*}
such that the following diagram is commutative
\begin{equation}\label{d:effcomp}
\xymatrix{\cent\otimes\homo{*}{(\calA\times X)^{\rs}/\calA}\otimes\fHQl\ar[r]^(.7){\sigma\otimes\id}\ar[d]^{\id\otimes j_!\otimes\id} & \homo{*}{\calP/\calA}\otimes\fHQl\ar[d]^{\cap}\\
\cent\otimes\homog{*}{X}\otimes\fHQl\ar[r]^(.7){\alpha} & \fHQl}
\end{equation}
where $j:(\calA\times X)^{\rs}\hookrightarrow\calA\times X$ is the open inclusion. 
\end{prop}
\begin{proof}
Recall from \cite[Rem. 4.3.7]{GSI} that we have a morphism:
\begin{equation*}
s:\xcoch(T)\times\tcA^0\to\Grass_J\to\calP.
\end{equation*}
This gives a push-forward map on homology
\begin{equation*}
s_!:\Ql[\xcoch(T)]\otimes\homo{*}{\tcArs/\calA}\to\homo{*}{\calP/\calA}
\end{equation*}
which is $W$-invariant ($W$ acts diagonally on the two factors on the LHS and acts trivially on the RHS). Therefore, it factors through the $W$-coinvariants of $\Ql[\xcoch(T)]\otimes\homo{*}{\tcArs/\calA}$. In particular, if we restrict to $\cent$, the map $s_!$ factors through a map
\begin{equation}\label{eq:th2}
\cent\otimes(\homo{*}{\tcArs/\calA})_W\to\homo{*}{\calP/\calA}
\end{equation}
Since $q^{\rs}:\tcArs\to(\calA\times X)^{\rs}$ is an \'etale $W$-cover, we have $(\homo{*}{\tcArs/\calA})_W=\homo{*}{(\calA\times X)^{\rs}/\calA}$. Therefore the map (\ref{eq:th2}) gives the desired map $\sigma$. The diagram (\ref{d:effcomp}) is commutative because the $\xcoch(T)$-action on $\tilp^{\rs,!}\fHQl$ comes from the following morphism
\begin{equation*}
\xcoch(T)\times\tcArs\times_{\calA}\MHit\xrightarrow{s}\calP\times_{\calA}\MHit\xrightarrow{\act}\MHit.
\qedhere
\end{equation*}
\end{proof}

Passing to the level of (co)homology sheaves, we get
\begin{cor}\label{c:comph}
The map $\sigma$ induces maps $\sigma_i$ ($i=0,1,2$) on homology sheaves:
\begin{equation*}
\sigma_i:\cent\otimes\homo{i}{(\calA\times X)^{\rs}/\calA}\to\homo{i}{\calP/\calA}
\end{equation*}
such that the following diagram is commutative for each $m\in\ZZ,i=0,1,2$:
\begin{equation}\label{d:effcoh}
\xymatrix{\cent\otimes \homo{i}{(\calA\times X)^{\rs}/\calA}\otimes\bR^m\fHQl\ar[r]^(.65){\sigma_i\otimes\id}\ar[d]^{\id\otimes j_!\otimes\id} & \homo{i}{\calP/\calA}\otimes\bR^m\fHQl\ar[d]^{\cap_i^n}\\
\cent\otimes\homog{i}{X}\otimes\bR^m\fHQl\ar[r]^(.7){\alpha_i^m} & \bR^{m-i}\fHQl}
\end{equation}
\end{cor}

Since $\homo{0}{(\calA\times X)^{\rs}/\calA}=\const{\calA}$, $\sigma_0$ gives a homomorphism of sheaves of algebras
\begin{equation}\label{eq:theta}
\sigma_0:\cent\to\Ql[\pi_0(\calP/\calA)].
\end{equation}

\begin{cor}\label{c:Hitfactor}
The action of $\cent$ on $\bR^m\fHQl\boxtimes\const{X}$ constructed in Prop. \ref{p:parenhance} factors through the $\Ql[\pi_0(\calP/\calA)]$-action on $\bR^m\fHQl$ via the map $\sigma_0$ in (\ref{eq:theta}).
\end{cor}

% comparison for parabolic Hitctin

Our final goal in this subsection is to prove:

\begin{theorem}\label{th:comp}
The action of $\cent$ on $\bR^m\fQl$ constructed in \cite[Th. 4.4.3]{GSI} factors through the $p^*\Ql[\pi_0(\calP/\calA)]$ action on $\bR^m\fQl$ via the map $p^*\sigma_0:\cent\to p^*\Ql[\pi_0(\calP/\calA)]$.
\end{theorem}

\begin{remark}
On the other hand, we will see in \cite[Sec. 5.2]{GSIII} that the action of the whole lattice $\xcoch(T)$ on $\bR^n\fQl$ does not factor through a finite quotient: the action can be unipotent.
\end{remark}

\subsubsection{Hecke modifications at two points} To prove the Th. \ref{th:comp}, we consider a more general Hecke correspondence which combines the two situations we considered in \cite[Sec. 4.1]{GSI} and Sec. \ref{ss:parenhance}:
\begin{equation}\label{d:twoptmod}
\xymatrix{ & \Hecke' \ar[dl]_{\overleftarrow{h'}}\ar[dr]^{\overrightarrow{h'}} & \\
\Mpar\times X\ar[dr]_{\fpar\times \id_X} & & \Mpar\times X\ar[dl]^{\fpar\times \id_X} \\
& \calA\times X^2 &}
\end{equation} 
For any scheme $S$, $\Hecke'(S)$ is the groupoid of tuples $(x,y,\calE_1,\varphi_1,\calE^B_{1,x},\calE_2,\varphi_2,\calE^B_{2,x},\alpha)$ where
\begin{itemize}
\item $(x,\calE_i,\varphi_i,\calE^B_{i,x})\in\Mpar(S)$;
\item $y\in X(S)$ with graph $\Gamma(y)$;
\item $\alpha$ is an isomorphism of Hitchin pairs $(\calE_1,\varphi_1)|_{S\times X-\Gamma(y)})\isom(\calE_2,\varphi_2)|_{S\times X-\Gamma(y)}$
\end{itemize}

For a point $(a,x,y)\in(\calA\times X^2)(k)$ such that $x\neq y$, the fibers of $\overleftarrow{h'}$ and $\overrightarrow{h'}$ over $(a,x,y)$ are isomorphic to the product of $M^{\Hit}_y(\gamma_{a,y})$ and a Springer fiber in $G/B$ corresponding to $\gamma_{a,x}$ (see the discussion in \cite[Sec. 3.3]{GSI}); while if we restrict to the diagonal $\Delta_X:\calA\times X\subset\calA\times X^2$, $\Hecke'|_{\Delta_X}$ is the same as $\Heckep$. The reader may notice the analogy between our situation and the situation considered by Gaitsgory in \cite{Ga}, where he uses Hecke modifications at two points to deform the product $\Grass_G\times G/B$ to $\Flag_G$.

As in the case of $\Heckep$, we have a morphism
\begin{equation*}
\Hecke'\to\Mpar\times_{\calA\times X}\Mpar\to\tcA\times_{\calA\times X}\tcA.
\end{equation*}
Let $\Hecke'_{[e]}$ be the preimage of the diagonal $\tcA\subset\tcA\times_{\calA\times X}\tcA$. We have a commutative diagram of correspondences
\begin{equation}\label{d:twoptbch}
\xymatrix{\Hecke'_{[e]}\ar@<-1ex>@/_/[d]_{\overleftarrow{h'_{[e]}}}\ar@<1ex>@/^/[d]^{\overrightarrow{h'_{[e]}}}\ar[r]^{q'} & \Hk{\bG}{\bG}\ar@<-1ex>@/_/[d]_{\overleftarrow{h_{\bG}}}\ar@<1ex>@/^/[d]^{\overrightarrow{h_{\bG}}}\\
\Mpar\times X\ar[d]^{\fpar}\ar[r] & \calM_{\bG}=\MHit\times X\ar[d]^{f^{\Hit}\times\id_X}\\
\tcA\times X\ar[r]^{q\times\id_X} & \calA\times X}
\end{equation}

By \cite[Lem. 3.5.4]{GSI}, this is a base change diagram if we restrict the base spaces to $\tcA^0\times X\to\calA\times X$. Recall from Construction \ref{cons:Heckelambda} that for each $W$-orbit $|\lambda|$ in $\xcoch(T)$, we have a graph-like closed substack $\calH_{|\lambda|}\subset\Hk{\bG}{\bG}$. Let $\calH'_{|\lambda|}\subset\Hecke'_{[e]}$ be closure of the preimage of $\calH^{\rs}_{|\lambda|}$ under $q'$. By the same argument as Prop. \ref{p:parenhance}, we can prove:

\begin{lemma}\label{l:twoptaction}
There is a unique action $\alpha'$ of $\cent$ on the complex $(f^{\parab}\times\id_X)_*\Ql=\fQl\boxtimes\const{X}$ on $\calA\times X^2$ such that $\Av_{W}(\lambda)$ acts as $[\calH'_{|\lambda|}]_\#$.
\end{lemma}

On the other hand, $p'^*\Ql[\pi_0(\calP/\calA)]$ acts on $\fQl\boxtimes\const{X}$ via its action on the first factor, here $p':\calA\times X^2\to\calA$ is the projection. Similar to the proof of Prop. \ref{p:comph} and Cor. \ref{c:Hitfactor}, we have

\begin{lemma}\label{l:twofactor}
The action $\alpha'$ of $\cent$ on $\coho{m}{\fQl\boxtimes\const{X}}=\bR^m\fQl\boxtimes\const{X}$ constructed in Lem. \ref{l:twoptaction} factors through the $p'^*\Ql[\pi_0(\calP/\calA)]$ action on $\fQl\boxtimes\const{X}$ via the homomorphism $p'^*\sigma_0:\cent\to p'^*\Ql[\pi_0(\calP/\calA)]$.
\end{lemma}

Now we are ready to prove the theorem.

\begin{proof}[Proof of Th. \ref{th:comp}.]
We denote by $\alpha$ the action of $\cent$ on $\fQl$ given by restricting the $\Ql[\tilW]$-action. We will define another action of $\cent$ on $\fQl$.
 
Restricting the correspondence diagram (\ref{d:twoptmod}) to the diagonal $\Delta_X:\calA\times X\hookrightarrow\calA\times X^2$, we recover the correspondence $\Heckep$. Restricting the commutative diagram (\ref{d:twoptbch}) to the diagonal, we recover the commutative diagram (\ref{d:heckecl}). The $\Delta_X$-restriction of the $\cent$-action $\alpha'$ on $\fQl\boxtimes\const{X}$ constructed in Lem. \ref{l:twoptaction} gives an action of $\cent$ on $\fQl=\Delta_X^*(\fQl\boxtimes\const{X})$. We denote this action by $\alpha'_\Delta$. 

We claim that the actions $\alpha$ and $\alpha'_\Delta$ are the same. In fact, by \cite[Lem. A.2.1]{GSI}, the action of $\alpha'_\Delta(\Av_W(\lambda))$ is given by the class $\Delta_X^*[\calH'_{|\lambda|}]\in\Corr(\Heckep;\Ql,\Ql)$. On the other hand, the action of $\alpha(\Av_W(\lambda))$ is given by the class $\sum_{\lambda'\in|\lambda|}[\calH_{\lambda'}]\in\Corr(\Heckep;\Ql,\Ql)$. When restricted to $(\calA\times X)^{\rs}$ both classes coincide with the fundamental class of $q^{*}_{\calH}[\calH^{\rs}_{|\lambda|}]$ (cf. diagram (\ref{d:heckecl})). Since both classes are supported on a graph-like substack of $\Heckep$ (see \cite[Lem. 4.4.4]{GSI}), their coincidence over $(\calA\times X)^{\rs}$ ensures that their actions on $\fQl$ are the same, by \cite[Lem. A.5.2]{GSI}.

By Lem. \ref{l:twofactor}, the action $\alpha'_\Delta$ of $\cent$ on $\bR^m\fQl$ factors through $p^*\sigma_0:\cent\to\Delta_X^*p'^*\Ql[\pi_0(\calP/\calA)]=p^*\Ql[\pi_0(\calP/\calA)]$. Since this action is the same as $\alpha$, the theorem is proved.
\end{proof}

\appendix

\section{Generalities on the cap product}\label{s:capapp}
In this appendix, we recall the formalism of cap product by the homology sheaf of a commutative smooth group scheme, partially following \cite[7.4]{NgoFL}.

\subsection{The Pontryagin product on homology}\label{ss:Pon}

Let $P$ be a commutative smooth group scheme (or Deligne-Mumford Picard stack such as $\calP$) of finite type over a scheme $S$. Let $g:P\to S$ be the structure map and let $\homo{*}{P/S}$ be the homology complex of $P$ on $S$.

\begin{lemma}\label{l:decomphomop} There is a canonical decomposition in $D^b(S)$:
\begin{equation}\label{eq:decomphomop}
\homo{*}{P/S}\cong\bigoplus_{i\geq0}\homo{i}{P/S}[i].
\end{equation}
\end{lemma}
\begin{proof}
Take any $N\in\ZZ$ which is coprime to the cardinalities of $\pi_0(P_s)$ for all $s\in S$ (such an integer exists because there are only finitely many isomorphism types of $\pi_0(P_s)$). The $N$-th power map $[N]:P\to P$ induces an endomorphism $[N]_*$ on $\homo{*}{P/S}$. Let $\calH_i$ be the direct summand of $\homo{*}{P/S}$ on which the eigenvalues of $[N]_*$ have archimedean norm $N^i$ for any embedding $\Ql\to\CC$. It is easy to see that $\calH_i$ is independent of the choice of $N$. Then $\homo{*}{P/S}$ is the direct sum of $\calH_i$ and each $\calH_i$ is isomorphic to $\homo{i}{P/S}[i]$.
\end{proof}

The multiplication map $\textup{mult}:P\times_SP\to P$ induces a {\em Pontryagin product}
\begin{equation*}
\homo{*}{P/S}\otimes\homo{*}{P/S}\to\homo{*}{P/S}.
\end{equation*}
which, in turn, induces a Pontryagin product on the homology sheaves $\homo{i}{P/S}$. Since the multiplication map is compatible with the $N$-th power map in the obvious sense, the decomposition (\ref{eq:decomphomop}) respects the Pontryagin product on the homology {\em complex} and the Pontryagin product on the homology {\em sheaves}.

We have the following facts about the homology sheaves of $P/S$:
\begin{itemize}
\item $\homo{0}{P/S}\cong\Ql[\pi_0(P/S)]$. Recall from \cite[6.2]{NgoFL} that there is a sheaf of abelian groups $\pi_0(P/S)$ on $S$ for the \'{e}tale topology whose fiber at $s\in S$ is the finite group of connected components of $P_s$. Therefore the group algebra $\Ql[\pi_0(P/S)]$ is a $\Ql$-sheaf of algebras on $S$ whose fiber at $s\in S$ is the $0^{\textup{th}}$ homology of $P_s$. This algebra structure is the same as the one induced from the Pontryagin product.
\item If $P_s$ is {\em connected} for some $s\in S$, the stalk of $\homo{1}{P/S}$ at $s$ is the {\em $\Ql$-Tate module} $V_\ell(P_s)=T_\ell(P_s)\otimes_{\ZZ_\ell}\Ql$ of $P_s$. Moreover, the Pontryagin product induces an isomorphism
\begin{equation}\label{eq:fiberwedge}
\bigwedge^iV_\ell(P_s)=\bigwedge^i\homo{1}{P_s}\cong\homo{i}{P_s}.
\end{equation}
\end{itemize}

\begin{remark}\label{rm:cohoP}
If we work with cohomology rather than homology, the $N$-th power map also gives a natural decomposition
\begin{equation}\label{eq:decompcohop}
\coho{*}{P/S}\cong\bigoplus_i\coho{i}{P/S}[-i].
\end{equation}
This decomposition respects the cup product on the cohomology complex and the cup product on the cohomology sheaves.
\end{remark}

\subsection{The stable parts}\label{ss:Pst}

We have seen from the decomposition (\ref{eq:decomphomop}) and the fact $\homo{0}{P/S}=\Ql[\pi_0(P/S)]$ that $\pi_0(P/S)$ acts on the homology complex $\homo{*}{P/S}$ and the cohomology complex $\coho{*}{P/S}$.

\begin{defn}\label{def:stP}
The {\em stable part} of $\homo{i}{P/S}$ (resp. $\coho{i}{P/S}$) is the direct summand on which the action of $\pi_0(P/S)$ is trivial. We denote the stable parts by $\homo{i}{P/S}_{\st}$ and $\coho{i}{P/S}_{\st}$. Let $\homo{*}{P/S}_{\st}=\oplus_i\homo{i}{P/S}_{\st}[i]$ and $\coho{*}{P/S}_{\st}=\bigoplus\coho{i}{P/S}_{\st}[-i]$ be the corresponding decompositions of $\homo{*}{P/S}_{\st}$ and $\coho{*}{P/S}_{\st}$.
\end{defn}

\begin{remark}
To make sense of the invariants of a sheaf under the action of another sheaf of finite abelian groups, we refer to \cite[Prop. 8.3]{NgoFib}.
\end{remark}

It is clear that the stable part $\homo{*}{P/S}_{\st}$ (resp. $\coho{*}{P/S}_{\st}$) inherits a Pontryagin product (resp. a cup product) from that of $\homo{*}{P/S}$ (resp. $\coho{*}{P/S}$).

Let $P^0\subset P$ be the Deligne-Mumford substack over $S$ of fiberwise neutral components of $P/S$ (which exists as an open substack of $P$, cf. \cite[Prop. 6.1]{NgoFib}). Let $V_\ell(P^0/S)$ be the sheaf of $\Ql$-Tate modules of $P^0$ over $S$.

% stable part of homology is an exterior algebra

\begin{lemma}\label{l:wedge}
\begin{enumerate}
\item []
\item The embedding $P^0\subset P$ and the Pontryagin product gives a natural isomorphism of $\Ql[\pi_0(P/S)]$-algebra objects in $D^b_c(S)$:
\begin{equation}\label{eq:homowedgeP}
\Ql[\pi_0(P/S)]\otimes\homo{*}{P^0/S}\isom\homo{*}{P/S}.
\end{equation}
\item The natural embedding $P^0\subset P$ followed by the projection onto the stable part gives a natural isomorphism of algebra objects in $D^b_c(S)$:
\begin{equation}\label{eq:sthomowedge}
\bigwedge(V_\ell(P^0/S)[1])\cong\homo{*}{P^0/S}\to\homo{*}{P/S}\twoheadrightarrow\homo{*}{P/S}_{\st}.
\end{equation}

\end{enumerate}
\end{lemma}
\begin{proof}
Both maps (\ref{eq:homowedgeP}) and (\ref{eq:sthomowedge}) are direct sums of maps between (shifted) sheaves. To check they are isomorphisms, it suffices to check on the stalks. Fix a geometric point $s\in S$. Since all connected components of $P_s$ are isomorphic to $P^0_s$, we have a $\pi_0(P_s)$-equivariant isomorphism
\begin{equation}\label{eq:Hten}
\homog{*}{P_s}\cong\homog{*}{P^0_s}\otimes\Ql[\pi_0(P_s)]
\end{equation}
on which $\pi_0(P_s)$ acts via the regular representation on $\Ql[\pi_0(P_s)]$. This proves (\ref{eq:homowedgeP}). Using (\ref{eq:Hten}), the natural embedding $P^0_s\subset P_s$ followed by the projection onto the stable part
\begin{equation}\label{eq:Hc}
\homog{*}{P_s^0}\hookrightarrow\homog{*}{P_s}\twoheadrightarrow\homog{*}{P_s}_{\st}
\end{equation}
becomes the tensor product of the identity map on $\homog{*}{P^0_s}$ with the map
\begin{equation}\label{eq:obvious}
\Ql\cdot e\hookrightarrow\Ql[\pi_0(P_s)]\twoheadrightarrow\Ql[\pi_0(P_s)]^{\pi_0(P_s)}
\end{equation}
where $e\in\pi_0(P_s)$ is the identity element. Now the composition of the maps in (\ref{eq:obvious}) is obviously an isomorphism, hence the composition of the maps in (\ref{eq:Hc}) is also an isomorphism. To obtain the first isomorphism in (\ref{eq:sthomowedge}), we only need to apply the isomorphism (\ref{eq:fiberwedge}) to the connected Picard stack $P^0/S$.
\end{proof}

\begin{remark}
If $P/S$ is smooth and proper, then the above lemma easily dualize to a similar statement about the cohomology complex $\coho{*}{P/S}$. In particular, we have an isomorphism of algebra objects in $D^b_c(S)$ (with the cup product on the LHS and the wedge product on the RHS):
\begin{equation}\label{eq:stcohowedge}
\coho{*}{P/S}_{\st}\cong\coho{*}{P^0/S}\cong\bigwedge(V_\ell(P^0/S)^*[-1]).
\end{equation}
\end{remark}

\subsection{The cap product}\label{ss:cap}

Suppose $P$ acts on a Deligne-Mumford stack $M$ over $S$, with the action and projection morphisms
\begin{equation*}
\xymatrix{P\times_{S}M\ar@<.7ex>[r]^{\act}\ar@<-.7ex>[r]_{\proj} & M.}
\end{equation*}
Suppose $\calF$ is a $P$-equivariant complex on $M$, then in particular we are given an isomorphism
\begin{equation*}
\phi:\act^!\calF\isom\proj^!\calF.
\end{equation*}
Therefore we have a map
\begin{equation}\label{eq:precap}
\act_!\proj^!\calF\xrightarrow{\act_!\phi^{-1}}\act_!\act^!\calF\xrightarrow{\adj}\calF.
\end{equation}
Let $f:M\to S$ be the structure map. Using K\"{u}nneth formula ($P$ is smooth over $S$), we get
\begin{equation}\label{eq:kun}
\homo{*}{P/S}\otimes f_!\calF=g_!\DD_{P/S}\otimes f_!\calF\cong(g\times f)_!\proj^!\calF=f_!\act_!\proj^!\calF.
\end{equation}
Applying $f_!$ to the map (\ref{eq:precap}) and combining with the isomorphism (\ref{eq:kun}), we get the {\em cap product}:
\begin{equation}\label{eq:cap}
\cap:\homo{*}{P/S}\otimes f_!\calF\to f_!\calF
\end{equation}
such that $f_!\calF$ becomes a module over the algebra $\homo{*}{P/S}$ under the Pontryagin product. Using the decomposition (\ref{eq:decomphomop}) we get the actions
\begin{eqnarray*}
\cap_i&:&\homo{i}{P/S}\otimes f_!\calF\to f_!\calF[-i];\\
\cap_i^{m}&:&\homo{i}{P/S}\otimes\bR^mf_!\calF\to\bR^{m-i}f_!\calF.
\end{eqnarray*}
When $i=0$, the cap product $\cap_0$ gives an action of $\Ql[\pi_0(P/S)]$ on $f_!\calF$. By the isomorphism (\ref{eq:homowedgeP}), to understand the cap product, we only need to understand $\cap_0$ and $\cap_1$.

\section{Complement on cohomological correspondences}\label{s:complcorr}

This appendix is a complement to \cite[App. A]{GSI}. We continue to use the notations from {\em loc.cit}. In particular, we fix a correspondence diagram
\begin{equation}\label{d:corr}
\xymatrix{ & C\ar[dl]_{\overleftarrow{c}}\ar[dr]^{\overrightarrow{c}} & \\
 X\ar[dr]_{f} & & Y\ar[dl]^{g}\\
& S & }
\end{equation}

% Cup

\subsection{Cup product and correspondences}
In this subsection, we study the interaction between the cup product and cohomological correspondences. For each $i\in\ZZ$, let
\begin{eqnarray*}
\Corr^i(C;\calF,\calG)=\Corr(C;\calF[i],\calG)\\
\Corr^*(C;\calF,\calG)=\oplus_i\Corr^i(C;\calF,\calG).
\end{eqnarray*}
We have a left action of $\cohog{*}{X}$ and a right action of $\cohog{*}{Y}$ on $\Corr^*(C;\calF,\calG)$. More precisely, for $\alpha\in \cohog{j}{X},\beta\in \cohog{j}{Y}$ and $\zeta\in\Corr^i(C;\calF,\calG)$, we define $\alpha\cdot\zeta,\zeta\cdot\beta\in\Corr^{i+j}(C;\calF,\calG)$ as
\begin{eqnarray*}
\alpha\cdot\zeta:\overrightarrow{c}^*\calG\xrightarrow{\zeta}\overleftarrow{c}^!\calF[i]\xrightarrow{\overrightarrow{c}^!(\cup\alpha)}\overleftarrow{c}^!\calF[i+j];\\
\zeta\cdot\beta:\overrightarrow{c}^*\calG\xrightarrow{\overrightarrow{c}^*(\cup\beta)}\overrightarrow{c}^*\calG[j]\xrightarrow{\zeta}\overleftarrow{c}^!\calF[i+j].
\end{eqnarray*}

The following lemma is obvious.

\begin{lemma}\label{l:cup}
For $\alpha\in \cohog{j}{X},\beta\in\cohog{j}{Y}$ and $\zeta\in\Corr^i(C;\calF,\calG)$, we have
\begin{equation*}
(\alpha\cdot\zeta)_{\#}=f_*(\cup\alpha)\circ\zeta_\#;\hspace{1cm}(\zeta\cdot\beta)_\#=\zeta_\#\circ g_!(\cup\beta).
\end{equation*}
\end{lemma}

On the other hand, $\cohog{*}{C}$ acts on $\Corr^*(C;\calF,\calG)=\Ext^*_C(\overrightarrow{c}^*\calG,\overleftarrow{c}^!\calF)$ by cup product, which we denote simply by $\cup$.

\begin{lemma}\label{l:cupcorr}
Let $\alpha\in \cohog{*}{X},\beta\in\cohog{*}{Y}$ and $\zeta\in\Corr^*(C;\calF,\calG)$, then we have
\begin{equation}\label{eq:leftrightcup}
\alpha\cdot\zeta=\zeta\cup(\overleftarrow{c}^*\alpha);\hspace{1cm}\zeta\cdot\beta=\zeta\cup(\overrightarrow{c}^*\beta).
\end{equation}
\end{lemma}
\begin{proof}
The second identity is obvious from definition. We prove the first one. By the projection formula and adjunction, we have a map
\begin{equation*}
\overleftarrow{c}_!(\overleftarrow{c}^!\calF\otimes\overleftarrow{c}^*\calK)\cong(\overleftarrow{c}_!\overleftarrow{c}^!\calF)\otimes\calK\to\calF\otimes\calK
\end{equation*}
bifunctorial in $\calF,\calK\in D^b_c(X,\Ql)$. Applying the adjunction $(\overleftarrow{c}_!,\overleftarrow{c}^!)$ again, we get a bifunctorial map
\begin{equation}\label{eq:cup}
\overleftarrow{c}^!\calF\otimes\overleftarrow{c}^*\calK\to\overleftarrow{c}^!(\calF\otimes\calK).
\end{equation}
Now taking $\calK=\Ql$, and view $\alpha\in\cohog{i}{X}$ as a map $\alpha:\calK\to\calK[i]$. The functoriality of the map (\ref{eq:cup}) in $\calK$ implies a commutative diagram
\begin{equation*}
\xymatrix{\overleftarrow{c}^!\calF\otimes\overleftarrow{c}^*\calK\ar[r]\ar[d]^{\id\otimes\overleftarrow{c}^*\alpha} & \overleftarrow{c}^!(\calF\otimes\calK)\ar[d]^{\overleftarrow{c}^!(\calF\otimes\alpha)}\\
\overleftarrow{c}^!\calF\otimes\overleftarrow{c}^*\calK[i]\ar[r] & \overleftarrow{c}^!(\calF\otimes\calK[i])}
\end{equation*}
which is equivalent to the first identity in (\ref{eq:leftrightcup}).
\end{proof}

If we have a base change diagram of correspondences induced from $S'\to S$ as in \cite[App. A.2]{GSI}, then the pull-back map
\begin{equation*}
\gamma^*:\Corr^*(C;\calF,\calG)\to\Corr^*(C';\phi^*\calF,\psi^*\calG)
\end{equation*}
commutes with the actions of $\cohog{*}{X},\cohog{*}{Y}$ and $\cohog{*}{C}$ in the obvious sense.

% cap

\subsection{Cap product and correspondences}\label{ss:capcorr}
In this subsection, we study the interaction between the cap product (see Sec. \ref{ss:cap}) and cohomological correspondences. Suppose a group scheme $P$ (or a Picard stack which is Deligne-Mumford) over $S$ acts on the correspondence diagram (\ref{d:corr}), i.e., $\overleftarrow{c}$ and $\overrightarrow{c}$ are $P$-equivariant. We use ``$\act$'' to denote the action maps by $P$ and ``$\proj$'' to denote the projections along $P$, and add subscripts to indicate the space on which $P$ is acting, e.g., $\act_C:P\times_SC\to C$. Let $\calF,\calG$ be $P$-equivariant complexes on $X$ and $Y$.

\begin{defn}\label{def:Pinv}
For $\zeta\in\Corr(C;\calF,\calG)$, we say $\zeta$ is {\em $P$-invariant} if the pull-backs $\act_C^!\zeta$ and $\proj_C^!\zeta$ correspond to each other under the isomorphism
\begin{equation*}
\Corr(P\times_SC;\act_X^!\calF,\act_Y^!\calG)\isom\Corr(P\times_SC;\proj_X^!\calF,\proj_Y^!\calG)
\end{equation*}
given by the equivariant structures of $\calF$ and $\calG$.
\end{defn}

\begin{remark} Here we use the $!$-pull-back rather than the $*$-pull-back of cohomological correspondences defined in \cite[App. A.2]{GSI}. Since the action morphisms and the projections are smooth, the $!$- and $*$-pull-backs only differ by a shift and a twist, so that the results in \cite[App. A.2]{GSI} are still applicable in this situation.
\end{remark}

\begin{lemma}\label{l:capcorr}
Suppose $X$ is proper over $S$ so that $f_!\calF=f_*\calF$. Let $\zeta\in\Corr(C;\calF,\calG)$ be $P$-invariant, then the cap product action of $\homo{*}{P/S}$ commutes with $\zeta_\#$, i.e., we have a commutative diagram
\begin{equation}\label{d:capcorr}
\xymatrix{\homo{*}{P/S}\otimes g_!\calG\ar[d]^{\cap}\ar[r]^{\id\otimes\zeta_\#} & \homo{*}{P/S}\otimes f_*\calF\ar[d]^{\cap}\\
g_!\calG\ar[r]^{\zeta_\#} & f_*\calF}
\end{equation}
\end{lemma}
\begin{proof}
Consider the correspondence of $P\times_SC$ between $P\times_SX$ and $P\times_SY$ over $P$ as the base change from the correspondence diagram (\ref{d:corr}) by the action maps $\act$. Let $h:P\to S$ be the structure morphism. By \cite[Lem. A.2.1]{GSI}, we a get commutative diagram (note that the action maps are smooth)
\begin{equation}\label{d:tr}
\xymatrix{g_!\act_{Y,!}\act_Y^!\calG\ar@{=}[r]_{\bch}\ar[dr]_{\adj}\ar@/^2pc/[rrr]^{(\act_C^!\zeta)_\#} & h_!h^!g_!\calG\ar[r]_{h_!h^!(\zeta_\#)}\ar[d]^{\adj} & h_!h^!f_*\calF\ar@{=}[r]_{\bch}\ar[d]^{\adj} & f_*\act_{X,!}\act^!_X\calF\ar[dl]^{\adj}\\
& g_!\calG\ar[r]^{\zeta_\#} & f_*\calF & }
\end{equation}
By assumption, we have $\act_C^!\zeta=\proj_C^!\zeta$. Therefore we can identify the top row of the diagram (\ref{d:tr}) with
\begin{equation}\label{eq:corrp}
h_!\circ(\proj_C^!\zeta)_\#:g_!\proj_{Y,!}\proj_{Y}^!\calG\to f_*\proj_{X,!}\proj^!_X\calF.
\end{equation}
It is easy to identify the map (\ref{eq:corrp}) with $\id\otimes\zeta_\#:\homo{*}{P/S}\otimes g_!\calG\to\homo{*}{P/S}\otimes f_*\calF$. This identifies the outer quadrangle of the diagram (\ref{d:tr}) with the diagram (\ref{d:capcorr}).
\end{proof}

\end{document}